\documentclass[11pt]{article}

\usepackage[cp1252]{inputenc}
\usepackage[english]{babel}
\usepackage[a4paper]{geometry}
\usepackage{amsmath,amsfonts,amssymb,amsthm,cases,upgreek}
\usepackage{empheq}
\usepackage{stmaryrd}
\usepackage{dsfont}
\usepackage{graphicx}
\usepackage{comment} 

\newtheorem{thm}{Theorem}[section]
\newtheorem{definition}[thm]{Definition}
\newtheorem{lemma}[thm]{Lemma}
\newtheorem{cor}[thm]{Corollary}
\newtheorem{prop}[thm]{Proposition}
\newtheorem{remark}[thm]{Remark}

\newcommand{\R}{\mathbb{R}}

\newcommand{\RD}{\mathbb{R}^{d}}
\newcommand{\RDD}{\mathbb{R}^{d+1}}
\newcommand{\Hdot}{\dot{H}}
\newcommand{\Hzeta}{H^{t_{0}+2}(\RD)}
\newcommand{\Hzetaa}{H^{t_{0}+1}(\RD)}
\newcommand{\G}{G_{\mu}[\epsilon \zeta,\beta b]}
\newcommand{\GNN}{G_{\mu}^{N \! N}[\epsilon \zeta, \beta b]}
\newcommand{\GDD}{G_{\mu}^{D \! D}[\epsilon \zeta, \beta b]}
\newcommand{\GND}{G_{\mu}^{N \! D}[\epsilon \zeta, \beta b]}
\newcommand{\E}{\mathcal{E}^{N}}
\newcommand{\w}{\underline{w}}
\newcommand{\V}{\underline{V}}
\newcommand{\ws}{\widetilde{\underline{w}}}
\newcommand{\Vs}{\widetilde{\underline{V}}}
\newcommand{\psia}{\psi_{(\alpha)}}
\newcommand{\zetaa}{\zeta_{(\alpha)}}
\newcommand{\Ua}{U_{(\alpha)}}
\newcommand{\Ra}{R_{\alpha}}
\newcommand{\Sa}{S_{\alpha}}
\newcommand{\rt}{\underline{\mathfrak{a}}}

\newcommand{\psih}{\psi^{\mathfrak{h}}}

\newcommand{\indicatrice}{\mathbf{1}}

\def \epsilon {\varepsilon}

\begin{document}
\title{\textbf{A mathematical study of meteo and landslide tsunamis : The Proudman resonance}}
\author{Benjamin MELINAND\footnote{benjamin.melinand@math.u-bordeaux.fr}}
\date{March 2015}

\maketitle
\vspace{-1cm}

\begin{abstract}
\noindent In this paper, we want to understand the Proudman resonance. It is a resonant respond in shallow waters of a water body on a traveling atmospheric disturbance when the speed of the disturbance is close to the typical water wave velocity. We show here that the same kind of resonance exists for landslide tsunamis and we propose a mathematical approach to investigate these phenomena based on the derivation, justification and analysis of relevant asymptotic models. This approach allows us to investigate more complex phenomena that are not dealt with in the physics literature such as the influence of a variable bottom or the generalization of the Proudman resonance in deeper waters. First, we prove a local well-posedness of the water waves equations with a moving bottom and a non constant pressure at the surface taking into account the dependence of small physical parameters and we show that these equations are a Hamiltonian system (which extend the result of Zakharov \cite{Zakharov}). Then, we justify some linear asymptotic models in order to study the Proudman resonance and submarine landslide tsunamis; we study the linear water waves equations and dispersion estimates allow us to investigate the amplitude of the sea level. To complete these asymptotic models, we add some numerical simulations.
\end{abstract}

\section{Introduction}

\subsection{Presentation of the problem}

\noindent A tsunami is popularly an elevation of the sea level due to an earthquake. However, tsunamis induced by seismic sources represent only 80 \% of the tsunamis. 6\% are due to landslides and 3\% to meteorological effects (see the book of B. Levin and M. Nosov \cite{Levin_tsunamis}). Big traveling storms for instance can give energy to the sea  and lead to an elevation of the surface. In some cases, this amplification is important and this phenomenon is called the Proudman resonance in the physics literature. Similarly, submarine landslides can significantly  increase the level of the sea and we talk about landslide tsunamis. In this paper, we study mathematically these two phenomena. We model the sea by an irrotational and incompressible ideal fluid bounded from below by the seabed and from above by a free surface. We suppose that the seabed and the surface are graphs above the still water level. We model an underwater landslide by a moving seabed (moving bottom) and the meteorological effects by a non constant pressure at the surface (air-pressure disturbance). Therefore, we suppose that $b(t,X) = b_{0}(X) + b_{m}(t,X)$, where $b_{0}$ represents a fixed bottom and $b_{m}$ the variation of the bottom because of the landslide. Similarly, the pressure at the surface is of the form $P + P_{\text{ref}}$, where $P_{\text{ref}}$ is a constant which represents the pressure far from the meteorological disturbance, and $P(t,X)$ models the meteorological disturbance (we assume that the pressure at the surface is known). We denote by $d$ the horizontal dimension, which is equal to $1$ or $2$. $X \in \RD$ stands for the horizontal variable and $z \in \mathbb{R}$ is the vertical variable. $H$ is the typical water depth. The water occupies a moving domain $\Omega_{t} := \{ (X,z) \in \RDD \text{ , } -H + b(t,X) < z < \zeta (t,X) \}$. The water is homogeneous (constant density $\rho$), inviscid, irrotational with no surface tension. We denote by $\textbf{U}$ the velocity and $\Phi$ the velocity potential. We have $\textbf{U}=\nabla_{\! X,z} \Phi$. The law governing the irrotational fluids is the Bernoulli law

\begin{equation}\label{Bernoulli_law}
\partial_{t} \Phi + \frac{1}{2} |\nabla_{\! X,z} \Phi|^{2} + gz =  \frac{1}{\rho} \left(P_{\text{ref}} - \mathcal{P} \right) \text{ in } \Omega_{t},
\end{equation}

\noindent where $\mathcal{P}$ is the pressure in the fluid domain. Changing $\Phi$ if necessary, it is possible to assume that $P_{\text{ref}} = 0$. Furthermore, the incompressibility of the fluid implies that 

\begin{equation}\label{Incompressibility}
\Delta_{\! X,z} \Phi = 0 \text{ in } \Omega_{t}.
\end{equation}

\noindent We suppose also that the fluid particles do not cross the bottom or the surface. We denote by \textbf{n} the unit normal vector, pointing upward and $\partial_{\textbf{n}}$ the upward normal derivative. Then, the boundary conditions are

\begin{equation}\label{surface_condition}
\partial_{t} \zeta - \sqrt{1+|\nabla \zeta|^{2}} \partial_{\textbf{n}} \Phi = 0 \text{ on } \{ z = \zeta(t,X) \},
\end{equation}

\noindent and

\begin{equation}\label{bottom_condition}
\partial_{t} b - \sqrt{1+|\nabla b|^{2}} \partial_{\textbf{n}} \Phi = 0 \text{ on } \{ z = -H + b(t,X) \}.
\end{equation}

\noindent In 1968, V. E. Zakharov (see \cite{Zakharov}) showed that the water waves problem is a Hamiltonian system and that  $\psi$, the trace of the velocity potential at the surface ($\psi = \Phi_{|z=\zeta}$), and the surface $\zeta$ are canonical variables. Then, W. Craig, C. Sulem and P.L. Sulem (see \cite{Craig_Sulem_1} and \cite{Craig_Sulem_2}) formulate this remark into a system of two non local equations. We follow their construction to formulate our problem. Using the fact that $\Phi$ satisfies \eqref{Incompressibility} and \eqref{bottom_condition}, we can characterize $\Phi$ thanks to $\zeta$ and $\psi = \Phi_{|z=\zeta}$

\begin{equation}\label{Laplace_mixte}
\left\{
\begin{aligned}
&\Delta_{\! X,z} \Phi = 0 \text{ in } \Omega_{t},\\
&\Phi_{|z= \zeta} = \psi \text{ , } \sqrt{1+|\nabla b|^{2}}\partial_{\textbf{n}} \Phi_{|z=-H+b} = \partial_{t} b.
\end{aligned}
\right.
\end{equation}

\noindent We decompose this equation in two parts, the surface contribution and the bottom contribution

\begin{equation*}
\Phi = \Phi^{S} + \Phi^{B},
\end{equation*}

\noindent such that 

\begin{equation}\label{Laplace_surface}
\left\{
\begin{aligned}
&\Delta_{\! X,z} \Phi^{S} = 0 \text{ in } \Omega_{t},\\
&\Phi^{S}_{\;\; |z=\zeta} = \psi \text{ , } \sqrt{1+|\nabla b|^{2}}\partial_{\textbf{n}} \Phi^{S}_{\;\; |z=-H+b} = 0,
\end{aligned}
\right.
\end{equation}

\noindent and

\begin{equation}\label{Laplace_bottom}
\left\{
\begin{aligned}
&\Delta_{\! X,z} \Phi^{B} = 0 \text{ in } \Omega_{t},\\
&\Phi^{B}_{\;\; |z= \zeta} = 0 \text{ , } \sqrt{1+|\nabla b|^{2}}\partial_{\textbf{n}} \Phi^{B}_{\;\; |z=-H+b} = \partial_{t} b.
\end{aligned}
\right.
\end{equation}

\noindent In the purpose of expressing \eqref{surface_condition} with $\zeta$ and $\psi$, we introduce two operators. The first one is the Dirichlet-Neumann operator 

\begin{equation}\label{Dirichlet_Neumann_operator}
G[\zeta,b] : \psi \mapsto \sqrt{1+|\nabla \zeta|^{2}} \partial_{\textbf{n}} \Phi^{S}_{\;\; |z=\zeta},
\end{equation}

\noindent where $\Phi^{S}$ satisfies \eqref{Laplace_surface}. The second one is the Neumann-Neumann operator 

\begin{equation}\label{Neumann_Neumann_operator}
G^{N \! N} [\zeta,b] : \partial_{t} b \mapsto \sqrt{1+|\nabla \zeta|^{2}} \partial_{\textbf{n}} \Phi^{B}_{\;\; |z=\zeta},
\end{equation}

\noindent where $\Phi^{B}$ satisfies \eqref{Laplace_bottom}. Then, we can reformulate \eqref{surface_condition} as

\begin{equation}\label{surface_condition2}
\partial_{t} \zeta - G[\zeta,b](\psi) = G^{N \! N} [\zeta,b](\partial_{t} b).
\end{equation}

\noindent Furthermore thanks to the chain rule, we can express $(\hspace{-0.02cm} \partial_{t} \Phi \hspace{-0.02cm} )_{|z=\zeta}$, $\hspace{-0.05cm} (\hspace{-0.02cm} \nabla_{\! X,z} \Phi \hspace{-0.02cm})_{|z=\zeta}$ and $(\hspace{-0.02cm} \partial_{z} \Phi \hspace{-0.02cm})_{|z=\zeta}$ in terms of $\psi$, $\zeta$, $G[\zeta,b](\psi)$ and $G^{N \! N} [\zeta,b](\partial_{t} b)$. Then, we take the trace at the surface of \eqref{Bernoulli_law} (since there is no surface tension we have $ \mathcal{P}_{|z=\zeta} = P$) and we obtain a system of two scalar equations that reduces to the standard Zakharov/Craig-Sulem formulation when $\partial_{t} b = 0$ and $P=0$,

\begin{equation}\label{ww_equations}
  \left\{
  \begin{aligned}
   &\partial_{t} \zeta - G[\zeta,b](\psi) =  G^{N \! N} [\zeta,b](\partial_{t} b),\\
   &\partial_{t} \psi +  g\zeta + \frac{1}{2} |\nabla \psi|^{2} - \frac{1}{2} \frac{\left( G[\zeta,b](\psi) + G^{N \! N} [\zeta,b](\partial_{t} b) + \nabla \zeta \cdot \nabla \psi\right)^{2}}{(1+|\nabla \zeta|^{2})}= -\frac{P}{\rho} \text{.}
  \end{aligned}
  \right.
\end{equation}

\noindent In the following, we work with a nondimensionalized version of the water waves equations with small parameters $\epsilon$, $\beta$ and $\mu$ (see section \ref{model}). The wellposedness of the water waves problem with a constant pressure and a fixed bottom was studied by many people. S. Wu proved it in the case of an infinite depth without nondimensionalization (\cite{Wu_2D} and \cite{Wu_3D}). Then, D. Lannes treated the case of a finite bottom without nondimensionalization (\cite{Lannes_wellposedness_ww}), T. Iguchi proved a local wellposedness result for $\mu$ small enough in order to justify shallow water approximations for water waves (\cite{Iguchi_shallow_water}), and D. Lannes and B. Alvarez-Samaniego showed, in the case of the nondimensionalized equations, that we can find an existence time $T = \frac{T_{0}}{\max(\epsilon,\beta)}$ where $T_{0}$ does not depend on $\epsilon$, $\beta$ and $\mu$ (\cite{Alvarez_Lannes}). More recently, B. M\'esognon-Gireau improved the result of  D. Lannes and B. Alvarez-Samaniego and proved that if we add enough surface tension we can find an existence time $T = \frac{T_{0}}{\epsilon}$ where $T_{0}$ does not depend on $\epsilon$ and $\mu$ (\cite{Benoit_large_time_existence}). T. Iguchi studied the case of a moving bottom in order to justify asymptotic models for tsunamis (\cite{Iguchi_tsunami}). Finally, T. Alazard, N. Burq and C. Zuily study  the optimal regularity for the initial data (\cite{Alazard_cauchy_problem}) and more recently, T. Alazard, P. Baldi and D. Han-Kwan show that a well-chosen non constant external pressure can create any small amplitude two-dimensional  gravity-capillary water waves (\cite{Alazard_control_water_waves}). We organize this paper in two part. Firstly in Section \ref{Local existence of the water waves equations}, we prove two local existence theorems  for the water waves problem with a moving bottom and a non constant pressure at the surface by differentiating and "quasilinearizing" the water waves equations and we pay attention to the dependence of the time of existence and the size of the solution with respect to the parameters $\epsilon$, $\beta$, $\lambda$ and $\mu$. This theorem extends the result of T. Iguchi (\cite{Iguchi_tsunami}) and D. Lannes (Chapter 4 in \cite{Lannes_ww}). We also prove that the water waves problem can be viewed as a Hamiltonian system. Secondly in Section \ref{Asymptotic models}, we justify some linear asymptotic models and study the Proudman resonance. First, in Section \ref{A shallow water model with small topography variations} we study the case of small topography variations in shallow waters, approximation used in the Physics literature to investigate the Proudman resonance; then in Section \ref{A shallow water model with a variable topography} we derive a model when the topography is not small in the shallow water approximation; and in Section \ref{Linear asymptotic and resonance in intermediate depths} we study the linear water waves equations in order to extend the Proudman resonance in deep water with a small fixed topography. Finally, Appendix \ref{Laplace_problem} contains results about the elliptic problem \eqref{Laplace_bott} and Appendix \ref{The Dirichlet-Neumann and the Neumann-Neumann operators} contains results about the Dirichlet-Neumann and the Neumann-Neumann operators. Appendix \ref{estimates} comprises standard estimates that we use in this paper.

\subsection{Notations}

\noindent A good framework for the velocity in the Euler equations is the Sobolev spaces $H^{s}$. But we do not work with $\textbf{U}$ but with $\psi$ the trace of $\Phi$, and $\textbf{U} = \nabla_{\! X,z} \Phi$. It will be too restrictive to take $\psi$ in a Sobolev space. A good idea is to work with the Beppo Levi spaces (see \cite{Deny_Lions_beppo_levi}). For $s \geq 0$, the Beppo Levi spaces are 

\begin{equation*}
  \Hdot^{s}(\RD) := \left\lbrace \psi \in  L^{2}_{\rm loc}(\RD) \text{, } \nabla \psi \in H^{s-1}(\RD) \right\rbrace.
\end{equation*}

\noindent In this paper, $C$ is a constant and for a function $f$ in a normed space $(X,\lvert \cdot \rvert)$ or a parameter $\gamma$, $C(|f|,\gamma)$ is a constant depending on $|f|$ and $\gamma$ whose exact value has non importance. The norm $| \cdot |_{L^{2}}$ is the $L^{2}$-norm and $| \cdot |_{\infty}$ is the $L^{\infty}$-norm in $\RD$. Let $f \in \mathcal{C}^{0}(\RD)$ and $m \in \mathbb{N}$ such that $ \frac{f}{1+|x|^{m}}\in L^{\infty}(\RD)$. We define the Fourier multiplier $f(D) : H^{m}(\RD) \shortrightarrow L^{2}(\RD)$ as

\begin{equation*}
\forall u \in H^{m}(\RD) \text{, } \widehat{f(D) u}(\xi) = f(\xi) \widehat{u}(\xi).
\end{equation*}

\noindent In $\RD$ we denote the gradient operator by $\nabla$ and in $\Omega$ or $S =  \RD \times (-1,0)$ the gradient operator is denoted $\nabla_{\! X,z}$. Finally, we denote by $\Lambda := \sqrt{1 + |D|^{2}}$ with $D = -i \nabla$.

\section{Local existence of the water waves equations}\label{Local existence of the water waves equations}

\noindent This part is devoted to the wellposedness of the water waves equations (Theorems \ref{existence_uniqueness} and \ref{long_time_existence}). We carefully study the dependence on the parameters $\epsilon$, $\beta$, $\lambda$ and $\mu$ of the existence time and of the size of the solution. Contrary to \cite{Lannes_ww} and \cite{Iguchi_tsunami}, we exhibit the nonlinearities of the water waves equations in order to obtain a better existence time. 

\subsection{The model}\label{model}

\noindent In this part, we present a nondimensionalized version of the water waves equations. In order to derive some asymptotic models to the water waves equations we introduce some dimensionless parameters linked to the physical scales of the system. The first one is the ratio between the typical free surface amplitude $a$ and the water depth $H$. We define $\epsilon := \frac{a}{H}$, called the nonlinearity parameter. The second one is the ratio between $H$ and the characteristic horizontal scale $L$. We define $\mu := \frac{H^{2}}{L^{2}}$, called the shallowness parameter. The third one is the ratio between the order of bottom bathymetry amplitude $a_{\text{bott}}$ and $H$. We define $\beta := \frac{a_{\text{bott}}}{H}$, called the bathymetric parameter. Finally, we denote by $\lambda$ the ratio of the typical landslide amplitude $a_{\text{bott,m}}$ and $a_{\text{bott}}$. We also nondimensionalize the variables and the unknowns. We introduce

\begin{equation}
 \left\{ 
 \begin{aligned}
 &X' = \frac{X}{L} \text{, } z' = \frac{z}{H} \text{, } \zeta' = \frac{\zeta}{a} \text{, } b' = \frac{b}{a_{\text{bott}}} \text{, } b_{0}' = \frac{b_{0}}{a_{\text{bott}}} \text{, } b_{m}' = \frac{b_{m}}{a_{\text{bott,m}}} \text{, } t' = \frac{\sqrt{gH}}{L} t \text{, }\\
 &\left(\Phi^{S}\right)^{'}=\frac{H}{aL \sqrt{gH}} \Phi^{S} \text{, } \left(\Phi^{B} \right)^{'}=\frac{L}{H a_{\text{bott,m}} \sqrt{gH}} \Phi^{B} \text{, } \psi'=\frac{H}{aL \sqrt{gH}} \psi \text{, } P' = \frac{P}{a \rho g},
 \end{aligned}
 \right.
\end{equation}

\noindent where 

\begin{equation*}
\Omega_{t}^{\prime} = \{ (X',z') \in \RDD \text{ , } -1 + \beta b'(t',X') < z' < \epsilon \zeta' (t',X') \}.
\end{equation*}

\begin{remark}

\noindent It is worth noting that the nondimensionalization of $\Phi^{S}$, $\psi$ and $t$ comes from the linear wave theory (in shallow water regime, the characteristic speed is $\sqrt{gH}$). See paragraph \textbf{1.3.2} in \cite{Lannes_ww}. Let us explain the nondimensionalization of $\Phi^{B}$. Consider the linear case

\begin{equation*}
\left\{
\begin{aligned}
&\Delta_{\! X,z} \Phi^{B} = 0 \text{, } -H < z < 0,\\
&\Phi^{B}_{\;\; |z= 0} = 0 \text{ , } \partial_{z} \Phi^{B}_{\;\; |z=-H} = \partial_{t} b.
\end{aligned}
\right.
\end{equation*}

\noindent A straightforward computation gives $\Phi^{B} = \frac{\sinh(z |D|)}{|D| \cosh(H |D|)} \partial_{t} b$. If the typical wavelength is $L$, the typical wave number is $\frac{2 \pi}{L}$. Furthermore, the typical order of magnitude of $\partial_{t} b$ is $\frac{a_{\text{bott,m}} \sqrt{gH}}{L}$. Then, the order of magnitude of $\Phi^{B}$ in the shallow water case is 

\begin{equation*}
\frac{L}{2 \pi} \frac{\sqrt{gH} a_{\text{bott,m}}}{L} \frac{\sinh(2 \pi \frac{H}{L})}{\cosh(2 \pi \frac{H}{L})} \sim \frac{\sqrt{gH} a_{\text{bott,m}} H}{L}.
\end{equation*}

\end{remark}

\noindent For the sake of clarity, we omit the primes. We can now nondimensionalize the water waves problem. Using the notation

\begin{equation*}
\nabla^{\mu}_{\! X,z} := (\sqrt{\mu} \nabla_{\! X} \; , \partial_{z} \; )^{t} \text{ and } \Delta^{\mu}_{\! X,z} \; := \mu \Delta_{\! X} \; + \partial_{z}^{2},
\end{equation*}

\noindent the water waves equations \eqref{ww_equations} become in dimensionless form

\begin{equation}\label{water_waves_equations}
  \left\{
  \begin{aligned}
   &\partial_{t} \zeta \hspace{-0.05cm} - \hspace{-0.05cm} \frac{1}{\mu} \G(\psi) \hspace{-0.05cm} = \hspace{-0.05cm} \frac{\beta \lambda}{ \epsilon} \GNN(\partial_{t} b), \\
   &\partial_{t} \psi \hspace{-0.05cm} + \hspace{-0.05cm} \zeta \hspace{-0.05cm} + \hspace{-0.05cm} \frac{\epsilon}{2} |\nabla \psi|^{2} \hspace{-0.05cm} - \hspace{-0.05cm} \frac{\epsilon}{2 \mu} \frac{\left( \G(\psi) \hspace{-0.05cm} + \hspace{-0.05cm} \frac{\lambda \beta \mu}{\epsilon}\GNN(\partial_{t} b) \hspace{-0.05cm} + \hspace{-0.05cm} \mu \nabla(\epsilon \zeta) \cdot \nabla \psi\right)^{2}}{(1+\epsilon^{2} \mu |\nabla \zeta|^{2})} \hspace{-0.05cm}=  \hspace{-0.05cm} - P \text{.}
  \end{aligned}
  \right.
\end{equation}

\noindent In the following $\partial_{\textbf{n}}$ is the upward \textit{conormal} derivative

\begin{equation*}
\partial_{\textbf{n}} \Phi^{S} = \textbf{n} \cdot \begin{pmatrix} \sqrt{\mu} I_{d} \hspace{0.1cm} 0 \\ \hspace{0.5cm} 0 \hspace{0.3cm} 1 \end{pmatrix} \nabla^{\mu}_{\! X,z} \Phi^{S}_{\;\, | \partial \Omega}.
\end{equation*}

\noindent Then, The Dirichlet-Neumann operator $\G$ is

\begin{equation}\label{DN_operator}
  \G(\psi) := \sqrt{1+\epsilon^{2} |\nabla \zeta|^{2}} \partial_{\textbf{n}} \Phi^{S}_{\;\, |z=\epsilon \zeta} = -\mu \epsilon \nabla \zeta \cdot \nabla_{\! X} \Phi^{S}_{\;\, |z=\epsilon \zeta} + \partial_{z} \Phi^{S}_{\;\, |z=\epsilon \zeta} \text{ ,}
\end{equation} 

\noindent where $\Phi^{S}$ satisfies 

\begin{equation}\label{Laplace_surf}
  \left\{
  \begin{aligned}
   & \Delta^{\mu}_{\! X,z} \Phi^{S} = 0 \text{ in } \Omega_{t} \text{ ,}\\
   &\Phi^{S}_{\;\, |z=\epsilon \zeta} = \psi \text{ , } \partial_{\textbf{n}} \Phi^{S}_{\;\, |z=-1+\beta b} = 0,
  \end{aligned}
  \right.
\end{equation} 

\noindent while the Neumann-Neumann operator $\GNN$ is

\begin{equation}\label{NN_operator}
  \GNN(\partial_{t} b) := \sqrt{1+\epsilon^{2} |\nabla \zeta|^{2}} \partial_{\textbf{n}} \Phi^{B}_{\;\;\, |z=\epsilon \zeta} = -\mu \nabla(\epsilon \zeta) \cdot \nabla_{\! X} \Phi^{B}_{\;\;\, |z=\epsilon \zeta} + \partial_{z} \Phi^{B}_{\;\;\, |z=\epsilon \zeta} \text{ ,}
\end{equation} 

\noindent where $\Phi^{B}$ satisfies 

\begin{equation}\label{Laplace_bott}
  \left\{
  \begin{aligned}
   & \Delta^{\mu}_{\! X,z} \Phi^{B} = 0 \text{ in } \Omega_{t} \text{ ,} \\
   &\Phi^{B}_{\;\;\, |z=\epsilon \zeta} = 0 \text{ , } \sqrt{1+\beta^{2} |\nabla b|^{2}} \partial_{\textbf{n}} \Phi^{B}_{\;\;\, |z=-1+\beta b} = \partial_{t} b.
  \end{aligned}
  \right.
\end{equation} 

\begin{remark}
\noindent We have nondimensionalized the Dirichlet-Neumann and the Neumann-Neumann operators as follows 

\begin{equation*}
G[\zeta,b](\psi) = \frac{a L \sqrt{gH}}{H^{2}} G_{\mu}[\epsilon \zeta^{'}, \beta b^{'}] (\psi^{'} ) \text{, } G^{N \! N}[\zeta,b](\partial_{t} b) = \frac{a_{\text{bott,m}} \sqrt{gH}}{L} G_{\mu}^{N \! N}[\epsilon \zeta^{'}, \beta b^{'}] (\partial_{t^{'}} b^{'}) .
\end{equation*}
\end{remark}

\noindent We add two classical assumptions. First, we assume some constraints on the nondimensionalized parameters and we suppose there exist $\rho_{\max} > 0$ and $\mu_{\max} > 0$, such that

\begin{equation}\label{parameters_constraints}
0 < \epsilon, \beta, \beta \lambda \leq 1 \text{ , } \frac{\beta \lambda}{\epsilon} \leq \rho_{\max} \text{ and } \mu \leq \mu_{\max}.
\end{equation}

\noindent Furthermore, we assume that the water depth is bounded from below by a positive constant

\begin{equation}\label{nonvanishing}
  \exists \, h_{\min} > 0 \text{ ,  } \epsilon \zeta + 1 - \beta b \geq h_{\min}.
\end{equation} 

\noindent In order to quasilinearize the water waves equations, we have to introduce the vertical speed at the surface $\w$ and horizontal speed at the surface $\V$. We define

\begin{equation}
 \w := \w[\epsilon \zeta, \beta b] \left( \psi, \frac{\beta \lambda}{\epsilon}\partial_{t} b \right) = \frac{\G(\psi) + \mu \frac{\beta \lambda}{\epsilon}  \GNN(\partial_{t} b) + \epsilon \mu \nabla \zeta \cdot \nabla \psi}{1+\epsilon^{2} \mu |\nabla \zeta|^{2}},     
\end{equation}

\hspace{-0.8cm} and

\begin{equation}
\V := \V[\epsilon \zeta, \beta b]  \left( \psi, \frac{\beta \lambda}{\epsilon}\partial_{t} b \right) = \nabla \psi - \epsilon \w[\epsilon \zeta, \beta b] \left( \psi, \frac{\beta \lambda}{\epsilon}\partial_{t} b \right) \nabla \zeta.     
\end{equation}

\subsection{Notations and statement of the main results}

\noindent In this paper, $d=1 \text{ or } 2$, $t_{0}>\frac{d}{2}$, $N \in \mathbb{N}$ and $s \geq 0$. The constant $T \geq 0$ represents a final time. The pressure $P$ and the bottom $b$ are given functions. We suppose that $b \in W^{3,\infty}(\R^{+}; H^{N}(\RD))$ and $P \in W^{1,\infty}(\R^{+}; \Hdot^{N+1}(\RD))$. We denote by $M_{N}$ a constant of the form

\begin{equation}
M_{N} = C \left( \frac{1}{h_{\min}},\mu_{\max}, \epsilon |\zeta|_{H^{\max(t_{0} + 2,N)}}, \beta |b|_{L^{\infty}_{t} H^{\max(t_{0}+2,N)}_{\! X}} \right).
\end{equation}

\noindent We denote by $U := (\zeta,\psi)^{t}$ the unknowns of our problem. We want to express \eqref{ww_equations} as a quasilinear system. It is well-known that the good energy for the water waves problem is

\begin{equation}\label{energy_water_waves}
\E(U) = |\mathfrak{P} \psi |^{2}_{H^{\frac{3}{2}}} + \underset{\alpha \in \mathbb{N}^{d}, |\alpha| \leq N}{\sum} \left( |\zetaa|^{2}_{L^{2}} + |\mathfrak{P} \psia|^{2}_{L^{2}} \right),
\end{equation}

\noindent where $\zetaa := \partial^{\alpha} \zeta$, $\psia := \partial^{\alpha} \psi - \epsilon \underline{w} \partial^{\alpha} \zeta$ and $\mathfrak{P} := \frac{|D|}{\sqrt{1+\sqrt{\mu} |D|}}$. This energy is motivated by the linearization of the system around the rest state (see \textbf{4.1} in \cite{Lannes_ww}). $\mathfrak{P}$ acts as the square root of the Dirichlet-Neumann operator (see \cite{Lannes_ww}). Here, $\zetaa$ and $\psia$ are the \textit{Alinhac's good unknowns of the system} (see \cite{Alinhac_good_unknown} and \cite{Alinhac_good_unknown_ww} in the case of the standard water waves problem). We define $\Ua := (\zetaa,\psia)^{t}$. We can introduce an associated energy space. Considering a $T \geq 0$,

\begin{equation}
E^{N}_{T} := \{ U  \in \mathcal{C}([0,T];H^{t_{0}+2}(\RD) \times \Hdot^{2}(\RD)) \text{ , } \E(U) \in L^{\infty}([0,T]) \}.
\end{equation}

\noindent Our main results are the following theorems. We give two existence results. The first theorem extends the result of T.Iguchi (Theorem 2.4 in \cite{Iguchi_tsunami}) since we give a control of the dependence of the solution with respect to the parameters $\epsilon$, $\beta$ and $\mu$ and we add a non constant pressure at the surface and also extends the result of D.Lannes (Theorem 4.16 in \cite{Lannes_ww}), since we improve the regularity of the initial data and add a non constant pressure pressure at the surface and a moving bottom. Notice that we explain later what is Condition \eqref{rt_constraints} (it corresponds to the positivity of the so called Rayleigh-Taylor coefficient).

\begin{thm}\label{existence_uniqueness}
Let $A > 0$, $t_{0} > \frac{d}{2}$, $N \geq \max(1,t_{0}) +3$, $U^{0} \in E_{0}^{N}$, $b \in W^{3,\infty}(\R^{+}; H^{N}(\RD))$ and $P \in W^{1,\infty}(\mathbb{R}^{+}; \Hdot^{N+1}(\RD))$ such that

\begin{equation*}
\E \left( U^{0} \right) + \frac{\beta \lambda}{\epsilon} \left\lvert \partial_{t} b \right\rvert_{L^{\infty}_{t} H_{\! X}^{N}} + \left\lvert \nabla P \right\rvert_{L^{\infty}_{t} H_{\! X}^{N}} \leq A.
\end{equation*}

\noindent We suppose that the parameters $\epsilon,\beta,\mu, \lambda$ satisfy \eqref{parameters_constraints} and that \eqref{nonvanishing} and \eqref{rt_constraints} are satisfied initially. Then, there exists $T > 0$ and a unique solution $U \in E^{N}_{T}$ to \eqref{water_waves_equations} with initial data $U^{0}$. Moreover, we have 

\begin{equation*}
T = \min \left( \frac{T_{0}}{\max(\epsilon, \beta)}, \frac{T_{0}}{\frac{\beta \lambda}{\epsilon} \left\lvert \partial_{t} b \right\rvert_{L^{\infty}_{t} H_{\! X}^{N}} + \left\lvert \nabla P \right\rvert_{L^{\infty}_{t} H_{\! X}^{N}}} \right) \text{ , }  \frac{1}{T_{0}} =c^{1} \text{ and } \underset{t \in \left[0, T \right]}{\sup} \E \hspace{-0.1cm} (U) = c^{2},
\end{equation*}

\noindent with $c^{j} = C \left(A, \frac{1}{h_{\min}}, \frac{1}{\mathfrak{a}_{\min}}, \mu_{\max}, \rho_{\max}, \left\lvert b \right\rvert_{W^{3,\infty}_{t} H_{\! X}^{N}}, \left\lvert \nabla P \right\rvert_{W^{1,\infty}_{t} H_{\! X}^{N}} \right) $.
\end{thm}

\noindent Notice that if $\partial_{t} b$ and $P$ are of size $\max(\epsilon, \beta)$, we find the same existence time that in Theorem 4.16 in \cite{Lannes_ww}. The second result shows that it is possible to go beyond the time scale of the previous theorem; although the norm of the solution is not uniformly bounded in terms of $\epsilon$ and $\beta$, we are able to make this dependence precise. This theorem will be used to justify some of the asymptotic models derived in Section \ref{Asymptotic models} over large time scales when the pressure at the surface and the moving bottom are not supposed small. We introduce $\delta := \max(\epsilon,\beta^{2})$.

\begin{thm}\label{long_time_existence}
\noindent Under the assumptions of the previous theorem, there exists $T_{0} > 0$ such that $U \in E^{N}_{\frac{T_{0}}{\sqrt{\delta}}}$. Moreover, for all $\alpha \in \left[0, \frac{1}{2} \right]$, we have

\begin{equation*}
\small{
\frac{1}{T_{0}} =c^{1} , \hspace{-0.3cm} \underset{\tiny{t \in \left[0, \frac{T_{0}}{\delta^{\alpha}} \right]}}{\sup} \hspace{-0.15cm} \E \hspace{-0.1cm} (U) \leq \frac{c^{3}}{\delta^{2\alpha}} \text{, } c^{j} = C \hspace{-0.1cm} \left(\hspace{-0.1cm} A, \frac{1}{h_{\min}}, \frac{1}{\mathfrak{a}_{\min}}, \mu_{\max}, \rho_{\max}, \left\lvert b \right\rvert_{W^{3,\infty}_{t} H_{\! X}^{N}}, \left\lvert \nabla P \right\rvert_{W^{1,\infty}_{t} H_{\! X}^{N}} \hspace{-0.1cm} \right) \hspace{-0.15cm}.
}
\end{equation*}
\end{thm}

\noindent Notice that when $\partial_{t} b$ and $P$ are of size $\max(\epsilon, \beta)$, the existence time of Theorem \ref{existence_uniqueness} is better than the one of Theorem \ref{long_time_existence}. Theorem \ref{long_time_existence} is only useful when $\partial_{t} b$ and $P$ are not small. Notice finally, that Condition \eqref{rt_constraints} is satisfied if $\epsilon$ is small enough. Hence, since in the following, $\epsilon$ is small, it is reasonable to assume it.

\subsection{Quasilinearization}

\noindent Firstly, we give some controls of $|\mathfrak{P} \psi|_{H^{s}}$ and $|\mathfrak{P} \psia|_{H^{s}}$ with respect to the energy $\E(U)$.

\begin{prop}\label{psi_controls}
Let $T > 0$, $t_{0} > \frac{d}{2}$ and $N \geq 2 + \max(1,t_{0})$. Consider $U \in E^{N}_{T}$, $b \in W^{1,\infty}(\R^{+};H^{N}(\RD))$, such that $\zeta$ and $b$ satisfy Condition \eqref{nonvanishing} for all $0 \leq t \leq T$. We assume also that $\mu$ satisfies \eqref{parameters_constraints}. Then, for $0 \leq t \leq T$, for $ \alpha \in \mathbb{N}^{d}$ with $|\alpha| \leq N-1$ and for $0 \leq s \leq N - \frac{1}{2}$, 

\begin{equation*}
|\partial^{\alpha} \mathfrak{P} \psi |_{L^{2}} + |\mathfrak{P} \psi_{(\alpha)}|_{H^{1}} + |\mathfrak{P} \psi |_{H^{s}} \leq M_{N} \E(U)^{\frac{1}{2}} + \frac{\beta \lambda}{\epsilon} M_{N} \left\lvert \partial_{t} b \right\rvert_{L^{\infty}_{t} H^{N}_{\! X}}.
\end{equation*}
\end{prop}

\begin{proof}
\noindent For the first inequality, we have thanks to Proposition \ref{P_estimates},

\begin{align*}
|\partial^{\alpha} \mathfrak{P} \psi|_{L^{2}} &\leq |\mathfrak{P} \psia|_{L^{2}} + \epsilon |\mathfrak{P} (\w \partial^{\alpha} \zeta)|_{L^{2}},\\
&\leq |\mathfrak{P} \psia|_{L^{2}} + \frac{\epsilon}{\mu^{\frac{1}{4}}} |\w \partial^{\alpha} \zeta|_{H^{\frac{1}{2}}}.\\
\end{align*}

\noindent But $\psi \in \Hdot^{2}(\RD)$. Then by Proposition \ref{controls_w}, $\w \in H^{1}(\RD)$ and $\partial^{\alpha} \zeta \in H^{1}(\RD)$. Using Proposition \ref{product_estimate1}, we obtain

\begin{equation*}
|\partial^{\alpha} \mathfrak{P} \psi |_{L^{2}} \leq |\mathfrak{P} \psia|_{L^{2}} + C \epsilon \left\lvert \frac{\w}{\mu^{\frac{1}{4}}} \right\rvert_{H^{1}} \hspace{-0.3cm} |\zeta|_{H^{N}} \leq |\mathfrak{P} \psia|_{L^{2}} + M_{N} \epsilon |\zeta|_{H^{N}} \left( \hspace{-0.07cm} |\mathfrak{P} \psi|_{H^{\frac{3}{2}}} + \frac{\beta \lambda}{\epsilon} |\partial_{t} b|_{H^{1}} \hspace{-0.07cm} \right) \hspace{-0.07cm}.
\end{equation*}

\noindent The other inequalities follow with the same arguments, see for instance Lemma 4.6 in \cite{Lannes_ww}.
\end{proof}

\noindent The following statement is a first step to the quasilinearization of the water waves equations. It is essentially Proposition 4.5 in \cite{Lannes_ww} and Lemma 6.2 in \cite{Iguchi_tsunami}. However, we improve the minimal regularity of $U$ (we decrease the minimal value of $N$ to $4$ in dimension 1) and we provide the dependence in $\partial_{t} b$ which does not given in \cite{Iguchi_tsunami}. For those reasons, we give a proof of this Proposition.

\begin{prop}\label{linearization_1}
Let $t_{0} > \frac{d}{2}$, $T>0$, $N \geq \max(t_{0},1) + 3$, $b \in W^{1,\infty}(\R^{+};H^{N}(\RD))$ and $U \in E^{N}_{T}$, such that $\zeta$ and $b$ satisfy Condition \eqref{nonvanishing} for all $0 \leq t \leq T$. We assume also that $\mu$ satisfies \eqref{parameters_constraints}. Then, for all $\alpha \in \mathbb{N}^{d}$, $1 \leq |\alpha| \leq N$, we have,

\begin{equation*}
\begin{aligned}
\partial^{\alpha} \hspace{-0.07cm} \left( \frac{1}{\mu} \G(\psi) \hspace{-0.05cm} + \hspace{-0.05cm} \frac{\lambda \beta}{\epsilon} \GNN(\partial_{t} b) \hspace{-0.05cm} \right) \hspace{-0.05cm} = &\frac{1}{\mu} \G(\psia) \hspace{-0.05cm} + \hspace{-0.05cm} \frac{\beta \lambda}{\epsilon} \GNN(\partial^{\alpha} \partial_{t} b) \\
&-\epsilon \indicatrice_{\{|\alpha| =  N \}} \nabla \cdot(\zetaa \V) + \Ra.
\end{aligned}
\end{equation*}

\noindent Furthermore $\Ra$ is controlled

\begin{equation*}
|\Ra|_{L^{2}} \leq M_{N} |(\epsilon \zeta,\beta b)|_{H^{N}} \E(U)^{\frac{1}{2}} + \frac{\beta \lambda}{\epsilon} M_{N} \left\lvert \partial_{t} b \right\rvert_{L^{\infty}_{t} H^{N}_{\! X}}.
\end{equation*}
\end{prop}

\begin{proof}
\noindent We adapt and follow the proof of Proposition 4.5 in  \cite{Lannes_ww}. See also Proposition 6.4 in \cite{Iguchi_tsunami}. Using Proposition \ref{differential_formula}, we obtain

\begin{equation*}
\begin{aligned}
\partial^{\alpha} \hspace{-0.05cm} \left( \frac{1}{\mu} \G(\psi) \hspace{-0.05cm} + \hspace{-0.05cm} \frac{\lambda \beta}{\epsilon} \GNN(\partial_{t} b) \right) \hspace{-0.05cm} = &\frac{1}{\mu} \G(\psia) \hspace{-0.05cm} + \hspace{-0.05cm} \frac{\beta \lambda}{\epsilon} \GNN(\partial^{\alpha} \partial_{t} b) \\
&-\epsilon \indicatrice_{\{|\alpha| =  N \}} \nabla \cdot(\zetaa \V) + \beta \GNN \left(\nabla \cdot \left( \partial^{\alpha} b \, \Vs \right) \right)\\
& + \Ra,
\end{aligned}
\end{equation*}

\noindent where $\Vs = \Vs[\epsilon \zeta, \beta b](\psi, B)$ is defined in Equation \eqref{Vs} and $\Ra$ is a sum of terms of the form (we adopt the notation of Remark \ref{G_notation} in Appendix \ref{shape_derivatives_estimates})

\begin{equation*}
A_{j,\iota,\nu} := d^{j} \left(\frac{1}{\mu} G_{\mu}(\partial^{\nu} \psi) + \frac{\beta \lambda}{\epsilon} G_{\mu}^{N \! N}(\partial^{\nu} \partial_{t} b) \right).(\partial^{\iota^{1}} \zeta,...,\partial^{\iota^{j}} \zeta;\partial^{\iota^{1}} b,...,\partial^{\iota^{j}} b),
\end{equation*}

\noindent where $j$ is an integer and $\iota^{1},...,\iota^{j}$ and $\nu$ are multi-index, and 

\begin{equation*}
\underset{1 \leq l \leq j}{\sum}|\iota^{l}| + |\nu| = N,
\end{equation*}

\noindent with $(j,|\iota^{l_{0}}|,|\nu|) \neq (1,N,0) \text{ and } (0,0,N)$. Here $\iota^{l_{0}}$ is such that $\underset{1 \leq l \leq j}{\max} |\iota^{l}| = |\iota^{l_{0}}|$. In particular, $1 \leq |\iota^{l_{0}}| \leq N$. We distinguish several cases. 

\medskip
\medskip
\hspace{-0.8cm} \underline{\textbf{a)} $|\iota^{l_{0}}| + |\nu| \leq N - 2$ and $|\iota^{l_{0}}| \leq N-3$ or $|\iota^{l_{0}}| + |\nu| \leq N$, $|\iota^{l_{0}}| \leq N-3$ and $|\nu| \leq N-2$ :}
\medskip

\noindent Applying the second point of Theorem 3.28 in \cite{Lannes_ww} and the first point of Proposition \ref{controls_dGNN} with $s=\frac{1}{2}$ and $t_{0}=\min(t_{0},\frac{3}{2})$, we get that

\begin{equation*}
|A_{j,\iota,\nu}|_{L^{2}} \leq \hspace{-0.07cm} M_{N} \underset{l}{\prod} |(\epsilon \partial^{\iota^{l}} \zeta,\beta \partial^{\iota^{l}} b)|_{H^{3}} \hspace{-0.1cm} \left\lbrack |\mathfrak{P} \partial^{\nu} \psi|_{H^{1}} + \frac{\beta \lambda}{\epsilon} |\partial^{\nu} \partial_{t} b|_{L^{2}} \right\rbrack,
\end{equation*}

\noindent and the result follows by Proposition \ref{psi_controls}.

\medskip
\medskip
\hspace{-0.8cm} \underline{\textbf{b)} $|\iota^{l_{0}}| = N-2$ and $|\nu| = 0 \text{ ,} 1 \text{ or } 2$ :}
\medskip
\medskip

\noindent We apply the fourth point of Theorem 3.28 in \cite{Lannes_ww} and the second point of Proposition \ref{controls_dGNN} with $s=\frac{1}{2}$ and $t_{0}=\max(t_{0},1)$, 

\begin{equation*}
|A_{j,\iota,\nu}|_{L^{2}} \hspace{-0.07cm} \leq \hspace{-0.07cm} M_{N} |(\epsilon \partial^{\iota^{l_{0}}} \zeta,\beta \partial^{\iota^{l_{0}}} b)|_{H^{\frac{3}{2}}} \hspace{-0.05cm} \underset{l \neq l_{0}}{\prod} |(\epsilon \partial^{\iota^{l}} \zeta,\beta \partial^{\iota^{l}} b)|_{H^{N-2}} \hspace{-0.1cm} \left\lbrack \hspace{-0.05cm} |\mathfrak{P} \partial^{\nu} \psi|_{H^{N-2}} \hspace{-0.05cm}+\hspace{-0.05cm} \frac{\beta \lambda}{\epsilon} |\partial^{\nu} \partial_{t} b|_{H^{N-2}} \hspace{-0.05cm} \right\rbrack \hspace{-0.07cm}.
\end{equation*}

\noindent Then, we get the result thanks to Proposition \ref{psi_controls}.

\medskip
\medskip
\hspace{-0.8cm} \underline{\textbf{c)} $\iota^{1} = \iota$ with $|\iota|=N-1$, $|\nu|=j=1$ :}
\medskip
\medskip

\noindent We proceed as in Proposition 4.5 in \cite{Lannes_ww}, using Theorem 3.15 in \cite{Lannes_ww} and Propositions \ref{differential_formula}, \ref{controls_GNN}, \ref{controls_w}.

\medskip
\medskip
\hspace{-0.8cm} \underline{\textbf{d)} $|\iota^{l_{0}}| = N-1$ and $|\nu| = 0$ :}
\medskip
\medskip

\noindent Here $j=2$ and $|\iota^{2}|=1$. For instance we consider that $l_{0} = 1$ and $|\iota^{2}| = 1$. Using the second inequality of Proposition \ref{controls_dGNN} we have

\begin{equation*}
\left\lvert d^{2} G_{\mu}^{N \! N}(\partial_{t} b).(\partial^{\iota^{1}} \zeta, \partial^{\iota^{2}} \zeta; \partial^{\iota^{1}} b, \partial^{\iota^{2}} b) \right\rvert_{L^{2}} \hspace{-0.05cm} \leq M_{N} \hspace{-0.07cm} \left\lvert \left(\epsilon \partial^{\iota^{1}} \zeta, \beta \partial^{\iota^{1}} b \right) \right\vert_{H^{1}} \hspace{-0.07cm} \left\lvert \left(\epsilon \partial^{\iota^{2}} \zeta, \beta \partial^{\iota^{2}} b \right) \right\lvert_{H^{2}} \hspace{-0.07cm} \left\lvert \partial_{t} b \right\lvert_{H^{2}} \hspace{-0.05cm}.
\end{equation*}

\noindent Furthermore, using two times Proposition \ref{differential_formula}, we get

\begin{align*}
 \frac{1}{\mu} d^{2} G_{\mu}(\psi).(\partial^{\iota^{1}} \zeta, \partial^{\iota^{2}} \zeta; \partial^{\iota^{1}} b, \partial^{\iota^{2}} b)& = -\frac{\epsilon}{\sqrt{\mu}} d\G \left(\partial^{\iota^{1}} \zeta  \frac{1}{\sqrt{\mu}} \w(\psi,0) \right).(\partial^{\iota^{2}} \zeta,0)\\
&\;\; - \frac{\epsilon}{\sqrt{\mu}} \G \left(\partial^{\iota^{1}} \zeta \frac{1}{\sqrt{\mu}} d \w(\psi,0).(\partial^{\iota^{2}} \zeta,0) \right)\\
&\;\; - \epsilon \nabla \cdot \left( \partial^{\iota^{1}} \zeta d \V(\psi,0). (\partial^{\iota^{2}} \zeta,0) \right)\\
&\;\; + \beta d \GNN \left(\partial^{\iota^{1}} b \; \Vs(\psi,0) \right).(0,\partial^{\iota^{2}} b)\\
&\;\; + \beta \GNN \left(\partial^{\iota^{1}} b \; d \Vs(\psi,0).(0,\partial^{\iota^{2}} b) \right). \\
\end{align*}

\noindent The control follows from the first inequality of Theorem 3.15 and Proposition 4.4 in \cite{Lannes_ww}, and Propositions \ref{controls_dVs}, \ref{controls_w} and \ref{controls_ws}.

\medskip
\medskip
\medskip
\hspace{-0.8cm} \underline{\textbf{e)} $|\nu| = N-1$ and $|\iota^{l_{0}}| = 1$ :}
\medskip
\medskip
\medskip

\noindent Here, $j = 1$. It is clear that

\begin{equation*}
\left\lvert \frac{\beta \lambda}{\epsilon} dG^{N \!N}_{\mu}(\partial^{\nu} \partial_{t} b).(\partial^{\iota^{1}} \zeta ;\partial^{\iota^{1}} b) \right\rvert_{L^{2}} \leq \frac{\beta \lambda}{\epsilon} M_{N} \left\lvert \partial_{t} b \right\rvert_{H^{N}}.
\end{equation*}

\noindent Furthermore,

\begin{align*}
\frac{1}{\mu} dG_{\mu}(\partial^{\nu} \psi).(\partial^{\iota^{1}} \zeta ;\partial^{\iota^{1}} b) = \frac{1}{\mu} dG_{\mu}( \psi_{(\nu)}).(\partial^{\iota^{1}} \zeta ;\partial^{\iota^{1}} b) + \frac{1}{\sqrt{\mu}} dG_{\mu} \left(\frac{\epsilon}{\sqrt{\mu}} \w \partial^{\nu} \zeta \right).(\partial^{\iota^{1}} \zeta ;\partial^{\iota^{1}} b).
\end{align*}

\noindent Then, using Theorem 3.15 in \cite{Lannes_ww} and Proposition \ref{psi_controls}, we get the result.
\end{proof}

\noindent This Proposition enables to quasilinearize the first equation of the water waves equations. For the second equation, it is the purpose of the following proposition.

\begin{prop}\label{linearization_2}
Let $T>0$, $N \geq \max(t_{0},1) + 3$, $b \in W^{1,\infty}(\R^{+};H^{N}(\RD))$ and $U \in E_{T}^{N}$, such that $\zeta$ and $b$ satisfy \eqref{nonvanishing} for all $0 \leq t \leq T$. We assume also that $\mu$ satisfies \eqref{parameters_constraints}. Then, for all $\alpha \in \mathbb{N}^{d}$, $1 \leq |\alpha| \leq N$, we have,

\begin{equation*}
\begin{aligned}
\partial^{\alpha} \! \left[ \frac{\epsilon}{2} |\nabla \psi|^{2} - \frac{\epsilon}{2 \mu} (1 + \epsilon^{2} \mu |\nabla \zeta|^{2}) \w^{2} \right] \! = \! \epsilon \V \cdot(\nabla \psia &+ \epsilon \partial^{\alpha} \zeta \nabla \w) - \frac{\epsilon}{\mu} \w \, \partial^{\alpha} \! G_{\mu}(\psi) \\
& - \beta \lambda \w \, \partial^{\alpha} \! G^{N \! N}_{\mu}(\partial_{t} b) +  \Sa.\\
\end{aligned}
\end{equation*}

\noindent Furthermore $\Sa$ is controlled

\begin{equation*}
|\mathfrak{P} \Sa|_{L^{2}} \leq  \epsilon M_{N} \E(U) + C \left( M_{N}, \frac{\beta \lambda}{\epsilon} \left\lvert \partial_{t} b \right\rvert_{L^{\infty}_{t} H^{N}_{\! X}} \right) \epsilon \E(U)^{\frac{1}{2}} + M_{N} \left( \frac{\beta \lambda}{\sqrt{\epsilon}} \left\lvert \partial_{t} b \right\rvert_{L^{\infty}_{t} H^{N}_{\! X}} \right)^{2}.
\end{equation*}
\end{prop}

\begin{proof}
\noindent The proof of this Proposition is similar to the proof of Proposition 4.10 in \cite{Lannes_ww} expect we use Propositions \ref{controls_w} and \ref{differential_formula}. See also Proposition 6.4 in \cite{Iguchi_tsunami}.
\end{proof}

\noindent Thanks to this linearization, we can "quasilinarize" equations \eqref{water_waves_equations}. It is the purpose of the next proposition. Let us introduce, the Rayleigh-Taylor coefficient

\begin{equation}\label{rt_coefficient}
\begin{aligned}
\rt := \rt(U, \beta b) = &1 + \epsilon \partial_{t} \left(\w[\epsilon \zeta, \beta b] \left( \psi, \frac{\beta \lambda}{\epsilon} \partial_{t} b \right) \right)\\
& + \epsilon^{2} \V[\epsilon \zeta, \beta b] \left( \psi, \frac{\beta \lambda}{\epsilon} \partial_{t} b \right) \cdot \nabla \left( \w[\epsilon \zeta, \beta b] \left( \psi, \frac{\beta \lambda}{\epsilon} \partial_{t} b \right) \right).
\end{aligned}
\end{equation} 

\noindent This quantity plays an important role. We also introduce two new operators, 

\begin{equation}
  \mathcal{A}[U,\beta b] := \begin{pmatrix}
\hspace{0.02cm} 0 \hspace{-0.02cm} & -\frac{1}{\mu} G_{\mu} [\epsilon \zeta, \beta b] \\
\rt(U, \beta b) & 0
\end{pmatrix}
\end{equation}

\noindent and

\begin{equation}
  \mathcal{B}[U, \beta b] := \begin{pmatrix}
\epsilon \nabla \cdot ( \bullet \hspace{0.01cm} \V)  & 0 \\
0 & \epsilon \V \cdot \nabla 
\end{pmatrix}.
\end{equation}

\noindent We can now quasilinearize the water waves equations. We use the same arguments as in Proposition 4.10 in \cite{Lannes_ww} and part \textbf{6} in \cite{Iguchi_tsunami}. Notice that we give here a precise estimate with respect to $\partial_{t} b$ and $P$ of the residuals $\Ra$ and $\Sa$ and that the minimal value of $N$, regularity of $U$, is smaller than in Proposition 4.10 in \cite{Lannes_ww}. 

\begin{prop}\label{quasilinearization}
Let $T > 0$, $N \geq \max(t_{0},1) + 3$, $b \in W^{2,\infty}(\R^{+}; H^{N}(\RD))$, $P \in L^{\infty}(\R^{+}; \Hdot^{N+1}(\RD))$ and $U \in E_{T}^{N}$  satisfies \eqref{nonvanishing} for all $0 \leq t \leq T$ and solving \eqref{water_waves_equations}. We assume also that $\mu$ satisfies \eqref{parameters_constraints}. Then, for all $\alpha \in \mathbb{N}^{d}$, $1 \leq |\alpha| \leq N$, we have,

\begin{equation}\label{quasilinear_system}
\begin{aligned}
\partial_{t} \Ua + \mathcal{A}[U, \beta b](\Ua) + \indicatrice_{\{|\alpha| =  N \}} \mathcal{B}[U, \beta b](\Ua) = & \left(\frac{\lambda \beta}{\epsilon} G^{N \! N}_{\mu}[0,0](\partial^{\alpha} \partial_{t} b),-  \partial^{\alpha} P \right)^{t}\\
&+ \left( \widetilde{\Ra},\Sa \right)^{t}.
\end{aligned}
\end{equation}

\noindent Furthermore, $\widetilde{\Ra}$ and $\Sa$ satisfy

\begin{equation*}
\left\{
\begin{aligned}
&|\widetilde{\Ra}|_{L^{2}} \leq M_{N} |(\epsilon \zeta,\beta b)|_{H^{N}} \E(U)^{\frac{1}{2}} + \frac{\beta \lambda}{\epsilon} M_{N} \left\lvert \partial_{t} b \right\rvert_{L^{\infty}_{t} H^{N}_{\! X}},\\
&|\mathfrak{P} \Sa|_{L^{2}} \leq  \epsilon M_{N} \E(U) + C \left( \hspace{-0.15cm} M_{N}, \frac{\beta \lambda}{\epsilon} \left\lvert \partial_{t} b \right\rvert_{L^{\infty}_{t} H^{N}_{\! X}} \hspace{-0.06cm} \right) \epsilon \E(U)^{\frac{1}{2}} + M_{N} \left( \frac{\beta \lambda}{\sqrt{\epsilon}} \left\lvert \partial_{t} b \right\rvert_{L^{\infty}_{t} H^{N}_{\! X}} \right)^{2} \hspace{-0.2cm}.
\end{aligned}
\right.
\end{equation*}
\end{prop}

\begin{proof}
\noindent Thanks to Proposition \ref{controls_dGNN}, we get

\begin{align*}
\left\vert \GNN(\partial^{\alpha} \partial_{t} b) - G_{\mu}^{N \! N}[0,0](\partial^{\alpha} \partial_{t} b) \right\rvert_{L^{2}} &\hspace{-0.05cm} \leq \hspace{-0.05cm} \int_{0}^{1} \hspace{-0.05cm} \left\lvert dG_{\mu}^{N \! N}[z\epsilon \zeta, z \beta b](\partial^{\alpha} \partial_{t} b).(\zeta, b) \right\rvert_{L^{2}} dz\\
&\leq M_{N} |(\epsilon \zeta, \beta b)|_{H^{N}} \left\lvert \partial^{\alpha} \partial_{t} b \right\rvert_{L^{\infty}_{t} H^{N}_{\! X}}.
\end{align*}

\noindent Then, denoting $\widetilde{\Ra} = \Ra + \GNN(\partial^{\alpha} \partial_{t} b) - G_{\mu}^{N \! N}[0,0](\partial^{\alpha} \partial_{t} b)$, we obtain the first equation thanks to Proposition \ref{linearization_1}. For the second equation, using Proposition \ref{linearization_2} and the first equation of the water waves problem, we have 

\begin{align*}
\partial_{t} \partial^{\alpha} \psi &= -\partial^{\alpha} \zeta - \epsilon \V \cdot(\nabla \psia + \epsilon \partial^{\alpha} \zeta \nabla \w) + \frac{\epsilon}{\mu} \w \, \partial^{\alpha} \! G_{\mu}(\psi) + \hspace{-0.05cm} \beta \w \, \partial^{\alpha} \! G^{N \! N}_{\mu}(\partial_{t}b) - \hspace{-0.05cm} \partial^{\alpha} P + \hspace{-0.05cm} \Sa\\
&= -\partial^{\alpha} \zeta - \epsilon \V \cdot(\nabla \psia + \epsilon \partial^{\alpha} \zeta \nabla \w) + \epsilon \w \partial_{t} \partial^{\alpha} \zeta - \partial^{\alpha} P + \Sa\\
&= -\partial^{\alpha} \zeta (1 + \epsilon \partial_{t} \w + \epsilon^{2} \V \cdot \nabla \w) - \epsilon \V \cdot \nabla \psia + \epsilon \partial_{t} (\w \partial^{\alpha} \zeta) - \partial^{\alpha} P + \Sa\\
&= -\rt \partial^{\alpha} \zeta - \epsilon \V \cdot \nabla \psia + \epsilon \partial_{t} (\w \partial^{\alpha} \zeta) - \partial^{\alpha} P + \Sa,
\end{align*}

\noindent and the result follows.

\end{proof}

\medskip

\noindent In the case of a constant pressure at the surface and a fixed bottom, it is well-known that system \eqref{quasilinear_system} is  symmetrizable if 

\begin{equation}\label{rt_constraints}
  \exists \, \mathfrak{a}_{min} > 0 \text{ ,  } \rt(U, \beta b) \geq \mathfrak{a}_{\text{min}}.
\end{equation} 

\noindent Then, we introduce the symmetrizer

\begin{equation}
 \mathcal{S}[U, \beta b] := \begin{pmatrix}
\rt(U, \beta b) & 0 \\
\hspace{0.02cm} 0 \hspace{0.02cm} & \frac{1}{\mu} \G
\end{pmatrix}.
\end{equation}

\noindent This symmetrization has an associated energy  

\begin{equation}
\begin{aligned}
& \mathcal{F}^{\alpha}(U) = \frac{1}{2} \left(\mathcal{S}[U, \beta b](\Ua),\Ua \right)_{L^{2}}, \text{ if } \alpha \neq 0 , \\
& \mathcal{F}^{0}(U) = \frac{1}{2} |\zeta|^{2}_{H^{\frac{3}{2}}} + \frac{1}{2} \left( \Lambda^{\frac{3}{2}} \psi,\frac{1}{\mu} \G(\Lambda^{\frac{3}{2}} \psi) \right)_{L^{2}},\\
&\mathcal{F}^{[N]}(U) = \underset{|\alpha| \leq N}{\sum} \mathcal{F}^{\alpha}(U).
\end{aligned}
\end{equation}

\noindent As in Lemma 4.27 in \cite{Lannes_ww}, it can be shown that $\mathcal{F}^{[N]}$ and $\mathcal{E}^{[N]}$ are equivalent in the following sense.

\begin{prop}\label{energy_equivalence}
Let $T > 0$, $N \in \mathbb{N}$, $U \in E^{N}_{T}$ satisfying \eqref{nonvanishing} and \eqref{rt_constraints} for all $0 \leq t \leq T$. Then, for all $0 \leq k \leq N$, $\mathcal{F}^{[k]}$ is comparable to $\mathcal{E}^{k}$

\begin{equation}
\frac{1}{|\rt(U, \beta b)|_{L^{\infty}} + M_{N}} \mathcal{F}^{[k]}[U,b] \leq \mathcal{E}^{k}(U) \leq \left(M_{N} + \frac{1}{\mathfrak{a}_{\text{min}}} \right) \mathcal{F}^{[k]}[U,b].
\end{equation}
\end{prop}

\subsection{Local existence}

\noindent The water water equations can be written as follow :

\begin{equation}
\partial_{t} U + \mathcal{N}(U) = (0,- P)^{t},
\end{equation}

\noindent with $\mathcal{N}(U) = (\mathcal{N}_{1}(U),\mathcal{N}_{2}(U))^{t}$ and 

\begin{equation}
\begin{aligned}
&\mathcal{N}_{1}(U) := -\frac{1}{\mu} \G(\psi) - \frac{\beta \lambda}{\epsilon} \GNN (\partial_{t} b), \\
&\mathcal{N}_{2}(U) := \zeta + \frac{\epsilon}{2} |\nabla \psi|^{2} - \frac{\epsilon}{2 \mu} \left(1+\epsilon^{2} \mu |\nabla \zeta|^{2} \right)  \left( \underline{w} [\epsilon \zeta, \beta b] \left( \psi, \frac{\beta \lambda}{\epsilon} \partial_{t} b \right) \right)^{2}.
\end{aligned}
\end{equation}

\noindent According to our quasilinearization, we need that $\rt$ be a positive real number. Therefore, we have to express $\rt$ without partial derivative with respect to $t$, particularly when $t=0$. It is easy to check that (we adopt the notation of Remark \ref{G_notation} in Appendix \ref{shape_derivatives_estimates})

\begin{equation}\label{rt_expression}
\begin{aligned}
\rt(U, \beta b) &= 1 + \epsilon^{2} \V[\epsilon \zeta, \beta b] \left(\psi, \frac{\beta \lambda}{\epsilon} \partial_{t} b \right) \cdot \nabla \left[ \w[\epsilon \zeta, \beta b] \left(\psi, \frac{\beta \lambda}{\epsilon} \partial_{t} b \right) \right] \\
&+\epsilon d \underline{w} \left( \psi,  \frac{\beta \lambda}{\epsilon} \partial_{t} b \right). \left( - \mathcal{N}_{1}(U), \partial_{t} b \right) + \epsilon \w[\epsilon \zeta, \beta b] \left(- P - \mathcal{N}_{2}(U), \frac{\beta \lambda}{\epsilon} \partial_{t}^{2} b \right).
\end{aligned}
\end{equation}

\noindent The following Proposition gives estimates for $\rt(U, \beta b)$. It is adapted from Proposition 6.6 in \cite{Iguchi_tsunami}.

\begin{prop}\label{controls_rt}
Let $T > 0$, $t_{0}> \frac{d}{2}$, $N \geq \max(t_{0},1) + 3$, $(\zeta,\psi) \in E_{T}^{N}$ is a solution of the water waves equations \eqref{water_waves_equations}, $P \in L^{\infty}(\R^{+}; \Hdot^{N+1}(\RD))$ and $b \in W^{2,\infty}(\R^{+};H^{N}(\RD))$, such that Condition \eqref{nonvanishing} is satisfied. We assume also that $\mu$ satisfies \eqref{parameters_constraints}. Then, for all $0 \leq t \leq T$,

\begin{align*}
| \rt(U, \beta b) -1 |_{H^{t_{0}}} \hspace{-0.05cm} & \leq \hspace{-0.1cm} C \hspace{-0.1cm} \left( \hspace{-0.1cm} M_{N}, \max( \beta \lambda, \beta) \left\lvert \partial_{t} b \right\rvert_{L^{\infty}_{t} H_{\! X}^{N}} \hspace{-0.1cm}, \epsilon \E(U)^{\frac{1}{2}} \right) \epsilon \E(U)^{\frac{1}{2}} \\
& + \epsilon M_{N} \hspace{-0.1cm} \left( \hspace{-0.1cm} \left\lvert \nabla P \right\rvert_{L^{\infty}_{t} H_{\! X}^{N}} \hspace{-0.1cm} + \hspace{-0.1cm} \frac{\beta \lambda}{\epsilon} \hspace{-0.05cm} \left\lvert \partial_{t}^{2} b \right\rvert_{L^{\infty}_{t} H_{\! X}^{N}} \hspace{-0.1cm} \right) \hspace{-0.1cm}.
\end{align*}

\noindent Furthermore, if $\partial_{t}^{3} b \in L^{\infty}(\R^{+}; H^{N}(\RD))$ and $\partial_{t} P \in L^{\infty}(\R^{+}; \Hdot^{N}(\RD))$, then,

\begin{align*}
\left\lvert \partial_{t}  \hspace{-0.05cm} \left( \rt(U, \beta b)  \right) \right\rvert_{H^{t_{0}}}  \hspace{-0.05cm} & \leq  \hspace{-0.1cm} C \hspace{-0.1cm} \left( \hspace{-0.1cm} M_{N}, \hspace{-0.05cm}  \max( \beta \lambda, \beta)\left\lvert \partial_{t} b \right\rvert_{W^{1,\infty}_{t} H_{\! X}^{N}} , \left\lvert \nabla P \right\rvert_{L^{\infty}_{t} H_{\! X}^{N}}, \epsilon \E(U)^{\frac{1}{2}} \right) \epsilon \E(U)^{\frac{1}{2}} \\
&+  \hspace{-0.05cm} \epsilon C \hspace{-0.1cm} \left(\hspace{-0.1cm} M_{N}, \hspace{-0.05cm}  \max( \beta \lambda, \beta) \left\lvert \partial_{t} b \right\rvert_{L^{\infty}_{t} H_{\! X}^{N}}  \hspace{-0.05cm} \right) \hspace{-0.1cm} \left( \hspace{-0.1cm} \left\lvert \nabla P \right\rvert_{W^{1,\infty}_{t} H_{\! X}^{N}}  \hspace{-0.1cm} +  \hspace{-0.1cm} \frac{\beta \lambda}{\epsilon} \left\lvert \partial_{t}^{2} b \right\rvert_{W^{1,\infty}_{t} H_{\! X}^{N}} \hspace{-0.1cm} \right) \hspace{-0.1cm}.
\end{align*}
\end{prop}

\begin{proof}
\noindent Using the first point of Proposition \ref{controls_w} and Product estimate \ref{product_estimate1} we have

\begin{equation*}
\left\lvert \V[\epsilon \zeta, \beta b] \hspace{-0.05cm} \left(\epsilon \psi, \beta \lambda \partial_{t} b \right) \hspace{-0.05cm} \cdot \hspace{-0.05cm} \nabla \left[ \w[\epsilon \zeta, \beta b] \left(\epsilon \psi, \beta \lambda \partial_{t} b \right) \right] \right\rvert_{H^{t_{0}}} \hspace{-0.05cm} \leq \hspace{-0.05cm} M_{N} \left( \hspace{-0.05cm} \left\lvert \mathfrak{P} \epsilon \psi \right\rvert_{H^{t_{0}+\frac{1}{2}}} \hspace{-0.05cm} + \hspace{-0.05cm} \beta \lambda \left\lvert \partial_{t} b \right\rvert_{L^{\infty}_{t} H_{\! X}^{t_{0}}} \right)^{\hspace{-0.05cm} 2} \hspace{-0.2cm}.
\end{equation*}

\noindent Furthermore, thanks to the first point of Proposition \ref{controls_dGNN} and the first point of Theorem 3.28 in \cite{Lannes_ww} we obtain

\begin{equation*}
\left\lvert \hspace{-0.03cm} \epsilon d \underline{w} \hspace{-0.08cm} \left( \hspace{-0.1cm} \psi, \hspace{-0.05cm}  \frac{\beta \lambda}{\epsilon} \partial_{t} b \hspace{-0.09cm} \right) \hspace{-0.12cm} . \hspace{-0.07cm} \left( - \mathcal{N}_{1}(\hspace{-0.02cm} U \hspace{-0.02cm}), \hspace{-0.05cm} \partial_{t} b \right) \hspace{-0.03cm} \right\rvert_{H^{t_{0}}} \hspace{-0.6cm} \leq \hspace{-0.12cm} M_{N} \hspace{-0.1cm} \left\lvert \left( \epsilon \mathcal{N}_{1}(U), \beta  \partial_{t} b \right) \right\rvert_{H^{t_{0}+1}} \hspace{-0.1cm} \left(\hspace{-0.1cm} \left\lvert \mathfrak{P} \epsilon \psi \right\rvert_{H^{t_{0}+\frac{1}{2}}} \hspace{-0.1cm}+\hspace{-0.1cm}  \beta \lambda \left\lvert \partial_{t} b \right\rvert_{L^{\infty}_{t} H_{\! X}^{t_{0}}} \hspace{-0.1cm} \right) \hspace{-0.12cm} .
\end{equation*}

\noindent Then, the first inequality follows easily from Proposition \ref{controls_w}, Proposition \ref{psi_controls} and Product estimate \ref{product_estimate1}. The second inequality can be proved similarly.
\end{proof}

\noindent We can now prove Theorems \ref{existence_uniqueness} and \ref{long_time_existence}. We recall that $\delta := \max(\epsilon,\beta^{2})$.

\begin{proof}
\noindent We slice up this proof in three parts. First we regularize and symmetrize the equations, then we find some energy estimates and finally we conclude by convergence. We only give the energy estimates in this paper and a carefully study of the nonlinearities of the water waves equations is done. We refer to the proof of Theorem 4.16 in \cite{Lannes_ww} for the regularization, the convergence and the uniqueness (see also part 7 in \cite{Iguchi_tsunami}). For Theorem \ref{existence_uniqueness} (respectively Theorem \ref{long_time_existence}), we assume that $U$ solves \eqref{water_waves_equations} on $\left[0,T \right]$ $\left( \text{ respectively on } \left[0, \frac{T}{ \sqrt{\delta}} \right] \right)$ and that \eqref{nonvanishing} and \eqref{rt_constraints} are satisfied for $\frac{h_{min}}{2}$ and $\frac{\mathfrak{a}_{min}}{2}$ on $\left[0,T \right]$ $\left( \text{ respectively on } \left[0, \frac{T}{ \sqrt{\delta}} \right] \right)$ for some $T > 0$.

\medskip
\medskip
\hspace{-0.8cm} \underline{\textbf{a) $|\alpha|=0$}, The 0 - energy}
\medskip
\medskip

\noindent We proceed as in Subsection 4.3.4.3 in \cite{Lannes_ww} and part 6 in \cite{Iguchi_tsunami}. We have

\begin{equation}
\begin{aligned}
\frac{d}{dt} \mathcal{F}^{0}(U) = &\frac{1}{2 \mu} \left(d \G (\Lambda^{\frac{3}{2}} \psi).(\partial_{t} \zeta, \partial_{t} b), \Lambda^{\frac{3}{2}} \psi \right)_{\hspace{-0.05cm} L^{2}} \hspace{-0.1cm} + \frac{\beta \lambda}{\epsilon} \left(\Lambda^{\frac{3}{2}} \GNN(\partial_{t} b), \Lambda^{\frac{3}{2}} \zeta \right )_{\hspace{-0.05cm} L^{2}} \\
& - \left(\Lambda^{\frac{3}{2}} \left(\mathcal{N}_{2}(U) - \zeta \right),\frac{1}{\mu} \G(\Lambda^{\frac{3}{2}} \psi) \right)_{\hspace{-0.05cm} L^{2}} \hspace{-0.3cm} - \left(\frac{1}{\mu} \G(\Lambda^{\frac{3}{2}} \psi), \Lambda^{\frac{3}{2}} P \right)_{\hspace{-0.05cm} L^{2}}.
\end{aligned}
\end{equation}

\noindent We have to control all the term in the r.h.s.

\medskip
\medskip
\noindent $\star$ Control of $\frac{\beta \lambda}{\epsilon} \left(\Lambda^{\frac{3}{2}} \GNN(\partial_{t} b), \Lambda^{\frac{3}{2}} \zeta \right)_{\hspace{-0.05cm} L^{2}}$.
\medskip
\medskip

\noindent Using Proposition \ref{controls_GNN}, we get

\begin{equation*}
\left\lvert \frac{\beta \lambda}{\epsilon} \left(\Lambda^{\frac{3}{2}} \GNN(\partial_{t} b), \Lambda^{\frac{3}{2}} \zeta \right)_{\hspace{-0.05cm} L^{2}} \right\lvert \leq M_{N} \frac{\beta \lambda}{\epsilon} \left\lvert \partial_{t} b \right\rvert_{L^{\infty}_{t} H_{\! X}^{N}} \E(U)^{\frac{1}{2}}.
\end{equation*}

\medskip
\medskip
\noindent $\star$ Control of $\left(\Lambda^{\frac{3}{2}} \left(\mathcal{N}_{2}(U) - \zeta \right),\frac{1}{\mu} \G(\Lambda^{\frac{3}{2}} \psi) \right)_{\hspace{-0.05cm} L^{2}}$.
\medskip
\medskip

\noindent Using Proposition \ref{psi_controls} and Proposition \ref{controls_w}, we get

\begin{align*}
\left\lvert \hspace{-0.05cm} \left(\hspace{-0.1cm} \Lambda^{\frac{3}{2}} \hspace{-0.1cm} \left(\mathcal{N}_{2}(U) \hspace{-0.05cm} - \hspace{-0.05cm} \zeta \right) \hspace{-0.05cm} , \hspace{-0.05cm} \frac{1}{\mu} \G(\Lambda^{\frac{3}{2}} \psi) \hspace{-0.05cm} \right)_{\hspace{-0.06cm} L^{2}} \hspace{-0.06cm} \right\lvert \hspace{-0.05cm} &\leq \hspace{-0.05cm} \left\lvert \mathcal{N}_{2}(U) - \zeta \right\rvert_{H^{\frac{3}{2}}} \left\lvert \frac{1}{\mu} \G(\Lambda^{\frac{3}{2}} \psi) \right\rvert_{L^{2}},\\
&\leq  \hspace{-0.05cm} \epsilon M_{N} \E(U)^{\frac{3}{2}} \hspace{-0.1cm} + M_{N} \left(\frac{\beta \lambda}{\epsilon} \left\lvert \partial_{t} b \right\vert_{L^{\infty}_{t} H_{\! X}^{N}}\right)^{2} \hspace{-0.2cm} \epsilon \E(U).
\end{align*}

\medskip
\medskip
\noindent $\star$ Control of $\left(\frac{1}{\mu} \G(\Lambda^{\frac{3}{2}} \psi), \Lambda^{\frac{3}{2}} P \right)_{\hspace{-0.05cm} L^{2}}$.
\medskip
\medskip

\noindent We get, using Remark 3.13 in \cite{Lannes_ww},

\begin{equation*}
 \left\lvert \left(\frac{1}{\mu} \G(\Lambda^{\frac{3}{2}} \psi), \Lambda^{\frac{3}{2}} P \right)_{\hspace{-0.05cm} L^{2}}  \right\lvert \leq  M_{N} E^{N}(U)^{\frac{1}{2}} |\nabla P|_{L^{\infty}_{t} H_{\! X}^{N}}.
\end{equation*}

\medskip
\medskip
\noindent $\star$ Control of $\frac{1}{2 \mu} \left(d \G (\Lambda^{\frac{3}{2}} \psi).(\partial_{t} \zeta, \partial_{t} b), \Lambda^{\frac{3}{2}} \psi \right)_{\hspace{-0.05cm} L^{2}}$.
\medskip
\medskip

\noindent Using Proposition 3.29 in \cite{Lannes_ww}, the second point of Theorem 3.15 in \cite{Lannes_ww}, Proposition \ref{controls_GNN} and Proposition \ref{psi_controls} , we get

\begin{align*}
 \left\lvert\frac{1}{\mu} \hspace{-0.1cm} \left( \hspace{-0.05cm} d \G (\Lambda^{\frac{3}{2}} \psi).(\partial_{t} \zeta, \partial_{t} b), \Lambda^{\frac{3}{2}} \psi \hspace{-0.08cm} \right)_{\hspace{-0.07cm} L^{2}} \hspace{-0.05cm} \right\lvert \hspace{-0.05cm} &\leq  M_{N} \left\lvert \left(\epsilon \mathcal{N}_{1}(U), \beta \partial_{t} b \right) \right\rvert_{H^{N-2}} \left\lvert \mathfrak{P} \psi \right\rvert_{H^{\frac{3}{2}}}^{2}\\
&\leq \hspace{-0.05cm} M_{N} \epsilon \E(U)^{\frac{3}{2}} + \hspace{-0.05cm} \max(\beta, \beta \lambda) \left\lvert \partial_{t} b \right\vert_{L^{\infty}_{t} H_{\! X}^{N}} \E(U).
\end{align*}

\noindent Finally, gathering all the previous estimates, we get that

\begin{equation}\label{zero_control}
\begin{aligned}
\frac{d}{dt} \mathcal{F}^{0}(U) &\leq \epsilon M_{N} \E (U)^{\frac{3}{2}} + \hspace{-0.01cm} M_{N} C \left(\rho_{\max}, \left\lvert \partial_{t} b \right\vert_{L^{\infty}_{t} H_{\! X}^{N}} \right) \max(\epsilon,\beta) \E(U)\\
&+ M_{N} \sqrt{E^{N}(U)} \left( |\nabla P|_{L^{\infty}_{t} H^{N}_{\! X}} + \frac{\beta \lambda}{\epsilon} \left\lvert \partial_{t} b \right\rvert_{L^{\infty}_{t} H^{N}_{\! X}} \right).
\end{aligned}
\end{equation}

\medskip
\medskip
\hspace{-0.8cm} \underline{\textbf{b) $|\alpha| > 0$}, the higher orders energies}
\medskip
\medskip

\noindent \noindent We proceed as in Subsection 4.3.4.3 in \cite{Lannes_ww} and part 6 in \cite{Iguchi_tsunami}. A simple computation gives

\begin{equation}
\begin{aligned}
\frac{d}{dt} \hspace{-0.05cm} &\left( \mathcal{F}^{\alpha}(U) \right) \hspace{-0.05cm} = - \epsilon \indicatrice_{\{|\alpha| =  N \}} \left( \rt \zetaa, \nabla \cdot \left(  \zetaa  \V \right) \right)_{\hspace{-0.05cm} L^{2}} + \left( \rt \zetaa, \frac{\beta \lambda}{\epsilon} G^{N \! N}_{\mu} [0,0](\partial_{t} \partial^{\alpha} b) + \widetilde{\Ra} \right)_{\hspace{-0.05cm} L^{2}} \\
&- \epsilon \indicatrice_{\{|\alpha| =  N \}} \left( \frac{1}{\mu} \G(\psia), [\V \cdot \nabla \psia]\right)_{\hspace{-0.05cm} L^{2}} \hspace{-0.1cm} + \left( \frac{1}{\mu} \G(\psia),  \Sa - \partial^{\alpha} P \right)_{\hspace{-0.05cm} L^{2}}\\
&+ \frac{1}{2} \left( \partial_{t} \rt \zetaa, \zetaa \right)_{\hspace{-0.05cm} L^{2}} + \frac{1}{2} \left( \frac{1}{\mu} d \G(\psia).(\partial_{t} \zeta, \partial_{t} b), \psia \right)_{\hspace{-0.05cm} L^{2}}.
\end{aligned}
\end{equation}

\noindent We have to control all the term in the r.h.s.

\medskip
\medskip
\noindent $\star$ Control of $\left( \partial_{t} \rt \zetaa, \zetaa \right)_{\hspace{-0.05cm} L^{2}}$.
\medskip
\medskip

\noindent Using the second point of Proposition \ref{controls_rt} we get

\begin{align*}
\left\lvert \hspace{-0.07cm} \left( \partial_{t} \rt \zetaa, \zetaa \right)_{\hspace{-0.05cm} L^{2}} \hspace{-0.05cm} \right\rvert \hspace{-0.07cm} &\leq  M_{N} C\left( \hspace{-0.1cm} \rho_{\max}, \left\lvert \partial_{t} b \right\rvert_{W^{1,\infty}_{t} H_{\! X}^{N}} , \left\lvert \nabla P \right\rvert_{L^{\infty}_{t} H_{\! X}^{N}}, \epsilon \E(U)^{\frac{1}{2}} \right)  \epsilon \E(U)^{\frac{3}{2}}\\
&+  \hspace{-0.05cm} C \hspace{-0.1cm} \left(\hspace{-0.1cm} M_{N}, \hspace{-0.05cm} \beta \lambda \left\lvert \partial_{t} b \right\rvert_{L^{\infty}_{t} H_{\! X}^{N}}  \hspace{-0.05cm} \right) \hspace{-0.1cm} \left( \hspace{-0.1cm} \left\lvert \nabla P \right\rvert_{W^{1,\infty}_{t} H_{\! X}^{N}}  \hspace{-0.1cm} +  \hspace{-0.1cm} \frac{\beta \lambda}{\epsilon} \left\lvert \partial_{t}^{2} b \right\rvert_{W^{1,\infty}_{t} H_{\! X}^{N}} \hspace{-0.1cm} \right) \epsilon \E(U).
\end{align*}

\medskip
\medskip
\noindent $\star$ Control of $\left( \rt \zetaa, \frac{\beta \lambda}{\epsilon} G^{N \! N}_{\mu} [0,0](\partial_{t} \partial^{\alpha} b) \right)_{\hspace{-0.05cm} L^{2}}$.
\medskip
\medskip

\noindent We get, thanks to Proposition \ref{controls_rt} and \ref{controls_GNN},

\begin{equation*}
\left\lvert \hspace{-0.1cm} \left( \hspace{-0.14cm} \rt \zetaa \hspace{-0.03cm}, \hspace{-0.08cm} \frac{\beta \hspace{-0.03cm} \lambda}{\epsilon} \hspace{-0.03cm} G^{N \! N}_{\mu} \hspace{-0.08cm} [\hspace{-0.03cm} 0 \hspace{-0.03cm}, \hspace{-0.09cm} 0 \hspace{-0.03cm}] \hspace{-0.06cm} (\hspace{-0.08cm} \partial_{\hspace{-0.02cm} t} \hspace{-0.03cm} \partial^{\hspace{-0.02cm} \alpha} \hspace{-0.03cm} b \hspace{-0.04cm}) \hspace{-0.14cm} \right)_{\hspace{-0.19cm} L^{2}} \hspace{-0.07cm} \right\rvert \hspace{-0.13cm} \leq  \hspace{-0.13cm} C \hspace{-0.12cm} \left( \hspace{-0.12cm} \rho_{\max} \hspace{-0.03cm},\hspace{-0.07cm} \mu_{\max} \hspace{-0.03cm}, \hspace{-0.08cm} \left\lvert \hspace{-0.03cm} b \hspace{-0.03cm} \right\rvert_{W^{2,\hspace{-0.03cm}\infty}_{t} \hspace{-0.08cm} H_{\! X}^{N}} \hspace{-0.12cm} , \hspace{-0.08cm} \left\lvert \hspace{-0.04cm} \nabla \hspace{-0.06cm} P \hspace{-0.03cm} \right\rvert_{L^{\infty}_{t} \hspace{-0.08cm} H_{\! X}^{N}} \hspace{-0.13cm} , \hspace{-0.07cm} \epsilon \hspace{-0.03cm} \E \hspace{-0.09cm} (\hspace{-0.03cm} U \hspace{-0.03cm})^{\frac{1}{2}} \hspace{-0.09cm} \right) \hspace{-0.15cm} \frac{\beta \hspace{-0.03cm} \lambda}{\epsilon} \hspace{-0.08cm} \left\lvert \hspace{-0.04cm} \partial_{\hspace{-0.01cm} t} \hspace{-0.03cm} b \hspace{-0.03cm} \right\rvert_{L^{\infty}_{t} \hspace{-0.08cm} H_{\! X}^{N}} \E(U)^{\frac{1}{2}} \hspace{-0.03cm}.
\end{equation*}

\medskip
\medskip
\noindent $\star$ Controls of $\epsilon \indicatrice_{\{|\alpha| =  N \}} \left( \rt \zetaa,  \nabla \cdot \left(  \zetaa  \V \right) \right)_{\hspace{-0.05cm} L^{2}}$.
\medskip
\medskip

\noindent Inspired by Subsection 4.3.4.3 in \cite{Lannes_ww}, a simple computation gives

\begin{align*}
\left\lvert \hspace{-0.03cm}  \epsilon \hspace{-0.1cm} \left( \rt \zetaa, \hspace{-0.1cm} \nabla \hspace{-0.1cm} \cdot \hspace{-0.1cm} \left( \zetaa \V \right) \right)_{\hspace{-0.06cm} L^{2}} \hspace{-0.03cm} \right\rvert \hspace{-0.03cm} &= \left\lvert \epsilon \left( \rt \zetaa \nabla \cdot \V, \zetaa \right)_{\hspace{-0.03cm} L^{2}} \hspace{-0.03cm}  \right\rvert\\
&\leq \hspace{-0.1cm} C \hspace{-0.1cm} \left(\hspace{-0.07cm} \rho_{\max}, \hspace{-0.05cm} \mu_{\max}, \hspace{-0.04cm} \left\lvert b \right\rvert_{W^{2,\infty}_{t} H_{\! X}^{N}} \hspace{-0.07cm} , \hspace{-0.07cm} \left\lvert \nabla P \right\rvert_{L^{\infty}_{t} H_{\! X}^{N}} \hspace{-0.07cm} , \hspace{-0.05cm} \delta \E(U) \hspace{-0.05cm} \right) \hspace{-0.1cm} \epsilon \hspace{-0.1cm} \left[\hspace{-0.03cm} \E(U)^{\frac{3}{2}} \hspace{-0.07cm} + \hspace{-0.07cm} \E(U) \hspace{-0.03cm} \right] \hspace{-0.13cm} .
\end{align*}

\noindent using Proposition \ref{controls_rt} and Proposition \ref{controls_w}.

\medskip
\noindent $\star$ Controls of $\left(\hspace{-0.04cm} \frac{1}{\mu} \G(\psia), \hspace{-0.02cm} \Sa \hspace{-0.05cm} - \hspace{-0.05cm} \partial^{\alpha} \hspace{-0.02cm} P \hspace{-0.03cm} \right)_{\hspace{-0.05cm} L^{2}} \hspace{-0.03cm}$,  $\hspace{-0.03cm} \left( \hspace{-0.03cm} \frac{1}{\mu} d \G(\psia).(\partial_{t} \zeta, \partial_{t} b), \psia \right)_{\hspace{-0.07cm} L^{2}} \hspace{-0.1cm}$ and $\left( \rt \zetaa, \widetilde{\Ra}\right)_{\hspace{-0.05cm} L^{2}}$.
\medskip

\noindent We can use the same arguments as in the third and the fourth point of part a) using 
Propositions \ref{quasilinearization} and \ref{controls_rt}.

\medskip
\noindent $\star$ Control of $\epsilon \left( \frac{1}{\mu} \G(\psia), [\V \cdot \nabla \psia]\right)_{\hspace{-0.05cm} L^{2}}$.
\medskip
\medskip

\noindent We refer to the Subsection 4.3.4.3 and Proposition 3.30 in \cite{Lannes_ww} for this control.

\medskip
\medskip
\noindent Gathering the previous estimates and using Proposition \ref{energy_equivalence}, we obtain that

\begin{equation}\label{alpha_control}
\small{
\begin{aligned}
\frac{d}{dt} \mathcal{F}^{N} \hspace{-0.1cm} (\hspace{-0.03cm} U \hspace{-0.03cm}) \leq  C \hspace{-0.06cm} \Big( \hspace{-0.07cm} \rho_{\max},  \hspace{-0.06cm} \frac{1}{h_{\text{min}}}, &\hspace{-0.01cm} \mu_{\max},\hspace{-0.06cm}  \frac{1}{\mathfrak{a}_{\text{min}}}, \hspace{-0.05cm} \left\lvert \hspace{-0.02cm} b \hspace{-0.02cm} \right\rvert_{W^{3,\infty}_{t} \hspace{-0.03cm} H_{\! X}^{N}} \hspace{-0.08cm} , \hspace{-0.06cm} \left\lvert \hspace{-0.03cm} \nabla P \hspace{-0.03cm} \right\rvert_{W^{1,\infty}_{t} \hspace{-0.03cm} H_{\! X}^{N}} \hspace{-0.08cm}, \hspace{-0.03cm} \epsilon  \mathcal{F}^{N} \hspace{-0.1cm} (U)^{\frac{1}{2}} \hspace{-0.1cm} \Big) \hspace{-0.08cm} \times\\
& \left(\hspace{-0.08cm} \epsilon \mathcal{F}^{N} \hspace{-0.1cm} (\hspace{-0.03cm} U \hspace{-0.03cm})^{\frac{3}{2}} \hspace{-0.08cm} + \hspace{-0.08cm} \max(\epsilon, \beta) \mathcal{F}^{N} \hspace{-0.1cm} (\hspace{-0.03cm} U \hspace{-0.03cm}) \hspace{-0.08cm} + \hspace{-0.08cm} \mathcal{F}^{N} \hspace{-0.1cm} (\hspace{-0.03cm} U \hspace{-0.03cm})^{\frac{1}{2}} \hspace{-0.13cm} \left[ \frac{\beta \lambda}{\epsilon} \left\lvert \partial_{t} b \right\rvert_{L^{\infty}_{t} H_{\! X}^{N}} \hspace{-0.08cm} + \hspace{-0.08cm} \left\lvert \nabla P \right\rvert_{L^{\infty}_{t} H_{\! X}^{N}}  \right] \right) \hspace{-0.08cm}.\\
\end{aligned}
}
\end{equation}

\noindent Then, we easily prove Theorem \ref{existence_uniqueness}, using the same arguments as Subsection 4.3.4.4 in \cite{Lannes_ww}. Furthermore, for $\alpha \in \left[0, \frac{1}{2} \right]$, defining $\widetilde{ \mathcal{F}^{N}}(U)(\tau) =  \delta^{2 \alpha} \mathcal{F}^{N}(U) \left( \frac{\tau}{\delta^{\alpha}} \right)$, we get

\begin{equation*}
\frac{d}{d\tau} \widetilde{\mathcal{F}^{N}}(U) \leq C \hspace{-0.1cm} \left(\hspace{-0.05cm} \rho_{\max}, \hspace{-0.05cm} \mu_{\max}, \hspace{-0.05cm} \frac{1}{\mathfrak{a}_{\text{min}}},  \hspace{-0.05cm} \frac{1}{h_{\text{min}}}, \hspace{-0.05cm} \left\lvert b \right\rvert_{W^{3,\infty}_{t} H_{\! X}^{N}} , \left\lvert \nabla P \right\rvert_{W^{1,\infty}_{t} H_{\! X}^{N}}, \widetilde{\mathcal{F}^{N}}(U) \right).
\end{equation*}

\noindent We can also apply the same arguments as Subsection 4.3.4.4 in \cite{Lannes_ww} and Theorem \ref{long_time_existence} follows.
\end{proof}

\subsection{Hamiltonian system}

\noindent In this section we prove that the water waves problem \eqref{water_waves_equations} is a Hamiltonian system in the Sobolev framework. This extends the classical result of Zakharov (\cite{Zakharov}) to the case where the bottom is moving and the atmospheric pressure is not constant (see also \cite{Craig_Guyenne_Hamiltonian}). In the case of a moving bottom, P. Guyenne and D. P. Nicholls already pointed out it in \cite{numeric_Guyenne_moving_bott} \footnote{It seems that there is a typo in their hamiltonian; "$-\zeta v$" should read "$+\zeta v$".}. We have to introduce the  Dirichlet-Dirichlet and the Neumann-Dirichlet operators 

\begin{equation}
\left\{
\begin{array}{l}
\GDD(\psi) = \left(\Phi^{S}\right)_{|z=-1+\beta b},\\
\GND(\partial_{t} b) = \left(\Phi^{B} \right)_{|z=-1+\beta b},
\end{array}
\right.
\end{equation}

\noindent where $\Phi^{S}$ is defined in \eqref{Laplace_surf} and $\Phi^{B}$ is defined in \eqref{Laplace_bott}.  We postpone the study of these operators to appendix \ref{The Dirichlet-Neumann and the Neumann-Neumann operators}.

\begin{remark}\label{value_Psi_boundary}
\noindent If we denote $\Phi := \Phi^{S} + \frac{\beta \lambda \mu}{\epsilon} \Phi^{B}$, $\Phi$ satisfies

 \begin{equation*}
  \left\{
  \begin{aligned}
   & \Delta^{\mu}_{\! X,z} \Phi = 0 \text{ in } \Omega_{t} \text{ ,} \\
   &\Phi_{|z=\epsilon \zeta} = \psi \text{ , } \sqrt{1+\beta^{2} |\nabla b|^{2}} \partial_{\textbf{n}} \Phi_{|z=-1+\beta b} = \frac{\beta \lambda \mu}{\epsilon} \partial_{t} b.
  \end{aligned}
  \right.
\end{equation*} 

\noindent Then 

\begin{equation}
\sqrt{1+\epsilon^{2} |\nabla \zeta|^{2}} \partial_{\textbf{n}} \Phi_{|z=\epsilon \zeta} = G_{\mu}[\epsilon \zeta, \beta b](\psi) + \frac{\beta \mu \lambda}{\epsilon} G_{\mu}^{N \! N}[\epsilon \zeta, \beta b](\partial_{t} b),
\end{equation}

\noindent and

\begin{equation}
\Phi_{|z=-1+\beta b} = G_{\mu}^{D \! D}[\epsilon \zeta, \beta b](\psi) + \frac{\beta \mu \lambda}{\epsilon} G_{\mu}^{N \! D}[\epsilon \zeta, \beta b](\partial_{t} b).
\end{equation}

\end{remark}

\begin{thm}
\noindent Let $T > 0$, $t_{0} > \frac{d}{2}$, $\zeta, b \in \mathcal{C}^{0}([0,T]; H^{t_{0}+1}(\RD))$, $\psi \in \mathcal{C}^{0}([0,T]; H^{2}(\RD))$, $\partial_{t} b \in \mathcal{C}^{0}([0,T]; H^{1}(\RD))$, $P \in \mathcal{C}^{0}([0,T]; L^{2}(\RD))$ such that $(\zeta, \psi)$ is a solution of  \eqref{water_waves_equations}. Define $H = H(\zeta,\psi) = \mathcal{T}(\zeta,\psi) + \mathcal{U}(\zeta,\psi)$, where $\mathcal{T}(\zeta,\psi) = \mathcal{T}$ is

\begin{equation}\label{kinetic_energy}
\mathcal{T} \hspace{-0.1cm} = \hspace{-0.1cm} \frac{1}{2 \mu} \hspace{-0.1cm} \int_{\Omega_{t}} \hspace{-0.1cm} \left\lvert \nabla^{\mu}_{\! X,z} \hspace{-0.1cm} \left( \Phi^{S} + \frac{\beta \lambda \mu}{\epsilon} \Phi^{B} \right) \right\rvert^{2} \hspace{-0.15cm} + \hspace{-0.1cm} \int_{\RD} \frac{\beta \lambda}{\epsilon} \partial_{t} b \left( \hspace{-0.1cm} G^{D \! D}_{\mu}[\epsilon \zeta,\beta b](\psi) \hspace{-0.05cm} + \hspace{-0.05cm} \frac{\beta \lambda \mu}{\epsilon} G^{N \!D}_{\mu}[\epsilon \zeta,\beta b](\partial_{t} b) \hspace{-0.1cm} \right) \hspace{-0.05cm} ,
\end{equation}

\noindent and $\mathcal{U}(\zeta,\psi) = \mathcal{U}$ is

\begin{equation}\label{potential_energy}
\mathcal{U} = \frac{1}{2} \int_{\RD} \zeta^{2} dX + \int_{\RD} \zeta P dX.
\end{equation} 

\noindent Then, the water waves equations \eqref{water_waves_equations} take the form

\begin{equation*}
\partial_{t} \begin{pmatrix}
\zeta \\
\psi
\end{pmatrix} = \begin{pmatrix}
\;0 \quad I \\
-I \quad 0
\end{pmatrix} \begin{pmatrix}
\partial_{\zeta} H \\
\partial_{\psi} H
\end{pmatrix}. 
\end{equation*}

\end{thm}

\begin{remark}\label{T_formula}
\noindent $\mathcal{T}$ is the sum of the kinetic energy and the moving bottom contribution and $\mathcal{U}$ the sum of the potential energy and the pressure contribution. Using Green's formula and Remark \ref{value_Psi_boundary} we obtain that 

\begin{align*}
\mathcal{T} = & \frac{1}{2} \int_{\RD} \psi \left( \frac{1}{\mu} G_{\mu}[\epsilon \zeta,\beta b](\psi) + \frac{\beta \lambda}{\epsilon} G^{N \! N}_{\mu}[\epsilon \zeta,\beta b](\partial_{t} b) \right) dX\\
& + \frac{1}{2} \int_{\RD} \frac{\beta \lambda}{\epsilon} \partial_{t} b \left( G^{D \! D}_{\mu}[\epsilon \zeta,\beta b](\psi)  + \frac{\beta \lambda \mu}{\epsilon} G^{N \!D}_{\mu}[\epsilon \zeta,\beta b](\partial_{t} b) \right)dX ,
\end{align*}
\end{remark}

\begin{proof}

\noindent Using the linearity of the Dirichlet-Neumann and the Dirichlet-Dirichlet operators with respect to $\psi$ and the fact that the adjoint of $\GNN$ is $\GDD$ (see Proposition \ref{main_properties_GNN}), we get that 

\begin{equation*}
\partial_{\psi} H = \frac{1}{\mu} G_{\mu}[\epsilon \zeta,\beta b](\psi) + \frac{\beta \lambda}{\epsilon} G^{N \! N}_{\mu}[\epsilon \zeta,\beta b](\partial_{t} b).
\end{equation*}

\noindent Applying Proposition \ref{differential_formula} (which provides explicit expressions for shape derivatives) and remark \ref{T_formula},  we obtain that

\begin{align*}
2 &\partial_{\zeta} H = - \frac{\epsilon}{\mu} G_{\mu}[\epsilon \zeta,\beta b](\psi) \w + \epsilon \nabla \psi \cdot \V - \epsilon \frac{\beta \lambda}{\epsilon} G^{N \! N}_{\mu}[\epsilon \zeta,\beta b](\partial_{t} b) \w + 2 P + 2 \zeta,\\
&= - \frac{\epsilon}{\mu} G_{\mu}[\epsilon \zeta,\beta b](\psi) \w + \epsilon \nabla \psi \cdot \nabla \psi - \epsilon^{2} \w \nabla \psi \cdot \nabla \zeta - \epsilon \frac{\beta \lambda}{\epsilon} G^{N \! N}_{\mu}[\epsilon \zeta,\beta b](\partial_{t} b) \w + 2 P + 2 \zeta,\\
&= \epsilon \left\lvert \nabla \psi \right\rvert^{2} - \frac{\epsilon}{\mu} \w^{2} \left( 1 + \epsilon^{2} \mu |\nabla \zeta|^{2} \right) + 2 P + 2 \zeta,
\end{align*}

\noindent which ends the proof.
\end{proof}

\noindent In fact, working in the Beppo Levi framework for $\psi$ requires that $\frac{1}{|D|} \partial_{t} b \in L^{2}(\RD)$ and results that are not dealing with this paper. 

\section{Asymptotic models}\label{Asymptotic models}

\noindent In this part, we derive some asymptotic models in order to model two different types of tsunamis. The most important phenomenon that we want to catch is the Proudman resonance (see for instance \cite{MVR_Meteotsunamis} or \cite{vilibic_Proudman_resonance} for an explanation of the Proudman resonance) and the submarine landslide tsunami phenomenon (see \cite{Levin_tsunamis}, \cite{Energy_landslide_tsunami} or \cite{Tinti_numerical_simulation}). These resonances occur in a linear case. The duration of the resonance depends on the phenomenon. For a meteotsunami, the duration of the resonance corresponds to the time the meteorological disturbance  takes to reach the coast (see \cite{MVR_Meteotsunamis}). However, for a landslide tsunami, the duration of the resonance corresponds to the duration of the landslide (which depends on the size of the slope, see \cite{Levin_tsunamis} or \cite{Energy_landslide_tsunami}). If the landslide is offshore, it is unreasonable to assume that the duration of the landslide is the time the water waves take to reach the coast. A variation of the pressure of $1$ hPa creates a water wave of $1$ cm whereas a moving bottom of $1$ cm tends to create a water wave of $1$ cm. Therefore we assume in the following that $a_{\text{bott,m}} = a$ (and hence $\beta \lambda = \epsilon$). However, it is important to notice that even if for storms, a variation of the pressure of $100$ hPa is very huge, it is quite ordinary that a submarine landslide have a thickness of $1$ m. Typically, a storm makes a variation of few Hpa, and the thickness of a submarine landslide is few dm (we refer to \cite{Levin_tsunamis}). In this part, we only study the propagation of such phenomena. Therefore, we take $d=1$. In the following, we give three linear asymptotic models of the water waves equations and we give examples of pressures and moving bottoms that create a resonance. The pressure at the surface $P$ and the moving bottom $b_{m}$ move from the left to the right. We consider that the system is initially at rest. We start this part by giving an asymptotic expansion with respect to $\mu$ and $\max(\epsilon,\beta)$ of $\GNN$.

\begin{prop}\label{approximation_GNN_delta}
Let $t_{0} > \frac{d}{2}$, $\zeta \text{ and } b \in H^{t_{0}+2}(\RD)$ such that Condition \eqref{nonvanishing} is satisfied. We suppose that the parameters $\epsilon$, $\beta$ and $\mu$ satisfy \eqref{parameters_constraints}. Then, for all $B \in H^{s-\frac{1}{2}}(\RD)$ with $0 \leq s \leq t_{0}+\frac{3}{2}$, we have

\begin{equation*}
\left\lvert \hspace{-0.04cm} G^{N \! N}_{\mu} \hspace{-0.03cm} [\hspace{-0.01cm} \epsilon \zeta \hspace{-0.01cm}, \hspace{-0.02cm} \beta b \hspace{-0.01cm}] \hspace{-0.03cm} (\hspace{-0.03cm} B \hspace{-0.03cm}) \hspace{-0.06cm} - \hspace{-0.06cm} G^{N \! N}_{\mu} \hspace{-0.03cm} [\hspace{-0.01cm} 0 \hspace{-0.01cm}, \hspace{-0.03cm} 0 \hspace{-0.01cm}] \hspace{-0.03cm} (\hspace{-0.03cm} B \hspace{-0.03cm}) \hspace{-0.04cm} \right\rvert_{H^{s-\frac{1}{2}}} \hspace{-0.09cm} \leq \hspace{-0.08cm} M_{0} |(\epsilon \zeta, \beta b)|_{H^{t_{0}+2}} \left\lvert \hspace{-0.03cm} B \hspace{-0.03cm} \right\rvert_{H^{s-\frac{1}{2}}} \hspace{-0.1cm}
\end{equation*} 

\noindent and

\begin{equation*}
\left\lvert \hspace{-0.04cm} G^{N \! N}_{\mu} \hspace{-0.03cm} [\hspace{-0.01cm} 0 \hspace{-0.01cm}, \hspace{-0.02cm} 0 \hspace{-0.01cm}] \hspace{-0.03cm} (\hspace{-0.03cm} B \hspace{-0.03cm}) \hspace{-0.06cm} - \hspace{-0.06cm} B \hspace{-0.04cm} \right\rvert_{H^{s-\frac{1}{2}}} \hspace{-0.09cm} \leq C \mu \left\lvert \hspace{-0.03cm} B \hspace{-0.03cm} \right\rvert_{H^{s+\frac{3}{2}}} \hspace{-0.1cm}.
\end{equation*} 

\end{prop}

\begin{proof}

\noindent The first inequality follows from Proposition \ref{controls_dGNN} and the second from Remark \ref{operator_in_0}.

\end{proof}

\begin{remark}\label{G_approx_Laplacian}
\noindent In the same way and under the assumptions of the previous proposition, we can prove that (see Proposition 3.28 in \cite{Lannes_ww}), for $0 \leq s \leq t_{0} + \frac{3}{2}$, 

\begin{equation*}
\left\lvert \hspace{-0.04cm} G_{\mu} \hspace{-0.02cm} [\hspace{-0.01cm} \epsilon \zeta \hspace{-0.01cm}, \hspace{-0.02cm} \beta b \hspace{-0.01cm}] \hspace{-0.03cm} (\hspace{-0.03cm} \psi \hspace{-0.03cm}) \hspace{-0.06cm} - \hspace{-0.06cm} G_{\mu} \hspace{-0.02cm} [\hspace{-0.01cm} 0 \hspace{-0.01cm}, \hspace{-0.03cm} 0 \hspace{-0.01cm}] \hspace{-0.03cm} (\hspace{-0.03cm} \psi \hspace{-0.03cm}) \hspace{-0.04cm} \right\rvert_{H^{s-\frac{1}{2}}} \hspace{-0.09cm} \leq \mu M_{0} |(\epsilon \zeta, \beta b)|_{H^{t_{0}+2}} \hspace{-0.09cm} \left\lvert \hspace{-0.03cm} \mathfrak{P} \psi \hspace{-0.03cm} \right\rvert_{H^{s+\frac{1}{2}}} \hspace{-0.1cm}
\end{equation*}

\noindent and 

\begin{equation*}
\left\lvert \hspace{-0.04cm} \frac{1}{\mu} G_{\mu} \hspace{-0.03cm} [\hspace{-0.01cm} 0 \hspace{-0.01cm}, \hspace{-0.02cm} 0 \hspace{-0.01cm}] \hspace{-0.03cm} (\hspace{-0.03cm} \psi \hspace{-0.03cm}) \hspace{-0.06cm} + \hspace{-0.06cm} \Delta \psi \right\rvert_{H^{s-\frac{1}{2}}} \hspace{-0.09cm} \leq \mu C \left\lvert \hspace{-0.03cm} \nabla \psi \hspace{-0.03cm} \right\rvert_{H^{s+\frac{5}{2}}} \hspace{-0.1cm}.
\end{equation*} 
\end{remark}

\noindent We denote by $\overline{V}$ the vertically averaged horizontal component,

\begin{equation}\label{med_V_def}
\overline{V} = \overline{V}[\epsilon \zeta, \beta b] \left(\psi, \partial_{t} b \right) =  \frac{1}{1+\epsilon \zeta - \beta b} \int_{-1+\beta b}^{\epsilon \zeta} \nabla_{\! X} \hspace{-0.05cm} \left(\Phi[\epsilon \zeta, \beta b] \left(\psi, \partial_{t} b \right)(\cdot,z) \right) dz,
\end{equation}

\noindent where $\Phi = \Phi[\epsilon \zeta, \beta b] \left(\psi, \partial_{t} b \right)$ satisfies 

\begin{equation*}
\left\{
\begin{aligned}
&\Delta^{\mu}_{\! X,z} \Phi = 0 \text{, } -1 + \beta b \leq z \leq \epsilon \zeta,\\
&\Phi_{\hspace{-0.02cm} |z=\epsilon \zeta} = \psi \text{ , } \sqrt{1+\beta^{2} |\nabla b|^{2}}\partial_{\textbf{n}} \Phi_{\hspace{-0.02cm} |z=-1+\beta b} = \mu \partial_{t} b.
\end{aligned}
\right.
\end{equation*}

\noindent The following Proposition is Remark 3.36 and a small adaptation of Proposition 3.37 and Lemma 5.4 in \cite{Lannes_ww} (see also Subsection A.5.5 in \cite{Lannes_ww}).

\begin{prop}\label{shallow water approx}

\noindent Let $T > 0$, $t_{0} > \frac{d}{2}$, $0 \leq s \leq t_{0}$ and $\zeta, b \in W^{1,\infty} \left([0,T]; \Hzeta \right)$ such that Condition \eqref{nonvanishing} is satisfied on $[0,T]$. We suppose that the parameters $\epsilon$, $\beta$ and $\mu$ satisfy \eqref{parameters_constraints}. We also assume that $\psi \in W^{1,\infty} \left([0,T]; \Hdot^{s+ 3}(\RD) \right)$. Then, 

\begin{equation*}
 \G(\psi) + \mu \GNN(\partial_{t} b) = - \mu \nabla \cdot \left( (1+\epsilon \zeta - \beta b) \overline{V} \right) + \mu \partial_{t} b,
\end{equation*}

\noindent and 

\begin{equation*}
\left\{
\begin{aligned}
&\hspace{-0.2cm} \left\lvert \overline{V} - \nabla \psi \right\rvert_{H^{s}} \leq \mu C \hspace{-0.07cm} \left( \hspace{-0.07cm} \frac{1}{h_{\min}}, \hspace{-0.04cm} \mu_{\max}, \hspace{-0.05cm} \epsilon |\hspace{-0.01cm} \zeta \hspace{-0.01cm}|_{H^{t_{0} + 2}}, \hspace{-0.05cm} \beta |\hspace{-0.01cm} b \hspace{-0.01cm}|_{L^{\infty}_{t} \hspace{-0.02cm} H^{t_{0}+2}_{\! X}} \hspace{-0.1cm} \right) \max \left( \left\lvert \nabla \psi \right\rvert_{H^{s+2}}, \left\lvert \partial_{t} b \right\rvert_{L^{\infty}_{t} \hspace{-0.02cm} H^{s+1}_{\! X}} \right),\\
&\hspace{-0.2cm} \left\lvert \hspace{-0.03cm} \partial_{t} \overline{V} \hspace{-0.08cm} - \hspace{-0.08cm} \nabla \partial_{t} \psi \hspace{-0.03cm} \right\rvert_{H^{s}} \hspace{-0.12cm} \leq \hspace{-0.08cm} \mu C \hspace{-0.08cm} \left( \hspace{-0.08cm} \frac{1}{h_{\min}}, \hspace{-0.04cm} \mu_{\max}, \hspace{-0.07cm} |\hspace{-0.02cm} \zeta \hspace{-0.02cm}|_{H^{t_{0} + 2}}, \hspace{-0.07cm} \left \lvert \hspace{-0.02cm} \partial_{t} \zeta \hspace{-0.02cm} \right\rvert_{H^{t_{0}+2}} \hspace{-0.06cm}, \hspace{-0.05cm} |\hspace{-0.01cm} b \hspace{-0.01cm}|_{W^{2,\infty}_{t} \hspace{-0.02cm} H^{t_{0}+2}_{\! X}}, \hspace{-0.02cm} \left\lvert \hspace{-0.02cm} \nabla \psi \hspace{-0.02cm} \right\rvert_{H^{s+2}} \hspace{-0.05cm} , \hspace{-0.02cm} \left\lvert \hspace{-0.02cm} \partial_{t} \nabla \psi \hspace{-0.02cm} \right\rvert_{H^{s+2}} \hspace{-0.1cm} \right) \hspace{-0.14cm}.
\end{aligned}
\right.
\end{equation*}
\end{prop}

\noindent In this part, we will consider symmetrizable linear hyperbolic systems of the first order. We refer to \cite{Benzoni_Serre} for more details about the wellposedness. In the following, we will only give the energy associated to the symmetrization.

\subsection{A shallow water model when $\beta$ is small}\label{A shallow water model with small topography variations}

\subsubsection{Linear asymptotic}

\noindent We consider the case that $\epsilon$, $\beta$, $\mu$ are small. Physically, this means that we consider small amplitudes for the surface and the bottom (compared to the mean depth) and waves with large wavelengths (compared to the mean depth). The asymptotic regime (in the sense of Definition 4.19 in \cite{Lannes_ww}) is

\begin{equation}
\mathcal{A}_{LW} = \left\lbrace (\epsilon,\beta,\lambda,\mu) \text{, } 0 <  \mu \text{, } \epsilon \text{, } \beta \leq \delta_{0} \text{, }  \beta \lambda  = \epsilon \right\rbrace,
\end{equation}

\noindent with $\delta_{0} \ll 1$.

\begin{prop}\label{approx_ww}
Let $t_{0} \hspace{-0.05cm} > \hspace{-0.05cm} \frac{d}{2}$, $N \geq \max(1,t_{0}) + 3$, $U^{0} \in E_{0}^{N} \hspace{-0.05cm}$, $P \hspace{-0.05cm} \in \hspace{-0.05cm} W^{1,\infty}(\mathbb{R}^{+}; \hspace{-0.05cm} \Hdot^{N+1}(\RD))$ and $b \in W^{3,\infty}(\R^{+}; H^{N}(\RD))$. We suppose \eqref{nonvanishing} and \eqref{rt_constraints} are satisfied initially. Then, there exists $T > 0$, such that for all  $ (\epsilon,\beta,\lambda,\mu) \in \mathcal{A}_{LW}$, there exists a solution $U = (\zeta,\psi) \in E^{N}_{\frac{T}{\sqrt{\delta_{0}}}}$ to the water waves equations with initial data $U^{0}$ and this solution is unique. Furthermore, for all $\alpha \in \left[0, \frac{1}{3} \right)$,

\begin{equation*}
\left\lvert \zeta - \widetilde{\zeta} \right\rvert_{L^{\infty}\left(\left[0, \frac{T}{\delta_{0}^{\alpha}} \right]; H^{N-4}(\RD) \right)} + \left\lvert \nabla \psi - \nabla \widetilde{\psi} \right\rvert_{L^{\infty}\left(\left[0, \frac{T}{\delta_{0}^{\alpha}} \right]; H^{N-2}(\RD) \right)} \leq T \delta_{0}^{1-3\alpha} \widetilde{C},
\end{equation*}

\noindent where 

\begin{equation*}
\widetilde{C} = C \left(\E \left(U^{0} \right), \frac{1}{h_{\min}}, \frac{1}{\mathfrak{a}_{\min}}, |b|_{W^{3,\infty}_{t} H^{N}_{\! X}}, \left\lvert \nabla P \right\rvert_{W^{1,\infty}_{t} H^{N}_{\! X}} \right),
\end{equation*}

\noindent and with, $(\widetilde{\zeta},\widetilde{\psi})$ solution of the waves equation

\begin{equation}\label{waves_equation}
  \left\{
  \begin{aligned}
   &\partial_{t} \widetilde{\zeta} + \Delta_{\! X} \widetilde{\psi} = \partial_{t} b,\\
   &\partial_{t} \widetilde{\psi} +  \widetilde{\zeta}  = - P,\\
  \end{aligned}  
  \right.
\end{equation}

\noindent with initial data $U^{0}$.
\end{prop}

\begin{proof}
\noindent First, the system \eqref{waves_equation} is wellposed since it can be symmetrized thanks to the energy

\begin{equation*}
\mathcal{E}(t) = \left\lvert \widetilde{\zeta} \right\rvert_{L^{2}}^{2} + \left\lvert \nabla \widetilde{\psi} \right\rvert_{L^{2}}^{2}.
\end{equation*}

\noindent Using Theorem \ref{long_time_existence} we get a uniform time of existence $\frac{T}{\sqrt{\delta_{0}}} > 0$ for the water waves equation and for all parameters in $\mathcal{A}_{LW}$. Then, using Proposition \ref{approximation_GNN_delta}, Remark \ref{G_approx_Laplacian}, Proposition \ref{controls_w} and \ref{P_estimates} and standard controls we get that

\begin{equation}\label{approximation_waves_equation_water_waves}
  \left\{
  \begin{aligned}
   &\partial_{t} \zeta + \Delta_{\! X} \psi = \partial_{t} b + R_{1},\\
   &\partial_{t} \psi +  \zeta  = - P + R_{2} \text{,}
  \end{aligned}
  \right.
\end{equation} 

\noindent with

\begin{equation*}
  \left\{
  \begin{aligned}
&\left\lvert R_{1} \right\rvert_{H^{N-4}} \leq C  \hspace{-0.1cm} \left(\epsilon |\zeta|_{H^{N}}, |b|_{L^{\infty}_{t} H^{N}_{\! X}} \right)\hspace{-0.1cm} \left( |(\epsilon \zeta, \beta b)|_{H^{N}} + \mu \right) \max \left( |\mathfrak{P} \psi|_{H^{N-\frac{1}{2}}}, |\partial_{t} b|_{H^{N}} \right),\\
&\left\lvert R_{2} \right\rvert_{H^{N-1}} \leq \epsilon C \hspace{-0.1cm} \left(\epsilon |\zeta|_{H^{N}}, |b|_{L^{\infty}_{t} H^{N}_{\! X}} \right) \hspace{-0.1cm} \max \left(|\mathfrak{P} \psi|^{2}_{H^{N-\frac{1}{2}}}, |\partial_{t} b|^{2}_{H^{N}} \right) \hspace{-0.1cm}.
\end{aligned}
\right.
\end{equation*}

\noindent If we denote $\zeta_{1} = \zeta - \widetilde{\zeta}$ and $\psi_{1} = \psi - \widetilde{\psi}$, we see that $(\zeta_{1},\psi_{1})$ satisfies

\begin{equation*}
  \left\{
  \begin{aligned}
   &\partial_{t} \zeta_{1} + \Delta_{\! X} \psi_{1} =  R_{1},\\
   &\partial_{t} \psi_{1} +  \zeta_{1}  =  R_{2}.
  \end{aligned}
  \right.
\end{equation*} 

\noindent Differentiating the energy

\begin{equation*}
\mathcal{E}^{N}(t) = \frac{1}{2} \left\lvert \zeta_{1} \right\rvert_{H^{N-4}}^{2} + \frac{1}{2} \left\lvert \nabla \psi_{1} \right\rvert_{H^{N-2}}^{2},
\end{equation*}

\noindent we get the estimate thanks to Proposition \ref{psi_controls} and energy estimate in Theorem \ref{long_time_existence}.
\end{proof}

\noindent This model is well-known in the physics literature (see \cite{Proudman_resonnance}).

\subsubsection{Resonance in shallow waters when $\beta$ is small}

\noindent We consider the equation \eqref{waves_equation} for $d=1$. We transform it in order to have a unique equation for $h :=  \widetilde{\zeta} - b$, 

\begin{equation}\label{waves_h}
  \left\{
  \begin{aligned}
  &\partial_{t}^{2} h - \partial^{2}_{X} h = \partial^{2}_{X} \left(P + b \right),\\
  &h_{|t=0} = - b(0,.),\\
  &\partial_{t} h_{|t=0} = 0.\\
  \end{aligned}
  \right.
\end{equation}

\noindent We denote $f(t,X) := \left( P + b \right)(t,X)$, which represents a disturbance. We want to understand the resonance for landslide and meteo tsunamis. In both cases, it is a linear respond, in the shallow water case, of a body of water due to a moving pressure or a moving bottom, when the speed of the storm or the landslide is close to the typical wave celerity (here $1$). We can compute $h$ thanks to the d'Alembert's formula

\begin{align*}
h(t,X) = &\underbrace{-\frac{1}{2} \left( b(0,X-t) + b(0,X+t) \right) }_{h_{T}(t,X)} + \underbrace{\frac{1}{2} \int_{0}^{t} \partial_{X} f(\tau, X+t-\tau) d\tau}_{:=h_{L}(t,X)}\\
&- \underbrace{\frac{1}{2} \int_{0}^{t} \partial_{X} f(\tau, X-t+\tau) d\tau}_{:=h_{R}(t,X)}.
\end{align*}

\noindent We are interesting in disturbances $f$ moving from the left to the right (propagation to a coast). Therefore, we study only $h_{R}$. The following Proposition shows that a disturbance moving with a speed equal to $1$ makes appear a resonance.

\begin{prop}
\noindent Let $f \in L^{\infty}(\R^{+}; H^{1}(\RD))$ and $\partial_{X} f \in L^{\infty}_{t \times X}(\R \times \RD)$. Then, for all $X \in \R$, $t > 0$,

\begin{equation*}
\left\lvert h_{R}(t,X) \right\rvert \leq \frac{t}{2} \left\lvert \partial_{X} f\right\rvert_{\infty}.
\end{equation*}

\noindent Furthermore, if $f(t,X) = f_{0}(X-t)$, $f_{0} \in H^{1}(\RD)$ and $\left\lvert f^{\prime}_{0}(X_{0}-t_{0}) \right\rvert = \left\lvert f' \right\rvert_{\infty}$ the equality holds for $(t_{0}, X_{0})$. If $f(t,X) = f_{0}(X-Ut)$ with $f_{0} \in H^{1}(\RD)$ and $U \neq 1$,

\begin{equation*}
\left\lvert h_{R} \right\rvert_{\infty} \leq \min\left( \frac{\left\lvert f_{0} \right\rvert_{\infty}}{|1-U|}, \frac{t}{2} \left\lvert f^{\prime}_{0} \right\rvert_{\infty} \right).
\end{equation*}
\end{prop}

\begin{proof}
\noindent If $f(t,X) = f_{0}(X-Ut)$, 

\begin{equation*}
h_{R}(t,X) = - \frac{1}{2} \int_{0}^{t} f^{\prime}_{0}(X-t+(1-U)\tau) d\tau,
\end{equation*}

\noindent and the result follows.
\end{proof}

\noindent This Proposition corresponds to the historical work of J. Proudman (\cite{Proudman_resonnance}). We rediscover the fact that the resonance occurs if the speed of the disturbance is $1$. For a disturbance with a speed different from $1$, we notice a saturation effect (also pointed out in \cite{Energy_landslide_tsunami}). The graph in Figure \ref{maxi_W}, gives the typical evolution of $\left\lvert h(t,\cdot) \right\rvert_{\infty}$ with respect to the time $t$ for different values of the speed. We can see the saturation effect. We compute $h$ with a finite difference method and we take $f(t,X)=e^{-\frac{1}{2}(X-Ut)^{2}}$. We see also that the landslide resonance and the Proudman resonance have the same effects. There are however two important differences that we exposed in the introduction of this part. The first one is the duration of the resonance. A landslide is quicker than a meteorological effect. The second one, is the fact that the typical size of the landslide (few dm) is bigger than the size of a storm (few hPa). For instance, for a moving storm which creates a variation of the pressure of 3 hPa during $15 t_{0}$, the final wave can reach a amplitude of $13$ cm (it is for example the case of the meteotsunami in Nagasaki in 1979, see \cite{MVR_Meteotsunamis}). Conversely, an offshore landslide with a thickness of $1$ m that lasts $t_{0}$, can create a wave of $50$ cm (which corresponds to the results in \cite{Energy_landslide_tsunami}). Therefore, we see that the principal difference between an offshore landslide and a moving storm is the size. 

\begin{center}
\begin{figure}[!h]
   \includegraphics[scale=0.2]{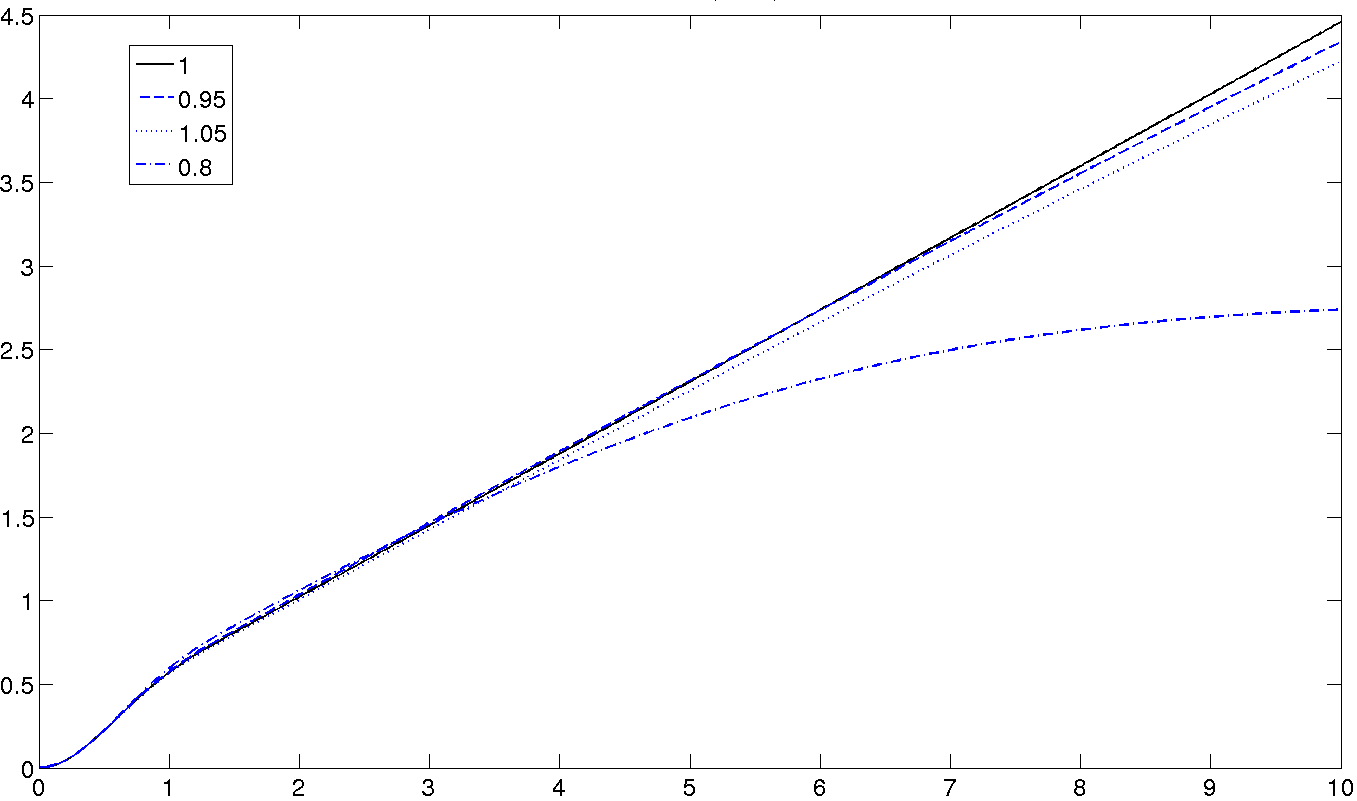}
   \caption{Evolution of the maximum of $h$, solution of equation \eqref{waves_h}, with different values of the speed $U$.}
   \label{maxi_W}
\end{figure}
\end{center}

\subsection{A shallow water model when $\beta$ is large}\label{A shallow water model with a variable topography}

\subsubsection{Linear asymptotic}

\noindent In this case, we suppose only that $\epsilon$ and $\mu$ are small. We recall that $\beta b(t,X) = \beta b_{0}(X) + \beta \lambda b_{m}(t,X)$. Then, we assume also that $1 - b_{0} \geq h_{\min} > 0$. In the following, we denote $h_{0} := 1 - \beta b_{0}$. The asymptotic regime is

\begin{equation}
\mathcal{A}_{LVW} = \left\lbrace (\epsilon,\beta,\lambda,\mu) \text{, } 0 < \epsilon \text{, }  \mu \leq \delta_{0} \text{, } 0 < \beta \leq 1 \text{, } \beta \lambda = \epsilon  \right\rbrace,
\end{equation}

\noindent with $\delta_{0} \ll 1$. We can now give a asymptotic model.

\begin{prop}
Let $t_{0} \hspace{-0.05cm} > \hspace{-0.05cm} \frac{d}{2}$, $N \geq \max(1,t_{0}) + 4$, $b \in W^{3,\infty}(\R^{+}; H^{N}(\RD))$, $U^{0} = \left(\zeta_{0}, \psi_{0} \right) \in E_{0}^{N} \hspace{-0.05cm}$, and $P \hspace{-0.05cm} \in \hspace{-0.05cm} W^{1,\infty}(\mathbb{R}^{+}; \hspace{-0.05cm} \Hdot^{N+1}(\RD))$. We suppose that \eqref{nonvanishing} and \eqref{rt_constraints} are satisfied initially. We suppose also that $b_{0} \in H^{N}(\RD)$ and that $h_{0} = 1-\beta b_{0} \geq h_{\min}$. Then, there exists $T > 0$, such that for all  $(\epsilon,\beta,\lambda,\mu) \in \mathcal{A}_{LVW}$, there exists a unique solution $U =(\zeta,\psi) \in E^{N}_{T}$ to the water waves equations with initial data $U^{0}$. Furthermore, for $\overline{V}$ as in \eqref{med_V_def},

\begin{equation*}
\left\lvert \zeta - \zeta_{1} \right\rvert_{L^{\infty}([0,T]; H^{N-4}(\RD))} + \left\lvert \overline{V} - \overline{V}_{1} \right\rvert_{L^{\infty}([0,T]; H^{N-4}(\RD))} \leq T \delta_{0} \widetilde{C},
\end{equation*}

\noindent where

\begin{equation*}
\widetilde{C} = C \hspace{-0.1cm} \left(\mathcal{E}^{N} \left(U^{0} \right), \frac{1}{h_{\min}}, \frac{1}{\mathfrak{a}_{\min}}, |b|_{W^{3,\infty}_{t} H^{N}_{\! X}}, \left\lvert \nabla P \right\rvert_{W^{1,\infty}_{t} H^{N}_{\! X}} \right) \hspace{-0.1cm} ,
\end{equation*}

\noindent and $(\zeta_{1},\overline{V}_{1})$ solution of the waves equation

\begin{equation}\label{variable_waves_equation}
  \left\{
  \begin{aligned}
   &\partial_{t} \zeta_{1} + \nabla \cdot \left( h_{0} \overline{V}_{1} \right) = \partial_{t} b_{m},\\
   &\partial_{t} \overline{V}_{1} + \nabla \zeta_{1}  = - \nabla P,\\
   &\left(\zeta_{1}\right)_{|t=0} = \zeta_{0} \text{, } \left( V_{1} \right)_{|t=0} = \overline{V}\left[ \epsilon \zeta_{0}, \beta b_{|t=0} \right] \left(\psi_{0}, \left( \partial_{t} b \right)_{|t=0} \right). 
  \end{aligned}  
  \right.
\end{equation} 

\end{prop}

\begin{proof}
\noindent The system \eqref{variable_waves_equation} is wellposed since it can be symmetrized thanks to the energy

\begin{equation*}
\mathcal{E}(t) = \frac{1}{2} \left\lvert \zeta_{1} \right\rvert_{L^{2}}^{2} + \frac{1}{2} \left( h_{0} \overline{V}_{1}, \overline{V}_{1} \right)_{L^{2}}.
\end{equation*}

\noindent For the inequality, we proceed as in Proposition \ref{approx_ww}, differentiating the energy 

\begin{equation*}
\mathcal{E}^{N}(t) = \frac{1}{2} \left\lvert \zeta_{2} \right\rvert_{H^{N-4}}^{2} + \frac{1}{2} \left( h_{0} \Lambda^{N-4} \overline{V}_{2}, \Lambda^{N-4} \overline{V}_{2} \right)_{L^{2}},
\end{equation*}

\noindent with $\zeta_{2} = \zeta - \zeta_{1}$ and $\overline{V}_{2} = \overline{V} - \overline{V}_{1}$. Using Gronwall's Lemma, Proposition \ref{shallow water approx} and standard controls, we get result.
\end{proof}

\noindent This model is well-known in the physics literature to investigate the landslide tsunami phenomenon (see \cite{Energy_landslide_tsunami}).

\subsubsection{Amplification in shallow waters when $\beta$ is large}

\noindent In this part, $d=1$ and we suppose that $P = 0$. The same study can be done for a non constant pressure. For the sake of simplicity, we assume also that initially the velocity of the landslide is zero and hence that $\left(\partial_{t} b_{m} \right)_{|t=0} = 0$ (the bottom does not move at the beginning). We transform the system \eqref{variable_waves_equation} in order to get an equation for $\zeta_{1}$ only. We obtain that $\zeta_{1}$ satisfies

\begin{equation}\label{variable_waves_equation_zeta}
\partial^{2}_{t} \zeta_{1} - \partial_{X} \left( h_{0} \partial_{X} \zeta_{1} \right) = \partial^{2}_{t} b_{m},
\end{equation} 

\noindent with $\left(\zeta_{1}\right)_{|t=0} = 0$ and $\left(\partial_{t} \zeta_{1} \right)_{|t=0} = 0$. We wonder now if we can catch an elevation of the sea level with this asymptotic model. Therefore, we are looking for solutions of the form

\begin{equation}\label{amplfification_form}
\zeta_{2}(t,X)= t \zeta_{3}(t,X).
\end{equation}

\noindent The following proposition gives example of such solutions for bounded moving bottoms (with finite energy).

\begin{prop}
Suppose that $h_{0} \geq h_{\min} > 0$ with $h_{0} \in H^{1}(\R)$. Let $\left(\zeta_{3},\overline{V}_{3} \right)$ be a solution of 

\begin{equation*}
  \left\{
  \begin{aligned}
   &\partial_{t} \zeta_{3} + \partial_{X} \left(h_{0} \overline{V}_{3} \right) = 0,\\
   &\partial_{t} \overline{V}_{3} + \partial_{X} \zeta_{3}  = 0,\\
  \end{aligned}  
  \right.
\end{equation*}

\noindent with $\left( \zeta_{3}, \overline{V}_{3} \right)_{|t=0} = \left(0, f^{\prime} \right)$ with $f \in H^{1}(\R)$. Then, $\zeta_{1}(t,X)= t \zeta_{3}(t,X)$ is a non trivial solution of \eqref{variable_waves_equation_zeta} with 

\begin{equation}\label{bm}
b_{m}(t,X)= 2 \int_{0}^{t} \zeta_{3}(s,X) ds,
\end{equation}

\noindent and $b_{m}(t,\cdot)$ is bounded in $L^{2}(\RD)$ and in $L^{\infty}(\RD)$ uniformly with respect to $t$

\begin{equation*}
\left\vert b_{m}(t,\cdot) \right|_{L^{2}} + \left\vert b_{m}(t,\cdot) \right|_{L^{\infty}} \leq C,
\end{equation*}

\noindent where $C$ is independent on $t$.
\end{prop}

\begin{proof}
\noindent Plugging the expression of $\zeta$ and $b_{m}$ in \eqref{variable_waves_equation_zeta}, we get the first result. We have to show that $\zeta_{3} \in L^{1}(\R^{+}; L^{2}(\RD))$. Consider the linear hyperbolic equation

\begin{equation*}
  \left\{
  \begin{aligned}
   &\partial_{t} \eta + \partial_{X} \left(h_{0} W \right) = 0,\\
   &\partial_{t} W + \partial_{X} \eta  = 0,\\
  \end{aligned}  
  \right.
\end{equation*}

\noindent with $\left( \eta, W \right)_{|t=0} = \left(-f,0\right)$. This system has a unique solution $\left( \eta, W \right) \in \mathcal{C}^{0}(\R; H^{1}(\R))$. Furthermore, $\left( \partial_{t} \eta, \partial_{t} W \right) \in \mathcal{C}^{0}(\R; L^{2}(\R))$, and $\left( \partial_{t} \eta, \partial_{t} W \right)$ satisfies the same linear hyperbolic system as $\left( \zeta_{3}, \overline{V}_{3} \right)$. By uniqueness, $\zeta_{3} = \partial_{t} \eta$ and 

\begin{equation*}
b_{m}(t,X)= 2 \eta(t,X) + 2 f(X).
\end{equation*}

\noindent Since, for all $t$,

\begin{equation*}
\int_{\R} \eta(t,X)^{2} + h_{0}(X) W(t,X)^{2} dX = \int_{\R} f(X)^{2} dX,
\end{equation*}

\noindent and $h_{0} \geq h_{\min} > 0$, we get the control of $\left\lvert b_{m}(t,\cdot) \right\rvert_{L^{2}}$. Finally, $\eta$ satisfies the waves equation 

\begin{equation*}
\partial_{t}^{2} \eta - \partial_{X} \left( h_{0} \partial_{X} \eta \right) = 0,
\end{equation*}

\noindent with $\left(\eta,\partial_{t} \eta \right)_{|t=0} = (-f,0) \in H^{1}(\RD)$. Then,
for all $t$,

\begin{equation*}
\int_{\R} \left\lvert \partial_{t} \eta(t,X) \right\rvert^{2} + h_{0}(X) \left\lvert \partial_{X} \eta(t,X) \right\rvert^{2} dX = \int_{\R} h_{0}(X) f^{\prime}(X)^{2} dX.
\end{equation*}

\noindent Therefore, $\left\lvert \eta \right\rvert_{H^{1}}$ (and $\left\lvert \eta \right\rvert_{L^{\infty}}$ by Sobolev embedding) is controlled uniformly with respect to $t$.
\end{proof}

\noindent In the following, we compute numerically some solutions of Equations \eqref{variable_waves_equation_zeta} of the form \eqref{amplfification_form} with a finite difference method. We take $b_{0}(X)= - \tanh(X)$, $\beta = \frac{1}{2}$ and $\left(\partial_{t} \zeta_{3} \right)_{|t=0} = (4X^{2}-2) e^{-X^2}$. The figure \ref{maxivariables_1} is the evolution of the maximum of $\zeta_{1}$. The figure \ref{wavesolution_1} is the graph at different times of the waves and the landslide. The dashed curves are the landslide, the solid curves are the waves and the dotted curve is the slope. Therefore, we see that an important elevation of the sea level is possible even if we do not consider that the seabed is flat.

\begin{remark}
\noindent In order to simplify, we consider that the system is initially at rest. But our study can easily be extended to waves with non trivial initial data. In particular, we can study a wave amplified by a landslide. This is what happened during the tsunami in Fukushima in 2011 (see \cite{landslide_Fukushima}). We compute numerically this amplification. We consider a wave moving with a speed equal to $1$ (typical speed in the sea after nondimensionalization) that is amplified by a landslide. Figure \ref{amplifiedwaves} represents the evolution of the maximum of this wave. We can see an amplification.
\end{remark}

\begin{center}
\begin{figure}[!h]
   \includegraphics[scale=0.25]{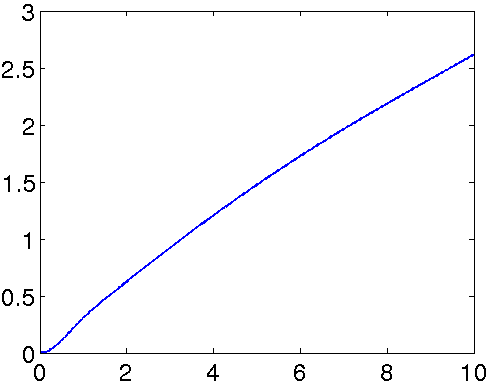}
   \caption{Evolution of the maximum of $\zeta_{1}$, solution of \eqref{variable_waves_equation_zeta}, for a non flat bottom $b_{0}$.}
   \label{maxivariables_1}
\end{figure}
\end{center}

\begin{center}
\begin{figure}[!h]
   \includegraphics[scale=0.15]{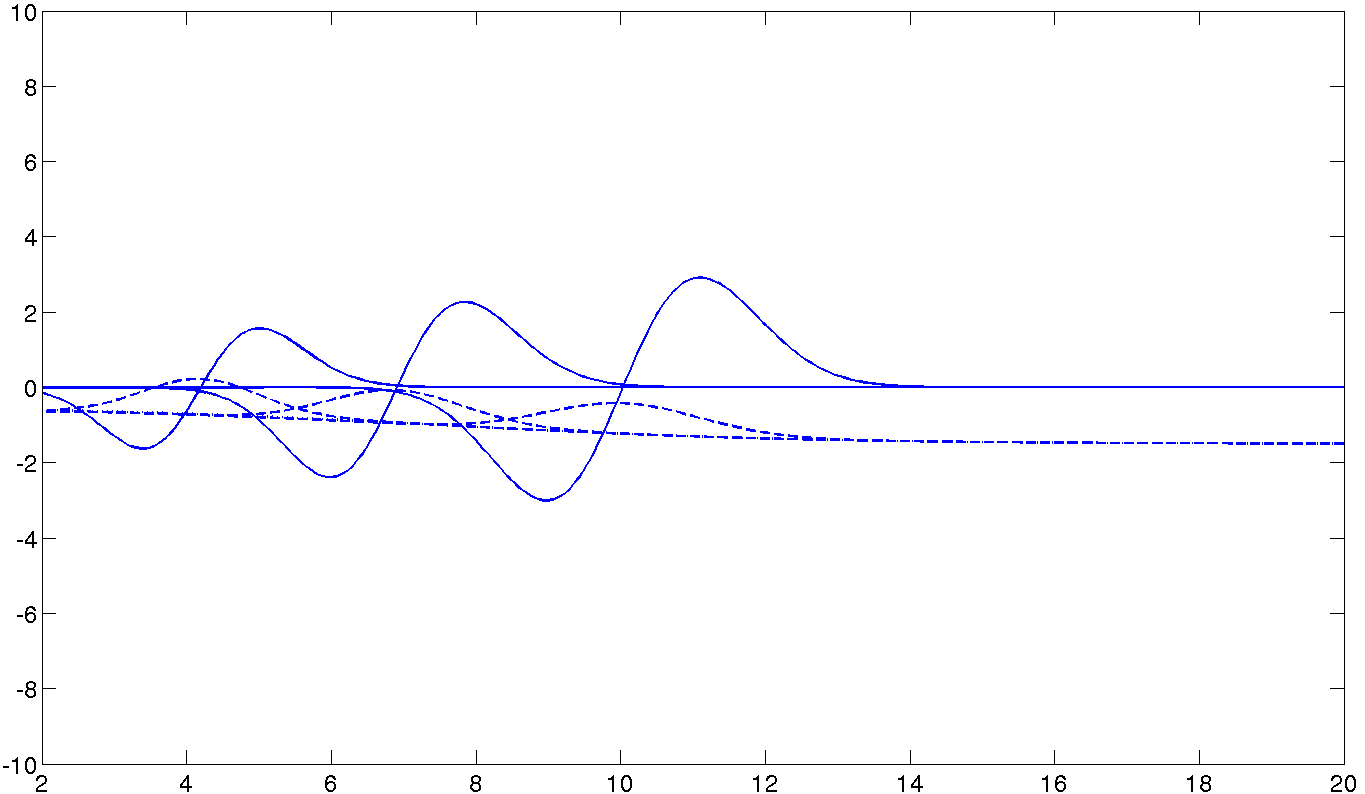}
   \caption{Evolution of the surface $\zeta_{1}$ (solid line), solution of \eqref{variable_waves_equation_zeta}, and the landslide $b_{m}$ (dashed line).}
   \label{wavesolution_1}
\end{figure}
\end{center}

\begin{center}
\begin{figure}[!h]
   \includegraphics[scale=0.15]{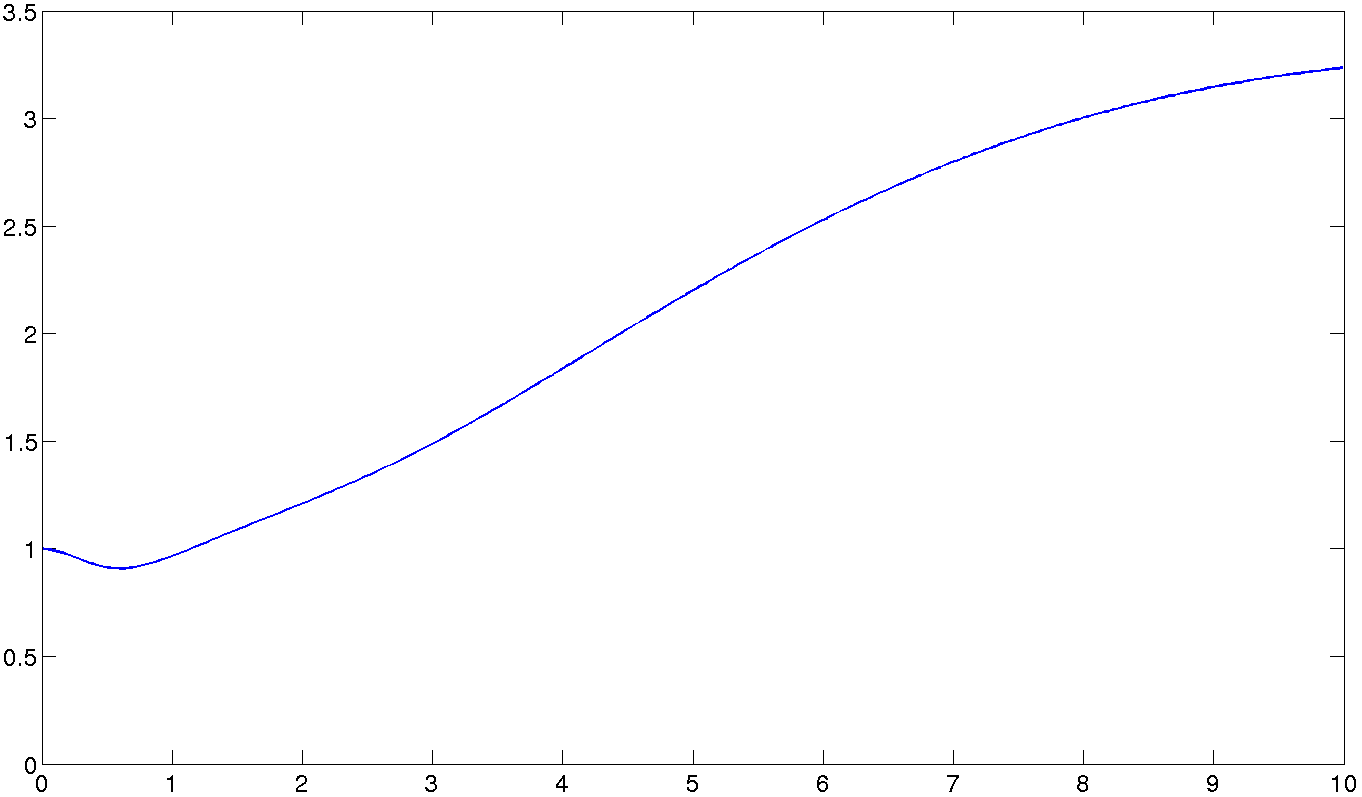}
   \caption{Evolution of the maximum of $h$, solution of \eqref{variable_waves_equation_zeta}, with non trivial initial data and with $b_{m}$ like in Figure \ref{wavesolution_1}.}
   \label{amplifiedwaves}
\end{figure}
\end{center}

\subsection{Linear asymptotic and resonance in intermediate depths}\label{Linear asymptotic and resonance in intermediate depths}

\noindent In this case, we consider only that $\epsilon$, $\beta$ are small. Physically, this means that we consider small amplitudes for the surface and the bottom (compared to the mean depth) and that the depth is not small compared to wavelength of the waves. In this part, we generalize the Proudman resonance in deeper waters. The asymptotic regime is

\begin{equation}
\mathcal{A}_{LWW} = \left\lbrace (\epsilon,\beta,\lambda, \mu) \text{, } 0 < \epsilon \text{, } \beta \leq \delta_{0} \text{, } \beta \lambda = \epsilon \text{ and } 0 < \mu \leq \mu_{\max} \right\rbrace,
\end{equation}

\noindent with $\delta_{0} \ll 1$ and $0 < \mu_{\max}$. Using the energy

\begin{equation*}
\mathcal{E}(t) = \frac{1}{2} \left\lvert \zeta \right\rvert_{L^{2}}^{2} + \frac{1}{2} \left( \frac{1}{\mu} G_{\mu}[0,0](\psi), \psi \right)_{L^{2}},
\end{equation*}

\noindent and proceeding as in Proposition \ref{approx_ww}  (we need also Proposition 3.12 in \cite{Lannes_ww}), we get a new asymptotic model.

\begin{prop}
Let $t_{0} \hspace{-0.05cm} > \hspace{-0.05cm} \frac{d}{2}$, $N \hspace{-0.05cm} \geq \hspace{-0.05cm} \max(1,t_{0}) +3$, $b \in W^{3,\infty}(\R^{+}; H^{N}(\RD))$, $U^{0} \hspace{-0.05cm} = \hspace{-0.05cm} \left(\zeta_{0}, \psi_{0} \right) \in E_{0}^{N}$ and $P \hspace{-0.05cm} \in \hspace{-0.05cm} W^{1,\infty}(\mathbb{R}^{+}; \hspace{-0.05cm} \Hdot^{N+1}(\RD))$. We suppose that \eqref{nonvanishing} and \eqref{rt_constraints} are satisfied initially. $\hspace{-0.3cm}$ Then, there exists $T > 0$, such that for all  $(\epsilon,\beta,\lambda, \mu) \hspace{-0.1cm} \in \hspace{-0.1cm} \mathcal{A}_{LWW}$, there exists a unique solution $U = (\zeta,\psi) \hspace{-0.05cm} \in E^{N}_{\frac{T}{\sqrt{\delta_{0}}}}$ to the water waves equations with initial data $U^{0}$. Furthermore, for all $\alpha \in \left[0, \frac{1}{3} \right)$,

\begin{equation*}
\left\lvert \zeta - \widetilde{\zeta} \right\rvert_{L^{\infty}\left(\left[0, \frac{T}{\delta^{\alpha}_{0}} \right]; H^{N-2}(\RD) \right)} + \left\lvert \frac{|D|}{\sqrt{1+|D|}} \left( \psi - \widetilde{\psi} \right) \right\rvert_{L^{\infty}\left(\left[0, \frac{T}{\delta^{\alpha}_{0}} \right]; H^{N-2}(\RD) \right)} \leq T \delta_{0}^{1-3\alpha} \widetilde{C},
\end{equation*}

\noindent where 

\begin{equation*}
\widetilde{C} = C \left(\E \left(U^{0} \right), \frac{1}{h_{\min}}, \frac{1}{\mathfrak{a}_{\min}}, \mu_{\max}, |b|_{W^{3,\infty}_{t} H^{N}_{\! X}}, \left\lvert \nabla P \right\rvert_{W^{1,\infty}_{t} H^{N}_{\! X}} \right),
\end{equation*}

\noindent where $\left(\widetilde{\zeta},\widetilde{\psi}\right)$ is a solution of the waves equation

\begin{equation}\label{linear_water_waves_equations}
  \left\{
  \begin{aligned}
   &\partial_{t} \widetilde{\zeta} - \frac{1}{\mu} G_{\mu}[0,0] (\widetilde{\psi}) = G_{\mu}^{N \! N}[0,0](\partial_{t} b),\\
   &\partial_{t} \widetilde{\psi} + \widetilde{\zeta}  = -P, \\
  \end{aligned}
  \right.
\end{equation} 

\noindent with initial data $U^{0}$. 

\end{prop}

\noindent The Proudman resonance is a phenomenon which occurs in shallow water regime. We wonder if there is also a resonance in deeper waters. In this part, we only work with a non constant pressure and hence $\partial_{t} b = 0$. The same study can be done for a moving bottom. We consider the equation \eqref{linear_water_waves_equations} for $d=1$. Since, the initial data does not affect the possible resonance, we suppose in the following that $U^{0} = 0$.  We transform the system \eqref{linear_water_waves_equations} in order to have a unique equation for $\widetilde{\zeta}$ (in the following we denote $\widetilde{\zeta}$ by $\zeta$ to simplify the notation) 

\begin{equation*}
\left\{
\begin{aligned}
&\partial_{t}^{2} \zeta + \frac{1}{\mu} G_{\mu}[0,0] (\zeta) = - \frac{1}{\mu} G_{\mu}[0,0] (P),\\
&\zeta_{|t=0} = 0 \text{, }  \partial_{t} \zeta_{|t=0} = 0.
\end{aligned}
\right.
\end{equation*}

\noindent We can solve explicitly the previous equation, we get that

\begin{align*}
\widehat{\zeta}(t,\xi) &= \underbrace{\frac{i}{2} \int_{0}^{t} \xi \sqrt{\frac{\tanh(\sqrt{\mu} |\xi|)}{\sqrt{\mu} |\xi|}} \widehat{P}(\tau,\xi) e^{i(t-\tau) \xi \sqrt{\frac{\tanh(\sqrt{\mu} |\xi|)}{\sqrt{\mu} |\xi|}}}  d\tau}_{:=\widehat{\zeta}_{L}(t,\xi)}\\
&- \underbrace{\frac{i}{2} \int_{0}^{t} \xi \sqrt{\frac{\tanh(\sqrt{\mu} |\xi|)}{\sqrt{\mu} |\xi|}} \widehat{P}(\tau,\xi) e^{i(\tau - t) \xi \sqrt{\frac{\tanh(\sqrt{\mu} |\xi|)}{\sqrt{\mu} |\xi|}}}  d\tau}_{:=\widehat{\zeta}_{R}(t,\xi)}.
\end{align*}

\noindent In order to find a resonant pressure, we suppose that $P$ has the form $e^{-it a(D)} P_{0}$, where $a$ is a real smooth odd function which is sublinear, there exists $C > 0$ such that $|a(\xi)| \leq C |\xi|$. We also suppose that the phase velocity of the disturbance is positive, $\frac{a(\xi)}{\xi} \geq 0$. $P_{0}$ is a smooth function in a Sobolev space with $\widehat{P_{0}}(0) \neq 0$. We denote $\omega(\xi) = \sqrt{\frac{\tanh(\xi)}{\xi}}$. A simple computation gives that

\begin{equation*}
| \zeta_{L}(t,\cdot)| \leq | \widehat{\zeta_{L}}(t,\cdot)|_{L^{1}} \leq \left\lvert \widehat{P_{0}} \right\rvert_{L^{1}}.
\end{equation*}

\noindent Furthermore, we have

\begin{align*}
|\widehat{\zeta_{R}}(t,\xi)| &= \frac{1}{2} \left\lvert \int_{0}^{t} \xi \omega(\sqrt{\mu} \xi) \widehat{P}_{0}(\xi) e^{i \tau \left(\xi \omega(\sqrt{\mu} \xi) - a(\xi) \right)}  d\tau \right\rvert \\
&\leq \frac{t}{2} \left\lvert \xi\omega(\sqrt{\mu} \xi) \widehat{P}_{0}(\xi) \right\rvert,
\end{align*}

\noindent with an equality if and only if $a(\xi) = \xi \omega(\sqrt{\mu} \xi)$. Hence, it is natural to consider that

\begin{equation}\label{resonantP}
\widehat{P}(t, \xi) = e^{-it \xi \omega(\sqrt{\mu} \xi)} P_{0}(\xi).
\end{equation}

\noindent A simple computation gives

\begin{equation}\label{zetaR}
\zeta_{R}(t,X) = - \frac{i t}{2} \int_{\R} \xi \omega(\sqrt{\mu} \xi) \widehat{P_{0}}(\xi) e^{-it \xi \omega(\sqrt{\mu} \xi)} e^{i X \xi} d\xi.
\end{equation}

\noindent We wonder now if a resonance occurs. We need a dispersion estimate for the linear water waves equation.

\begin{prop}
\noindent Let $f \in W^{1,1}(\R)$ such that $\widehat{f}(0)=0$. Then, 

\begin{equation*}
\left\lvert \int_{\R} e^{-it \xi \omega(\sqrt{\mu} \xi)} e^{i X \xi} \widehat{f}(\xi) d\xi \right\rvert \leq \frac{C}{\sqrt{t}} \left( \frac{1}{\sqrt{\mu}} \left\lvert \frac{1}{\sqrt{\left\lvert \xi \right\rvert}} \left(\widehat{f}\right)^{\prime} \right\rvert_{L^{1}(\R)} + \mu^{\frac{1}{8}} \left\lvert \left\lvert \xi \right\rvert^{\frac{3}{4}} \left(\widehat{f}\right)^{\prime} \right\rvert_{L^{1}(\R)} \right).
\end{equation*}
\end{prop}

\begin{proof}
\noindent We denote $I(t)$,

\begin{align*}
I(t) &:= \int_{\R} e^{-it \xi \omega( \sqrt{\mu} \xi)} e^{i X \xi} \widehat{f}(\xi) d\xi\\
&= \frac{1}{\sqrt{\mu}} \int_{\R} e^{-i \frac{t}{\sqrt{\mu}} \left( y \omega(y)  - \frac{X}{t} y \right)} \widehat{f} \left(\frac{y}{\sqrt{\mu}} \right) dy.
\end{align*}

\noindent We denote $\phi$,

\begin{equation*}
\phi(y)=y \omega(y) - \frac{X}{t} y,
\end{equation*}

\noindent and $y_{0}$ the unique minimum of $\phi^{\prime \prime}$. Figure \ref{profilphase2} represents $\phi^{\prime \prime}$ on $[0,+\infty[$.

\begin{center}
\begin{figure}[!h]
   \includegraphics[scale=0.35]{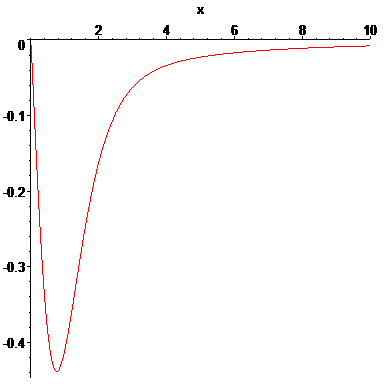}
   \caption{Profile of $\phi^{\prime \prime}$ .}
   \label{profilphase2}
\end{figure}
\end{center}

\noindent To estimate $I(t)$ we decompose $I(t)$ into four parts.

\begin{align*}
I_{1}(t) &= \frac{1}{\sqrt{\mu}} \int_{0}^{y_{0}} e^{-i \frac{t}{\sqrt{\mu}} \phi(y)} \widehat{f} \left(\frac{y}{\sqrt{\mu}} \right) dy\\
&= \frac{1}{\sqrt{\mu}} \int_{0}^{y_{0}} - \frac{d}{dy} \left( \int_{y}^{y_{0}} e^{-i \frac{t}{\sqrt{\mu}} \phi(z)} dz \right) \widehat{f} \left(\frac{y}{\sqrt{\mu}} \right) dy\\
&= \frac{1}{\mu} \int_{0}^{y_{0}} \int_{y}^{y_{0}} e^{-i \frac{t}{\sqrt{\mu}} \phi(z)} dz \left(\widehat{f}\right)^{\prime} \left(\frac{y}{\sqrt{\mu}} \right) dy.\\
\end{align*}

\noindent Then, using Van der Corput's Lemma (see \cite{Stein_harmonic}) and the fact that for $z \in [y,y_{0}]$,
\newline
$|\phi^{\prime \prime}(z) | \geq |\phi^{\prime \prime}(y)|$ and $|\phi^{\prime \prime}(z)| \geq Cz$,

\begin{align*}
|I_{1}(t)| &\leq \frac{C}{\mu^{\frac{3}{4}} \sqrt{t}} \int_{0}^{y_{0}} \left\lvert \frac{1}{\sqrt{y}} \left(\widehat{f}\right)^{\prime} \left(\frac{y}{\sqrt{\mu}} \right) \right\rvert  dy\\
&\leq \frac{C}{\sqrt{\mu} \sqrt{t}} \int_{0}^{+\infty} \left\lvert \frac{1}{\sqrt{\xi}} \left(\widehat{f}\right)^{\prime} \left(\xi \right) \right\rvert  d\xi.
\end{align*}

\noindent Furthermore, for $M > y_{0}$ large enough,

\begin{align*}
I_{2}(t) &= \frac{1}{\sqrt{\mu}} \int_{y_{0}}^{M} e^{-i \frac{t}{\sqrt{\mu}} \phi(y)} \widehat{f} \left(\frac{y}{\sqrt{\mu}} \right) dy\\
&= \frac{1}{\sqrt{\mu}} \int_{y_{0}}^{M} \frac{d}{dy} \left( \int_{y_{0}}^{y} e^{-i \frac{t}{\sqrt{\mu}} \phi(z)} dz \right) \widehat{f} \left(\frac{y}{\sqrt{\mu}} \right) dy\\
&= \int_{y_{0}}^{M} \hspace{-0.1cm} e^{-i \frac{t}{\sqrt{\mu}} \phi(z)} \frac{dz}{\sqrt{\mu}} \widehat{f} \left(\frac{M}{\sqrt{\mu}} \right) - \frac{1}{\mu} \int_{y_{0}}^{M} \int_{y_{0}}^{y} e^{-i \frac{t}{\sqrt{\mu}} \phi(z)} dz \left(\widehat{f}\right)^{\prime} \left(\frac{y}{\sqrt{\mu}} \right) dy.\\
\end{align*}

\noindent Then, \! using Van der Corput's Lemma and the fact that for $z \in [y_{0},y]$,
\newline
$|\phi^{\prime \prime}(z) | \geq |\phi^{\prime \prime}(y)|$ and $|\phi^{\prime \prime}(z)| \geq Cz^{-\frac{3}{2}}$,

\begin{align*}
|I_{2}(t)| &\leq \left\lvert \frac{M}{\sqrt{\mu}} \widehat{f}\left(\frac{M}{\sqrt{\mu}}\right) \right\rvert + \frac{C}{\mu^{\frac{3}{4}} \sqrt{t}} \int_{y_{0}}^{M} \left\lvert y^{\frac{3}{4}} \left(\widehat{f}\right)^{\prime} \left(\frac{y}{\sqrt{\mu}} \right) \right\rvert  dy\\
&\leq \left\lvert\widehat{f^{\prime}}\left(\frac{M}{\sqrt{\mu}}\right) \right\rvert  + \frac{C \mu^{\frac{1}{8}}}{\sqrt{t}} \int_{0}^{+\infty} \left\lvert \xi^{\frac{3}{4}} \left(\widehat{f}\right)^{\prime} \left(\xi \right) \right\rvert.
\end{align*}

\noindent Tending $M$ to $+\infty$ we get the result. The control for $\xi < 0$ is similar.
\end{proof}

\noindent Therefore, in the linear case, we have also a resonance.

\begin{cor}\label{resonance_p_resonant}

\noindent Let $P_{0} \in H^{3}(\R) \cap W^{2,1}(\R)$ such that $X P_{0} \in H^{3}(\R)$ and let 
\newline
$0 < \mu \leq \mu_{\max}$. Consider,

\begin{equation*}
\zeta_{R}(t,X) = - \frac{i t}{2} \int_{\R} \xi \omega(\sqrt{\mu} \xi) \widehat{P_{0}}(\xi) e^{-it \xi \omega(\sqrt{\mu} \xi)} e^{i X \xi} d\xi.
\end{equation*}

\noindent Then,

\begin{equation*}
\left\lvert \zeta_{R}(t,\cdot) \right\rvert_{\infty} \leq  C(\mu_{\max}) \sqrt{\frac{t}{\mu}} \left( \left\lvert P_{0} \right\rvert_{H^{3}} + \left\lvert P_{0} \right\rvert_{L^{1}} + \left\lvert X P_{0} \right\rvert_{H^{3}} \right),
\end{equation*}

\noindent and 

\begin{equation*}
\underset{t \shortrightarrow +\infty}{\underline{\lim}} \left\lvert \frac{1}{\sqrt{t}} \zeta_{R}(t,\cdot) \right\rvert_{\infty} \geq  C(P_{0}) > 0.
\end{equation*}
\end{cor}

\begin{proof}
\noindent We take $\widehat{f}(\xi)=\xi \omega(\sqrt{\mu} \xi) \widehat{P}_{0}(\xi)$. Then,

\begin{equation*}
\left\lvert \left(\widehat{f}\right)^{\prime} \hspace{-0.2cm}(\xi) \right\rvert \leq \left( 1+\sqrt{\mu}|\xi| \right) \left\lvert  \widehat{P}_{0}(\xi)\right\rvert + |\xi| \left\lvert  \left(\widehat{P}_{0}\right)^{\prime} \hspace{-0.2cm} (\xi)\right\rvert,
\end{equation*}

\noindent and the first inequality follows from the previous Proposition. For the second inequality, we use a stationary phase approximation. We denote $\phi(\xi) = \xi \omega(\xi)$. Let $\xi_{0} > 0$, such that $\left\lvert \xi_{0} \widehat{P_{0}}(\xi_{0}) \right\rvert = \left\lvert \xi \widehat{P_{0}} \right\rvert_{L^{\infty}}$, and $X_{\mu} < 0$, such that $\phi^{\prime}(\sqrt{\mu} \xi_{0}) = X_{\mu}$. Then, we have,

\begin{align*}
\underset{t \shortrightarrow +\infty}{\lim} \left\lvert \frac{1}{\sqrt{t}} \zeta_{R}(t,t X_{\mu}) \right\rvert &= \underset{t \shortrightarrow +\infty}{\lim}  \frac{\sqrt{t}}{2 \mu} \left\lvert \int_{\R} \xi \omega(\xi) \widehat{P_{0}}\left(\frac{\xi}{\sqrt{\mu}} \right) e^{-i \frac{t}{\sqrt{\mu}} \xi \left(\omega(\xi) - X_{\mu} \right)} d\xi \right\rvert\\
&= \frac{\sqrt{2 \pi}}{2 \mu^{\frac{1}{4}}} \left\lvert \frac{\omega(\xi_{0} \sqrt{\mu}) \xi_{0} \widehat{P_{0}}(\xi_{0})}{\sqrt{|\phi^{\prime \prime}(\xi_{0} \sqrt{\mu})|}} \right\rvert.\\
\end{align*}

\noindent Since $|\phi^{\prime \prime}(\xi)| \leq C |\xi|$ and $\omega(\xi_{0} \sqrt{\mu}) \geq C(\xi_{0}) \sqrt{\mu}$, we get the result.

\end{proof}

\begin{remark}
\noindent Notice that for all $s\in \R$, 

\begin{equation*}
\left\lvert \zeta_{R}(t,\cdot) + \frac{t}{2} P_{0}^{\prime}(\cdot - t) \right\rvert_{H^{s}} \leq \sqrt{\mu} t^{2} \left\lvert \nabla P_{0} \right\rvert_{H^{s+2}}.
\end{equation*}

\noindent Hence, by tending formally $\mu$ to $0$, we rediscover the result we get in the shallow water case (section \ref{A shallow water model with small topography variations}).
\end{remark}

\begin{remark}
\noindent Notice that for a general pressure term $P(t,X)$ we can show that the amplitude $\zeta$ satisfying

\begin{align*}
\widehat{\zeta}(t,\xi) &= \frac{i}{2} \int_{0}^{t} \xi \sqrt{\frac{\tanh(\sqrt{\mu} |\xi|)}{\sqrt{\mu} |\xi|}} \widehat{P}(\tau,\xi) e^{i(t-\tau) \xi \sqrt{\frac{\tanh(\sqrt{\mu} |\xi|)}{\sqrt{\mu} |\xi|}}}  d\tau\\
&-\frac{i}{2} \int_{0}^{t} \xi \sqrt{\frac{\tanh(\sqrt{\mu} |\xi|)}{\sqrt{\mu} |\xi|}} \widehat{P}(\tau,\xi) e^{i(\tau - t) \xi \sqrt{\frac{\tanh(\sqrt{\mu} |\xi|)}{\sqrt{\mu} |\xi|}}}  d\tau,
\end{align*}

\noindent satisfies also 

\begin{equation*}
\left\lvert \zeta(t,\cdot) \right\rvert_{\infty} \leq  C(\mu_{\max}) \sqrt{\frac{t}{\mu}} \left( \left\lvert P \right\rvert_{L^{\infty}(\R^{+}; L^{1}(\RD))} + \left\lvert P \right\rvert_{L^{\infty}(\R^{+}; H^{3}(\RD)} + \left\lvert X P \right\rvert_{L^{\infty}(\R^{+}; H^{3}(\RD)} \right).
\end{equation*}

\noindent Hence, contrary to the shallow water case, we can not hope a linear amplification with respect to the time $t$. Corollary \ref{resonance_p_resonant} also shows that the factor of amplification of $\sqrt{t}$ is optimal.
\end{remark}

\noindent Hence, we observe that in intermediate water depths, a resonance can occur but with a factor of amplification of $\sqrt{t}$ and not $t$. But we saw that in the shallow water case, the resonance occurs for a moving pressure with a speed equal to $1$, $P(t,X) = P_{0}(X-t)$. We wonder if this pressure can create a resonance. The following Proposition shows that the previous pressure can create a resonance with a factor of amplification of $t^{\frac{1}{3}}$.

\begin{prop}
\noindent Let $0 < \mu \leq \mu_{\max}$. Let $P_{0} \in L^{1}(\R) \cap H^{1}(\R)$ such that $\widehat{P_{0}}(0) \neq 0$. Consider, the amplitude $\zeta_{R}$ created by $P(t,X)=P_{0}(X-t)$,

\begin{equation}\label{zetaR1}
\widehat{\zeta_{R}}(t,\xi) = -\frac{i}{2} \xi \omega(\sqrt{\mu} \xi) \widehat{P_{0}}(\xi) e^{-it\xi} \int_{-t}^{0} e^{is \xi \left(\omega( \sqrt{\mu}\xi) -1 \right)} ds.
\end{equation}

\noindent Then, 

\begin{equation*}
\left\lvert \zeta_{R}(t,\cdot) \right\rvert_{\infty} \leq C(\mu_{\max}) \left( \frac{t^{\frac{1}{3}}}{\mu} \left\lvert P_{0} \right\rvert_{L^{1}} + \mu^{\frac{1}{4}} \left\lvert P_{0} \right\rvert_{H^{1}} \right).
\end{equation*}

\noindent Furthermore, if $X P_{0} \in H^{1}(\R)$,

\begin{equation*}
\underset{t \shortrightarrow +\infty}{\underline{\lim}} \left\lvert \frac{1}{t^{\frac{1}{3}}} \zeta_{R}(t,\cdot) \right\rvert_{\infty} \geq \frac{C}{\mu^{\frac{2}{3}}} \left\lvert \widehat{P_{0}}(0) \right\rvert.
\end{equation*}
\end{prop}

\begin{proof}
\noindent We have

\begin{align*}
\zeta_{R}(t,X) &= - \frac{i}{2} \int_{\R} \xi \omega(\sqrt{\mu} \xi) \widehat{P_{0}}(\xi) e^{-it\xi} \int_{-t}^{0} e^{is \xi \left(\omega( \sqrt{\mu}\xi) -1 \right)} e^{i X \xi} ds d\xi\\
&= - \frac{i}{2} \frac{1}{\mu} \int_{\R} \xi \omega(\xi) \widehat{P_{0}}\left(\frac{\xi}{\sqrt{\mu}} \right) e^{-i \frac{t}{\sqrt{\mu}} \xi} \int_{-t}^{0} e^{i \frac{s}{\sqrt{\mu}} \xi \left(\omega(\xi) -1 \right)} e^{i \frac{X}{\sqrt{\mu}} \xi} ds d\xi.
\end{align*}

\noindent We decompose this integral into $3$ parts.

\begin{align*}
\left\lvert I_{1}(t) \right\rvert &= \left\lvert \frac{1}{\mu} \int_{|\xi| \leq t^{-\frac{1}{3}}} \xi \omega(\xi) \widehat{P_{0}}\left(\frac{\xi}{\sqrt{\mu}} \right) e^{-i \frac{t}{\sqrt{\mu}} \xi} \int_{-t}^{0} e^{i \frac{s}{\sqrt{\mu}} \xi \left(\omega(\xi) -1 \right)} e^{i \frac{X}{\sqrt{\mu}} \xi} d\xi ds \right\rvert\\
&\leq \frac{t^{\frac{1}{3}}}{\mu} \left\lvert \widehat{P_{0}} \right\rvert_{\infty}.
\end{align*}

\noindent Furthermore, since $\left\lvert \omega(\xi) - 1 \right\rvert \geq C \xi^{2}$ for $0 \leq |\xi| \leq 1$, we have

\begin{align*}
\left\lvert I_{2}(t) \right\rvert &= \left\lvert \frac{1}{\mu} \int_{t^{-\frac{1}{3}} \leq |\xi| \leq 1} \xi \omega(\xi) \widehat{P_{0}}\left(\frac{\xi}{\sqrt{\mu}}\right) e^{-i \frac{t}{\sqrt{\mu}} \xi} \int_{-t}^{0} e^{i \frac{s}{\sqrt{\mu}} \xi \left(\omega(\xi) -1 \right)} e^{i \frac{X}{\sqrt{\mu}} \xi} d\xi ds \right\rvert\\
& = \left\lvert \frac{1}{\sqrt{\mu}} \int_{t^{-\frac{1}{3}} \leq |\xi| \leq 1} e^{i \frac{X}{\sqrt{\mu}} \xi} \frac{\omega(\xi)}{\omega(\xi)-1} \widehat{P_{0}}\left(\frac{\xi}{\sqrt{\mu}} \right) \left( e^{-i \frac{t}{\sqrt{\mu}} \xi}- e^{-i \frac{t}{\sqrt{\mu}} \xi \omega(\xi)} \right) d\xi \right\rvert\\
&\leq  C \frac{t^{\frac{1}{3}}}{\sqrt{\mu}} \left\lvert \widehat{P_{0}} \right\rvert_{\infty}.
\end{align*}

\noindent Finally,

\begin{align*}
\left\lvert I_{3}(t) \right\rvert &= \left\lvert \frac{1}{\mu} \int_{|\xi| \geq 1} \xi \omega(\xi) \widehat{P_{0}} \left(\frac{\xi}{\sqrt{\mu}} \right) e^{-i \frac{t}{\sqrt{\mu}} \xi} \int_{-t}^{0} e^{i \frac{s}{\sqrt{\mu}} \xi \left(\omega( \sqrt{\mu}\xi) -1 \right)} e^{i \frac{X}{\sqrt{\mu}} \xi} d\xi ds \right\rvert\\
& = \left\lvert \frac{1}{\sqrt{\mu}} \int_{|\xi| \geq 1} e^{i \frac{X}{\sqrt{\mu}} \xi} \frac{\omega(\xi)}{\omega(\xi)-1} \widehat{P_{0}}\left(\frac{\xi}{\sqrt{\mu}} \right) \left( e^{-i \frac{t}{\sqrt{\mu}} \xi}- e^{-i \frac{t}{\sqrt{\mu}} \xi \omega(\xi)} \right) d\xi \right\rvert\\
&\leq C \int_{|\xi| \geq \frac{1}{\sqrt{\mu}}} \left\lvert \widehat{P_{0}}(\xi) \right\rvert d\xi,\\
&\leq C \mu^{\frac{1}{4}} \left\lvert P_{0} \right\rvert_{H^{1}},
\end{align*}

\noindent and the first inequality follows. For the second inequality, we use a stationary phase approximation. We denote $\phi(\xi) := \xi(\omega(\xi) - 1)$. We recall that $\phi(\xi) = -\frac{1}{6} \xi^{3} + o(\xi^{3})$. Using a generalization of Morse Lemma at the order $3$, there exists $a > 0$ and $\psi \in \mathcal{C}^{\infty} \left([-a,a]\right)$, such that for all $|y| \leq a$,

\begin{equation*}
\phi(\psi(y)) = \frac{1}{6} \phi^{\prime \prime \prime}(0) y^{3}, \psi(0)= 0 \text{ and } \psi^{\prime}(0)=1.
\end{equation*}

\noindent Then,

\begin{align*}
I(s) &:= \int_{\R} \omega(\xi) \xi \widehat{P_{0}} \left( \frac{\xi}{\sqrt{\mu}} \right) e^{i \frac{s}{\sqrt{\mu}} \xi (\omega(\xi) -1)} d\xi\\
&= \int_{-a}^{a} \psi^{\prime}(y) \omega(\psi(y)) \psi(y) \widehat{P_{0}} \left( \frac{\psi(y)}{\sqrt{\mu}} \right) e^{i \frac{s}{6 \sqrt{\mu}} y^{3}} dy + o(s^{-\frac{2}{3}}) \\
&= \left(\frac {6 \sqrt{\mu}}{s} \right)^{\frac{2}{3}} \widehat{P_{0}}(0) \int_{z \in \R} z e^{i z^{3}} dz + o(s^{-\frac{2}{3}}).\\
\end{align*}

\noindent Therefore,

\begin{equation*}
\underset{t \shortrightarrow +\infty}{\lim} \left\lvert \frac{1}{t^{\frac{1}{3}}} \zeta_{R}(t,t) \right\rvert = \frac{C}{\mu^{\frac{2}{3}}} \left\lvert \widehat{P_{0}}(0) \right\rvert.
\end{equation*}
\end{proof}

\noindent Then, in intermediate water depths, a traveling pressure with a constant speed equal to $1$ is also resonant, but it takes more time to obtain a significant elevation of the level of the sea. In the following, we compute numerically some solutions. We take $P_{0}(X) = - e^{-X^{2}}$ and $\mu = 1$. The figure \ref{profil} is the evolution of a water wave because of a pressure of the form  \eqref{resonantP}. The solid curve is the wave and the dashed curve is the moving pressure. The figure \ref{profilspeed1} is the evolution is a water wave when the pressure moves with a speed $1$. The figure \ref{comparisonmax} compares the evolution of the maximum of the resonant case and the case when the speed is equal to $1$.

\vspace{-3cm}
\begin{center}
\begin{figure}[!h]
   \includegraphics[scale=0.15]{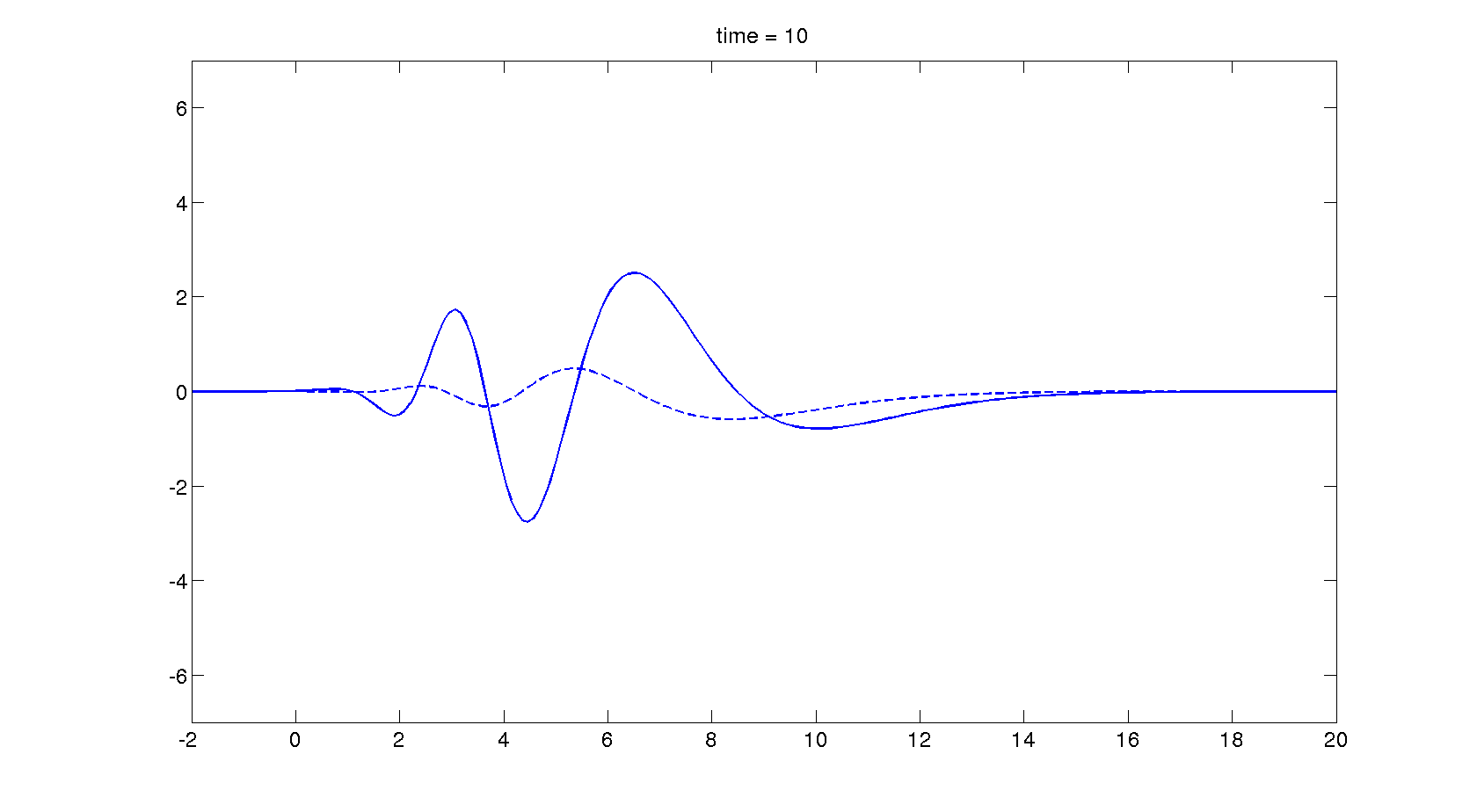}
   \caption{Evolution of the surface elevation $\zeta_{R}$ in \eqref{zetaR} (solid line) because of a resonant moving pressure $P$ in \eqref{resonantP} (dashed line).}
   \label{profil}
\end{figure}
\end{center}

\begin{center}
\begin{figure}[!h]
   \includegraphics[scale=0.15]{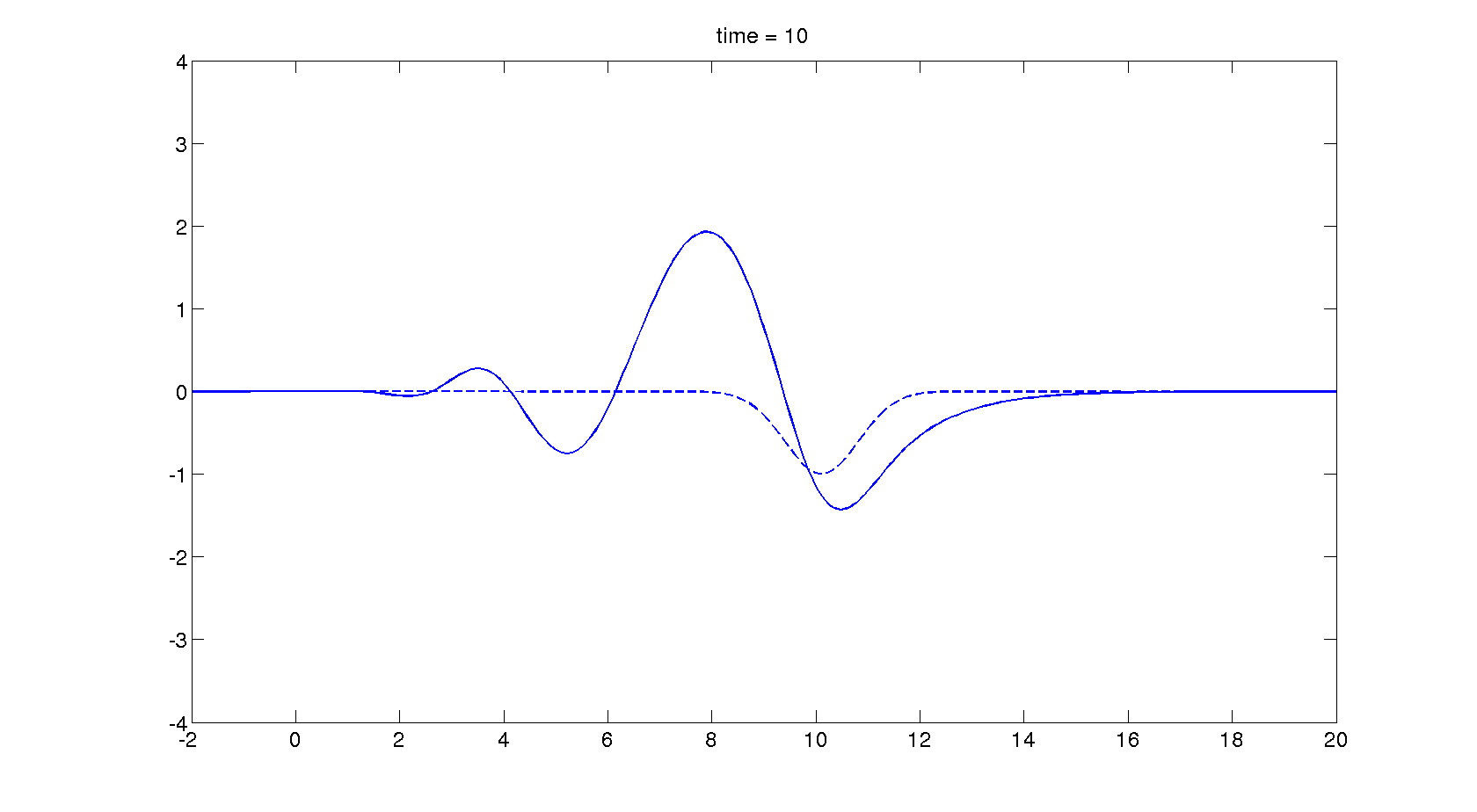}
   \caption{Evolution of the surface elevation $\zeta_{R}$ in \eqref{zetaR1} (solid line) because of a moving pressure $P$ with a speed of $1$ (dashed line).}
   \label{profilspeed1}
\end{figure}
\end{center}

\begin{center}
\begin{figure}[!h]
   \includegraphics[scale=0.15]{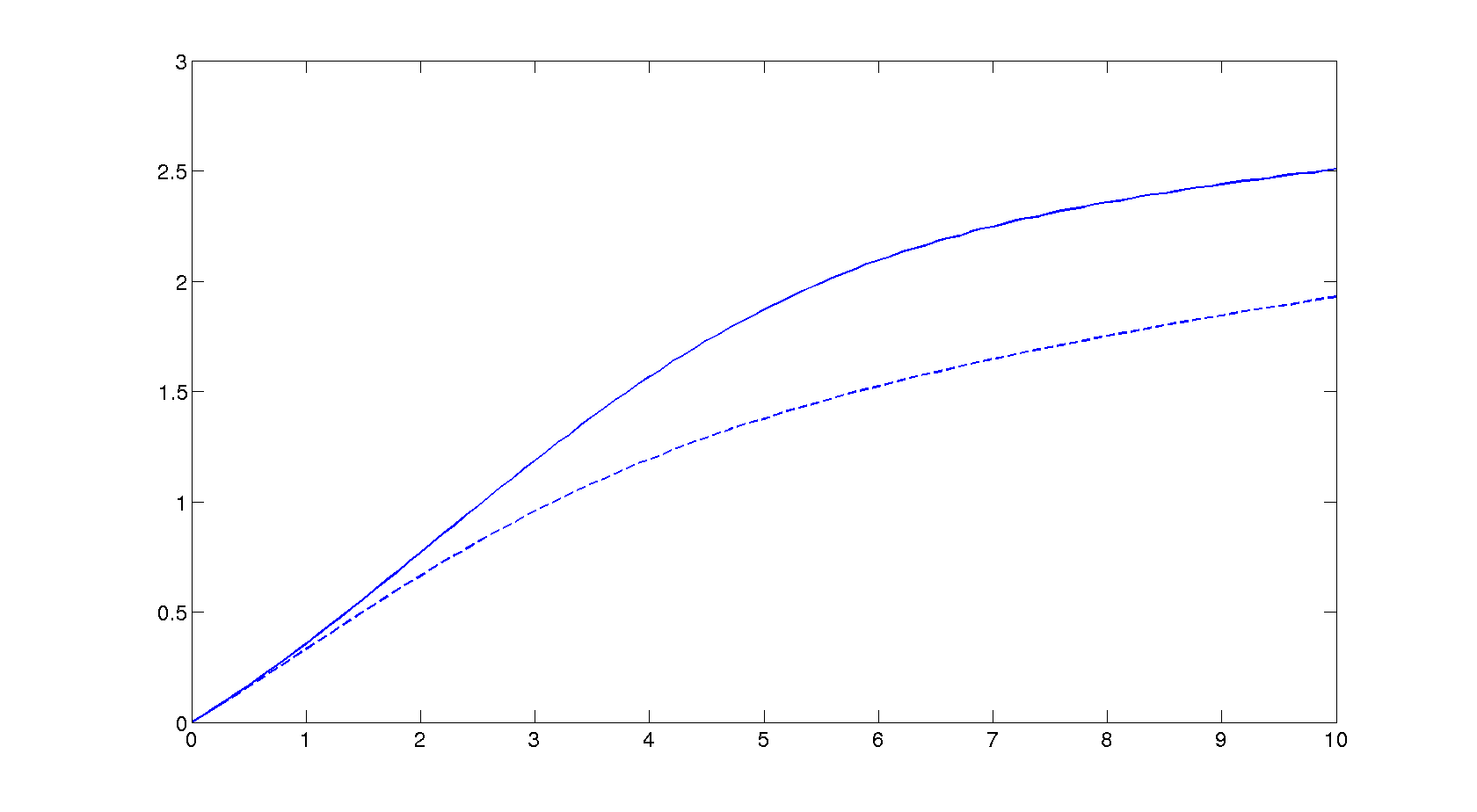}
   \caption{Evolution of the maximum of $\zeta_{R}$ in the resonant case (solid line) and the moving pressure with a speed of $1$ (dashed line).}
   \label{comparisonmax}
\end{figure}
\end{center}

\begin{remark}
\noindent In our work, we neglect the Coriolis effect. However, in view of the duration of the meteotsunami phenomenon, it would be more realistic to consider it. It will be studied in a future work (\cite{my_coriolis}) based on the work of A. Castro and D. Lannes (\cite{Castro_Lannes_shallow_water} and \cite{Castro_Lannes_vorticity}).
\end{remark}

\newpage
\appendix

\section{The Laplace problem}\label{Laplace_problem}

\subsection{Formulation of the problems}

\noindent In this part, we recall some results of Chapter 2 in \cite{Lannes_ww} and Section 4 of \cite{Iguchi_tsunami} and study the Laplace problem \eqref{Laplace_bott} in the Beppo Levi spaces. We suppose that the parameters $\epsilon$, $\mu$ and $\beta$ satisfy Condition \eqref{parameters_constraints}. The Laplace problem \eqref{Laplace_bott} is

\begin{equation*}
  \left\{
  \begin{aligned}
   & \Delta^{\mu}_{\! X,z} \Phi^{B} = 0 \text{ in } \Omega_{t} \text{ ,} \\
   &\Phi^{B}_{\;\;\, |z=\epsilon \zeta} = 0 \text{ , } \sqrt{1+\beta^{2} |\nabla b|^{2}} \partial_{\textbf{n}} \Phi^{B}_{\;\;\, |z=-1+\beta b} = B,
  \end{aligned}
  \right.
\end{equation*}

\noindent where $B = \partial_{t} b$. Notice that $\partial_{\textbf{n}}$ is here the upward \textit{conormal} derivative

\begin{equation*}
\partial_{\textbf{n}} \Phi^{B}= \textbf{n} \cdot \begin{pmatrix} \sqrt{\mu} I_{d} \hspace{0.1cm} 0 \\ \hspace{0.5cm} 0 \hspace{0.3cm} 1 \end{pmatrix} \nabla^{\mu}_{\! X,z} \Phi^{B}_{\;\;\, |z=-1+\beta b} .
\end{equation*}

\noindent For the study of \eqref{Laplace_surf} we refer to \cite{Lannes_ww}. We work with Beppo Levi spaces. We refer to \cite{Deny_Lions_beppo_levi} and Proposition 2.3 in \cite{Lannes_ww} for general results about these spaces. We recall that, for $s \geq 0$, 

\begin{equation*}
  \Hdot^{s}(\RD) := \left\lbrace \psi \in L^{2}_{\rm loc}(\RD) \text{, } \nabla \psi \in H^{s-1}(\RD) \right\rbrace,
\end{equation*}

\noindent that $\Hdot^{1}(\RD \times (-1,0))/\raisebox{-.65ex}{\ensuremath{\R}}$ is a Hilbert space for the norm $|\nabla_{\! X,z} \cdot |_{L^{2}}$ and that $H^{s}(\RD)$ is dense in $\Hdot^{s}(\RD)$. In order to fix the domain, we transform the problem into variable coefficients elliptic problem on $S := \RD \times (-1,0)$ (the flat strip). We introduce a regularizing diffeomorphism. Let $\theta :\R \shortrightarrow \R$ be a positive, compactly supported, smooth, even function equal to one near $0$. For $\delta > 0$ we define

\begin{equation*}
\Sigma := \begin{array}{l}
\quad S \quad \longrightarrow \quad \quad \quad \quad \Omega \\
(X,z) \mapsto \left( X,z + \sigma(X,z) \right),
\end{array}
\end{equation*}

\noindent and

\begin{equation*}
\sigma(X,z) := \left[\theta(\delta z |D|) \epsilon \zeta(X) - \theta(\delta (z+1) |D|) \beta b(X) \right] z + \epsilon \theta(\delta z |D|) \zeta(X).
\end{equation*}

\noindent We omit the dependence on $t$ here. In the following, we denote by $M$ a constant of the form 

\begin{equation*}
M = C \left( \frac{1}{h_{\min}},\mu_{\max},\epsilon |\zeta|_{H^{t_{0} + 1}(\RD)}, \beta |b|_{H^{t_{0} + 1}(\RD)} \right).
\end{equation*}

\noindent In order to study the Laplace problems in $S$, we have to treat the regularity in the direction $X$ and in the direction $z$ one at a time. We introduce the following spaces.

\begin{definition}
\noindent Let $s \in \R$. We define $\left( H^{s,1}(S), \lvert \cdot \rvert_{s,1} \right)$ and 
$\left( H^{s,0}(S), \lvert \cdot \rvert_{0,1} \right)$

\begin{equation*}
H^{s,1}(S) := L^{2}_{z} H^{s}_{X}(S) \cap H^{1}_{z} H^{s-1}_{X}(S) \text{, and } \lvert u \rvert_{H^{s,1}}^{2} = \lvert \Lambda^{s} u \rvert_{L^{2}}^{2} + \lvert \Lambda^{s-1} \partial_{z} u \rvert_{L^{2}}^{2},
\end{equation*}

\noindent and

\begin{equation*}
H^{s,0}(S) := L^{2}_{z} H^{s}_{X}(S) \text{, and } \lvert u \rvert_{H^{s,0}}^{2} = \lvert \Lambda^{s} u \rvert_{L^{2}}^{2}.
\end{equation*}

\end{definition}

\begin{remark}\label{L_infty_H_s}
\noindent We have the following embedding (see Proposition 2.10 in \cite{Lannes_ww}) for $s\in \R$

\begin{equation*}
H^{s+\frac{1}{2},1}(S) \subset L^{\infty}_{z} H^{s}_{\! X}(S).
\end{equation*}
\end{remark}

\noindent In the following, we fix $\delta > 0$ small enough. Then, we can transform our equations. We denote $\phi^{B} := \Phi^{B} \! \circ \! \Sigma$ and we get that

\begin{equation}\label{Laplace_bott_S}
  \left\{
  \begin{aligned}
   & \nabla^{\mu}_{\! X,z} \cdot P(\Sigma) \nabla^{\mu}_{\! X,z} \phi^{B} = 0 \text{ in } S \text{,}\\
   &\phi^{B}_{\;\, |z=0} = 0 \text{ , } \partial_{\textbf{n}} \phi^{B}_{\;\, |z=-1} = B,
  \end{aligned}
  \right.
\end{equation} 

\noindent with $P(\Sigma) = I_{d+1 \times d+1} + Q(\Sigma)$ and

\begin{equation}\label{Q_Sigma}
Q(\Sigma) := \begin{pmatrix}
\partial_{z} \sigma I_{d \times d} & -\sqrt{\mu} \nabla_{\! X} \sigma \\
-\sqrt{\mu} \nabla_{\! X} \sigma^{t} & \frac{-\partial_{z} \sigma + \mu \left\lvert \nabla_{\! X} \sigma \right\lvert^{2}}{1+ \partial_{z} \sigma}
\end{pmatrix}.
\end{equation}

\noindent Notice that $P(\Sigma)$ is well defined if $\delta$ is small enough and that $\partial_{\textbf{n}} := \textbf{e}_{z} \cdot (P(\Sigma) \nabla^{\mu}_{\! X,z} \; \cdot \; )$. We have to know the regularity of $P(\Sigma)$. It is the subject of the next proposition (see Proposition 2.18 and Lemma 2.26 in \cite{Lannes_ww}).

\begin{prop}\label{P_Sigma_regularity}
Let $t_{0} > \frac{d}{2}$, $\zeta,b \in \Hzetaa$ such that Condition \eqref{nonvanishing} is satisfied. Then,

\begin{equation*}
\left\lvert Q(\Sigma) \right\lvert_{H^{t_{0}+\frac{1}{2},1}} \text{,} \left\lvert \Lambda^{t_{0}} Q(\Sigma) \right\lvert_{L^{\infty}_{z} L^{2}_{X} (S)} \text{,} \left\lvert \Lambda^{t_{0}-1} \partial_{z} Q(\Sigma) \right\lvert_{L^{\infty}_{z} L^{2}_{X} (S)} \leq M.
\end{equation*}

\noindent Furthermore, $P(\Sigma)$ is coercive. There exist a constant $k(\Sigma) > 0$ such that $\frac{1}{k(\Sigma)} \leq M$ and

\begin{equation*}
\forall \Theta \in \RDD \text{ , } \forall (X,z) \in S \text{ , } P(\Sigma)(X,z) \Theta \cdot \Theta \geq k(\Sigma) |\Theta|^{2}.
\end{equation*}
\end{prop}

\noindent We have a variational formulation of the Laplace problem \eqref{Laplace_bott_S}. We introduce

\begin{equation*}
H^{1}_{0,surf}(S) := \overline{\mathcal{D}(S \cup \{z=-1 \})}^{\; |\;|_{H^{1}(S)}} = \overline{\mathcal{D}(S \cup \{z=-1 \})}^{\; |\;|_{\Hdot^{1}(S)}}.
\end{equation*}

\noindent See Proposition 2.3 (3) in \cite{Lannes_ww} for a proof of the second equality.

\begin{definition}\label{variational_formulation_Bd}
Let $B \in H^{-\frac{1}{2}}(\RD)$. We say that $\phi \in H^{1}_{0,surf}(S)$ is a variational solution of \eqref{Laplace_bott_S} if for all $\varphi \in H^{1}_{0,surf}(S)$,

\begin{equation*}
\int_{S} \nabla^{\mu}_{\! X,z} \phi \cdot P(\Sigma) \nabla^{\mu}_{\! X,z} \varphi =  - \langle B , \varphi_{|z=-1} \rangle_ {H^{-\frac{1}{2}} - H^{\frac{1}{2}}}.
\end{equation*}
\end{definition}

\noindent We have also the following trace result that we can prove easily using a density argument.

\begin{lemma}\label{trace_formula}
For all $\varphi \in H^{1}_{0,surf}(S)$ we have

\begin{equation*}
\left\lvert \sqrt{1+\sqrt{\mu}|D|} \;  \varphi_{|z=-1} \right\lvert_{L^{2}(\RD)} \leq 2 \left\lvert \nabla^{\mu}_{\! X,z} \varphi \right\lvert_{L^{2}(S)}.
\end{equation*}
\end{lemma}

\noindent We can now establish existence and uniqueness results. 

\begin{prop}\label{existence_variational_problems}
Let $B \in H^{-\frac{1}{2}}(\RD)$ and $\zeta \text{, } b \in \Hzetaa$ satisfying \eqref{nonvanishing}. Then, the problem \eqref{Laplace_bott_S} has a unique variational solution named $B^{\mathfrak{d}} \in H^{1}_{0,surf}(S)$.
\end{prop}

\begin{proof}
Because $S$ is bounded in the direction $z$ and that $P(\Sigma)$ is uniformly coercive, the results follow from the Lax-Milgram theorem and Poincar\'e inequality in $H^{1}_{0,surf}(S)$.
\end{proof}

\noindent In this part, we study the Laplace problem \eqref{Laplace_bott}, but the same work can be done for \eqref{Laplace_surf} (see Chapter 2 in \cite{Lannes_ww}) and we can transform \eqref{Laplace_bott} as follows

\begin{equation}\label{Laplace_surf_S}
  \left\{
  \begin{aligned}
   & \nabla^{\mu}_{\! X,z} \cdot P(\Sigma) \nabla^{\mu}_{\! X,z} \phi^{S} = 0 \text{ in } S \text{,}\\
   &\phi^{S}_{\;\, |z=0} = \psi \text{ , } \partial_{\textbf{n}} \phi^{S}_{\;\, |z=-1} = 0.
  \end{aligned}
  \right.
\end{equation} 

\noindent In the following, we denote by $\psih$, the unique solution of \eqref{Laplace_surf_S}.

\subsection{Regularity estimates of the solutions}

\noindent In this part, we give some regularity estimates.

\begin{thm}\label{Bd_regularity}
Let $t_{0} > \frac{d}{2}$ and $0 \leq s \leq t_{0} + \frac{1}{2}$. Let $\zeta,b \in \Hzetaa$ be such that Condition \eqref{nonvanishing} is satisfied. Then, for all $B \in H^{s-\frac{1}{2}}(\RD)$, we have 

\begin{equation*}
\left\lvert \Lambda^{s} \nabla^{\mu}_{\! X,z} B^{\mathfrak{d}} \right\lvert_{L^{2}(S)} \leq  M \left\lvert \frac{1}{\sqrt{1+\sqrt{\mu}|D|}} \; B \right\lvert_{H^{s}}.
\end{equation*}

\noindent Futhermore, if $s \geq \max(0,1-t_{0})$, we have

\begin{equation*}
\left\lvert \Lambda^{s-1} \partial_{z} \nabla^{\mu}_{\! X,z} B^{\mathfrak{d}} \right\lvert_{L^{2}(S)} \leq  M \left\lvert \frac{1}{\sqrt{1+\sqrt{\mu}|D|}} \; B \right\lvert_{H^{s}}.
\end{equation*}
\end{thm}

\begin{proof}
\noindent Let $\delta > 0$ and $\chi$ be a smooth compactly supported real function that is equal to $1$ near $0$. We introduce the smoothing operator  $\Lambda^{s}_{\delta} :=  \chi(\delta \Lambda) \Lambda^{s}$. We know that $B^{\mathfrak{d}} \in H^{1}_{0,surf}(S)$. Therefore, using $\Lambda^{2s}_{\delta} B^{\mathfrak{d}}$ a test function, we have

\begin{equation*}
\int_{S} \nabla^{\mu}_{\! X,z} B^{\mathfrak{d}} \cdot P(\Sigma) \nabla^{\mu}_{\! X,z} \Lambda^{2s}_{\delta} B^{\mathfrak{d}} =  -\int_{\RD} B (\Lambda^{2s}_{\delta} B^{\mathfrak{d}})_{|z=-1}.
\end{equation*}

\noindent Since $P(\Sigma)$ is symmetric, $\Lambda^{s}_{\delta}$ commutes with $\nabla^{\mu}_{\! X,z}$ and is independent of $z$ we obtain that

\begin{align*}
\int_{S} P(\Sigma) \Lambda^{s}_{\delta} \nabla^{\mu}_{\! X,z} B^{\mathfrak{d}} \cdot  \nabla^{\mu}_{\! X,z} \Lambda^{s}_{\delta} B^{\mathfrak{d}} =& -\int_{S} [\Lambda^{s}_{\delta},Q(\Sigma)] \nabla^{\mu}_{\! X,z} B^{\mathfrak{d}} \cdot  \nabla^{\mu}_{\! X,z} \Lambda^{s}_{\delta} B^{\mathfrak{d}}\\
& -\int_{\RD} \frac{\Lambda^{s}_{\delta}}{\sqrt{1+\sqrt{\mu} |D|}} B \left(\sqrt{1+\sqrt{\mu} |D|} \Lambda^{s}_{\delta} B^{\mathfrak{d}} \right)_{|z=-1}.
\end{align*}

\noindent Then by coercivity of $P(\Sigma)$ and trace inequality \ref{trace_formula}

\begin{align*}
k(\Sigma) \left\lvert \Lambda^{s}_{\delta} \nabla^{\mu}_{\! X,z} B^{\mathfrak{d}} \right\lvert^{2}_{L^{2}(S)} &\leq \left\lvert [\Lambda^{s}_{\delta},Q(\Sigma)] \nabla^{\mu}_{\! X,z} B^{\mathfrak{d}} \right\lvert_{L^{2}} \left\lvert \Lambda^{s}_{\delta} \nabla^{\mu}_{\! X,z} B^{\mathfrak{d}} \right\lvert_{L^{2}(S)} \\
&+ 2 \left\lvert \Lambda^{s}_{\delta} \nabla^{\mu}_{\! X,z} B^{\mathfrak{d}} \right\lvert_{L^{2}(S)} \left\lvert \frac{\Lambda^{s}_{\delta}}{\sqrt{1+\sqrt{\mu} |D|}} B \right\lvert_{L^{2}} \hspace{-0.3cm},
\end{align*}

\noindent and

\begin{equation*}
k(\Sigma) \left\lvert \Lambda^{s}_{\delta} \nabla^{\mu}_{\! X,z} B^{\mathfrak{d}} \right\lvert_{L^{2}(S)} \leq \left\lvert [\Lambda^{s}_{\delta},Q(\Sigma)] \nabla^{\mu}_{\! X,z} B^{\mathfrak{d}} \right\lvert_{L^{2}(S)}  + 2 \left\lvert \frac{\Lambda^{s}_{\delta}}{\sqrt{1+\sqrt{\mu} |D|}} B \right\lvert_{L^{2}} \hspace{-0.9cm}.
\end{equation*}

\noindent We have to distinguish two cases.

\medskip
\medskip
\medskip
\hspace{-0.8cm} \underline{\textbf{a) $0 \leq s \leq t_{0}$ :}}
\medskip
\medskip
\medskip

\noindent The commutator estimate \ref{commutator_estimate1} (with $T_{0} = t_{0}$) and Proposition \ref{P_Sigma_regularity} give

\begin{align*}
k(\Sigma) \left\lvert \Lambda^{s}_{\delta} \nabla^{\mu}_{\! X,z} B^{\mathfrak{d}} \right\lvert_{L^{2}(S)} &\leq C \left\lvert Q(\Sigma) \right\lvert_{L^{\infty}_{z} H^{t_{0}}_{\! X}(S)} \left\lvert \Lambda^{s-\epsilon}_{\delta} \nabla^{\mu}_{\! X,z} B^{\mathfrak{d}} \right\lvert_{L^{2}(S)}  + 2 \left\lvert \frac{\Lambda^{s}_{\delta}}{\sqrt{1+\sqrt{\mu} |D|}} B \right\lvert_{L^{2}} \hspace{-0.9cm}\\
&\leq M \left\lvert \Lambda^{s-\epsilon}_{\delta} \nabla^{\mu}_{\! X,z} B^{\mathfrak{d}} \right\lvert_{L^{2}(S)}  + 2 \left\lvert \frac{\Lambda^{s}_{\delta}}{\sqrt{1+\sqrt{\mu} |D|}} B \right\lvert_{L^{2}}
\end{align*}

\noindent for some $\epsilon > 0$ small enough ($\epsilon < t_{0} - \frac{d}{2}$). Using a finite induction on $s$ and taking the limit when $\delta$ goes to $0$, the first inequality follows. For the second estimate, we only need to give a control of $\partial_{z}^{2} B^{\mathfrak{d}}$. We use Equation \eqref{Laplace_bott_S} satisfied by $B^{\mathfrak{d}}$. We express $P(\Sigma)$ as

\begin{equation*}
P(\Sigma) := \begin{pmatrix}
(1+a(X,z)) I_{d \times d} & \textbf{q}(X,z) \\
\textbf{q}^{t}(X,z) & 1+q_{d+1}(X,z)
\end{pmatrix}.
\end{equation*}

\noindent A simple computation gives

\begin{align*}
(1+q_{d+1}) \partial_{z}^{2} B^{\mathfrak{d}} = &-\sqrt{\mu} \nabla_{\! X} \cdot \left((1+ a)\sqrt{\mu} \nabla_{\! X} B^{\mathfrak{d}} \right) - \sqrt{\mu} \nabla_{\! X} \cdot \left( \partial_{z} B^{\mathfrak{d}} \textbf{q} \right) \\
&-\sqrt{\mu} \partial_{z} \textbf{q} \cdot \nabla_{\! X} B^{\mathfrak{d}} - \sqrt{\mu} \partial_{z} \nabla_{\! X} B^{\mathfrak{d}} \cdot \textbf{q} - \partial_{z} q_{d+1} \partial_{z} B^{\mathfrak{d}}.
\end{align*}

\noindent We have $a \text{, } \textbf{q} \text{, } q_{d+1} \in L^{\infty}_{z} H^{t_{0}}_{X}(S)$, $\partial_{z} \textbf{q} \text{, } \partial_{z} q_{d+1} \in L^{\infty}_{z} H^{t_{0}-1}_{X}(S)$ and $1+q_{d+1} \geq k(\Sigma)$. Then, since $s \geq 1-t_{0}$ and $\nabla_{\! X} B^{\mathfrak{d}} \in H^{s,1}(S)$, by the product estimates \ref{product_estimate2} and \ref{inverse_estimate1} (with $T_{0} = t_{0}$), we obtain the result.

\medskip
\medskip
\medskip
\hspace{-0.8cm} \underline{\textbf{b) $t_{0} \leq s \leq t_{0} + \frac{1}{2}$ :}}
\medskip
\medskip
\medskip

\noindent The commutator estimate \ref{commutator_estimate2} (with $T_{0} = t_{0} + \frac{1}{2}$ and $t_{1} > \frac{1}{2}$) and Proposition \ref{P_Sigma_regularity} give

\begin{align*}
k(\Sigma) \left\lvert \Lambda^{s}_{\delta} \nabla^{\mu}_{\! X,z} B^{\mathfrak{d}} \right\lvert_{L^{2}(S)}  &\leq M \Bigg[ \left\lvert \Lambda^{s + \frac{1}{2} - t_{1}}_{\delta} \nabla^{\mu}_{\! X,z} B^{\mathfrak{d}} \right\lvert_{L^{2}(S)}\\
& + \left\lvert \Lambda^{s - 1 + \frac{1}{2} -t_{1}}_{\delta} \partial_{z} \nabla^{\mu}_{\! X,z} B^{\mathfrak{d}} \right\lvert_{L^{2}(S)} + 2 \left\lvert \frac{\Lambda^{s}_{\delta}}{\sqrt{1+\sqrt{\mu} |D|}} B \right\lvert_{L^{2}} \Bigg].
\end{align*}

\noindent We denote $\epsilon := \frac{1}{2} -t_{1}$. We obtain the first inequality for $t_{0} \leq s \leq t_{0} + \epsilon$ thanks to the previous case. Furthermore, we saw that 

\begin{align*}
(1+q_{d+1}) \partial_{z}^{2} B^{\mathfrak{d}} = &-\sqrt{\mu} \nabla^{\mu}_{\! X} \cdot \left( (1+a)\sqrt{\mu} \nabla_{\! X} B^{\mathfrak{d}} \right) - \sqrt{\mu} \nabla_{\! X} \cdot \left( \partial_{z} B^{\mathfrak{d}} \textbf{q} \right) \\
&-\sqrt{\mu} \partial_{z} \textbf{q} \cdot \nabla_{\! X} B^{\mathfrak{d}} - \sqrt{\mu} \partial_{z} \nabla_{\! X} B^{\mathfrak{d}} \cdot \textbf{q} - \partial_{z} q_{d+1} \partial_{z} B^{\mathfrak{d}}.
\end{align*}

\noindent We have $a \text{, } \textbf{q} \text{, } q_{d+1} \in L^{2}_{z} H^{t_{0}+ \frac{1}{2}}_{X}(S)$, $\partial_{z} \textbf{q} \text{, } \partial_{z} q_{d+1} \in L^{2}_{z} H^{t_{0} - \frac{1}{2}}_{X}(S)$ and $1+q_{d+1} \geq k(\Sigma)$. Then, since $s \geq 1-t_{0}$ and $\nabla_{\! X} B^{\mathfrak{d}} \in L^{\infty}_{z} H^{s-\frac{1}{2}}_{\! X}(S)$, by the product estimates \ref{product_estimate2} and \ref{inverse_estimate2} (with $T_{0} = t_{0}$), and we obtain the second inequality for $t_{0} \leq s \leq t_{0} + \epsilon$. Using a finite induction, we obtain the first and the second inequality.
\end{proof}

\section{The Dirichlet-Neumann and the Neumann-Neumann operators}\label{The Dirichlet-Neumann and the Neumann-Neumann operators}

\noindent We refers to Chapter 3 of \cite{Lannes_ww} for more details about the Dirichlet-Neumann operator and Section 3 in \cite{Iguchi_tsunami} for the study of these operators.

\subsection{Main properties}

\noindent We can express the Neumann-Neumann operator with the formalism of the previous section.  For $\psi \in  \Hdot^{\frac{3}{2}}(\RD)$ and $B \in H^{\frac{1}{2}}(\RD)$ we have

\begin{equation}\label{DN_operator_S}
\G(\psi) = \left( \textbf{e}_{z} \cdot P(\Sigma) \nabla^{\mu}_{\! X,z} \psih \right)_{|z=0},
\end{equation}

\noindent and

\begin{equation}\label{NN_operator_S}
\GNN(B) = \left( \textbf{e}_{z} \cdot P(\Sigma) \nabla^{\mu}_{\! X,z} B^{\mathfrak{d}} \right)_{|z=0}.
\end{equation}

\begin{remark}\label{operator_in_0}
\noindent Notice that (see Proposition 3.2 in \cite{Iguchi_tsunami})

\begin{equation*}
\frac{1}{\mu} G_{\mu}[0,0](\psi) = |D|^{2} \frac{\tanh(\sqrt{\mu} |D|)}{\sqrt{\mu} |D|} \psi \text{  and  } G_{\mu}^{N \! N}[0,0](B) = \frac{1}{\cosh(\sqrt{\mu} |D|)} B.
\end{equation*}

\end{remark}

\noindent We recall that $\G$ is symmetric and maps continuously $\Hdot^{\frac{1}{2}}(\RD)$ into $\left( \Hdot^{\frac{1}{2}}(\RD)/\raisebox{-.65ex}{\ensuremath{\R}} \right)^{\prime}$ (see Paragraph 3.1. in \cite{Lannes_ww}). We need an extension result in $H^{-\frac{1}{2}}(\RD)$ in order to give a dual formulation of the Neumann-Neumann operator.

\begin{definition}\label{psidiese_definition}
\noindent Let $\varphi \in H^{-\frac{1}{2}}(\RD)$. We define $\varphi^{\#}$ as

\begin{equation*}
\varphi^{\#} = \frac{\sinh([z+1]\sqrt{\mu}|D|)}{\sinh(\sqrt{\mu}|D|)} \varphi.
\end{equation*} 

\end{definition}

\begin{remark}

\noindent $\varphi^{\#}$ satisfies weakly

\begin{equation*}
\left\{
  \begin{aligned}
    &\Delta^{\mu}_{\! X,z} \varphi^{\#} = 0 \text{ in } S \text{, }\\
    &\varphi^{\#}_{\; \; |z=0} = \varphi \text{, } \varphi^{\#}_{\; \; |z=-1} = 0.
  \end{aligned}
  \right.
\end{equation*}

\end{remark}

\noindent We can prove easily regularity results for $\varphi^{\#}$ similar to $\varphi^{\mathfrak{h}}$.

\begin{prop}\label{psidiese_regularity}
Let $s \geq 0$ and $\varphi \in H^{s-\frac{1}{2}}(\RD)$. Then, 

\begin{equation*}
\left\lvert \Lambda^{s-1} \nabla^{\mu}_{\! X,z} \varphi^{\#} \right\lvert_{L^{2}(S)} + \frac{1}{\sqrt{\mu}} \left\lvert \Lambda^{s-2} \partial_{z} \nabla^{\mu}_{\! X,z} \varphi^{\#} \right\lvert_{L^{2}(S)} \leq C \left\lvert \sqrt{1+\sqrt{\mu}|D|} \varphi \right\lvert_{H^{s-1}}.
\end{equation*}
\end{prop}

\noindent We can now give a dual formulation of the Neumann-Neumann operator. We introduce the Dirichlet-Dirichlet operator, for $\psi \in H^{\frac{1}{2}}(\RD)$, 

\begin{equation}\label{DD_operator_S}
\GDD(\psi) := \left(\psih \right)_{|z=-1}.
\end{equation}

\noindent The following result is Proposition 3.3 in \cite{Iguchi_tsunami}.

\begin{prop}\label{main_properties_GNN}
\noindent Let $t_{0} > \frac{d}{2}$, $B \in H^{-\frac{1}{2}}(\RD)$ and $\zeta,b \in \Hzetaa$ such that \eqref{nonvanishing} is satisfied. $\GNN(\cdot)$ can be extended to $H^{-\frac{1}{2}}(\RD)$ with the dual formulation 

\begin{equation}\label{dual_definition_GNN}
  \GNN(B) = \left\{
  \begin{aligned}
    &H^{\frac{1}{2}}(\RD) \; \longrightarrow \quad \quad \quad \R\\
    &\quad \quad \varphi  \quad \quad \longmapsto \int_{S} P(\Sigma) \nabla^{\mu}_{\! X,z}  B^{\mathfrak{d}} \cdot \nabla^{\mu}_{\! X,z} \varphi^{\#}.
  \end{aligned}
  \right.
\end{equation}

\noindent Furthermore, the adjoint of $\GNN$ is $\GDD$. For all $B \in H^{-\frac{1}{2}}(\RD)$ and $\varphi \in H^{\frac{1}{2}}(\RD)$,

\begin{equation*}
\left(\GNN(B), \varphi \right)_{H^{-\frac{1}{2}} - H^{\frac{1}{2}}} = \left(B, \GDD(\varphi) \right)_{H^{-\frac{1}{2}} - H^{\frac{1}{2}}}.\\
\end{equation*}
\end{prop}

\noindent In order to study shape derivatives of the Dirichlet-Neumann and the Neumann-Neumann operators, we have to introduce the Neumann-Dirichlet operator. For $B \in H^{-\frac{1}{2}}(\RD)$,  we define

\begin{equation}\label{ND_operator_S}
\GND(B) := \left(B^{\mathfrak{d}} \right)_{|z=-1}.
\end{equation}

\noindent The following result is a symmetry property and a dual formulation of the Neumann-Dirichlet operator. 

\begin{prop}
Let $B \in H^{-\frac{1}{2}}(\RD)$ and $\zeta,b \in \Hzetaa$ such that \eqref{nonvanishing} is satisfied. $\GND(B)$ can be view as 

\begin{equation}\label{dual_definition_GND}
  \GND(B) = \left\{
  \begin{aligned}
    &H^{-\frac{1}{2}}(\RD) \; \longrightarrow \quad \quad \quad \R\\
    &\quad \quad C \quad \quad \longmapsto - \int_{S} P(\Sigma) \nabla^{\mu}_{\! X,z} B^{\mathfrak{d}} \cdot \nabla^{\mu}_{\! X,z} C^{\mathfrak{d}} .
  \end{aligned}
  \right.
\end{equation}

\noindent Furthermore, $\GND(\cdot)$ is a negative symmetric operator and,  for all \,$B_{1} \text{, } B_{2}$ in $H^{-\frac{1}{2}}(\RD)  \text{, }$

\begin{equation*}
\left( \GND(B_{1}),B_{2} \right)_{(H^{-1/2})^{'} - H^{-1/2}} = \left(\GND(B_{2}), B_{1} \right)_{(H^{-1/2})^{'} - H^{-1/2}}.
\end{equation*}
\end{prop}

\noindent We refer to Proposition 3.3 in \cite{Iguchi_tsunami} for a proof of this result. 

\subsection{Regularity Estimates}\label{Operators_estimates}

\noindent In this part we give some controls the Neumann-Neumann operators.

\begin{prop}\label{controls_GNN}
Let $t_{0} > \frac{d}{2}$, $0 \leq s \leq t_{0} + \frac{1}{2}$ and $\zeta,b \in \Hzetaa$ such that Condition \eqref{nonvanishing} is satisfied. Then, $\GNN$ maps continuously $H^{s-\frac{1}{2}}(\RD)$ into itself

\begin{equation*}
|\GNN(B)|_{H^{s-\frac{1}{2}}} \leq M \left\lvert  B \right\lvert_{H^{s-\frac{1}{2}}}.
\end{equation*}
\end{prop}

\begin{proof}
\noindent This Proposition follows by Theorem \ref{Bd_regularity} and by using the same arguments that Theorem 3.15 in \cite{Lannes_ww}.
\end{proof}

\noindent We can extend these estimates to $\w[\epsilon \zeta, \beta b]$, the vertical velocity at the surface and to $\V[\epsilon \zeta, \beta b]$ the horizontal velocity at the surface. These operators appear naturally when we differentiate the Dirichlet-Neumann and the Neumann-Neumann operator with respect to the surface $\zeta$. We define

\begin{equation}
  \w[\epsilon \zeta, \beta b] := \left\{
  \begin{aligned}
   & \Hdot^{s+\frac{1}{2}}(\RD) \times H^{s-\frac{1}{2}}(\RD) \rightarrow \quad \quad \quad \quad H^{s-\frac{1}{2}}(\RD) \\
   &\quad \, \quad (\psi,B) \quad \;\; \mapsto \frac{\G(\psi) + \mu \GNN(B) + \epsilon \mu \nabla \zeta \cdot\nabla \psi}{1+\epsilon^{2} \mu |\nabla \zeta|^{2}},     
  \end{aligned}
   \right.
\end{equation}

\hspace{-0.8cm} and

\begin{equation}
  \hspace{-3.8cm} \V[\epsilon \zeta, \beta b] := \left\{
  \begin{aligned}
   & \Hdot^{s+\frac{1}{2}}(\RD) \times H^{s-\frac{1}{2}}(\RD) \rightarrow \quad H^{s-\frac{1}{2}}(\RD) \\
   &\quad \, \quad (\psi,B) \quad \;\; \mapsto \nabla \psi - \epsilon \w[\epsilon \zeta, \beta b](\psi,B) \nabla \zeta.     
  \end{aligned}
   \right.
\end{equation}

\begin{prop}\label{controls_w}

Let $t_{0} > \frac{d}{2}$, $0 \leq s \leq t_{0} + \frac{1}{2}$ and $\zeta,b \in \Hzetaa$ such that Condition \eqref{nonvanishing} is satisfied. Then, 
$\w[\epsilon \zeta, \beta b]$ maps continuously $\Hdot^{s+\frac{1}{2}}(\RD) \times H^{s-\frac{1}{2}}(\RD)$ into $H^{s-\frac{1}{2}}(\RD)$ 

\begin{equation*}
|\w[\epsilon \zeta, \beta b](\psi,B)|_{H^{s-\frac{1}{2}}} \leq M \left(\mu^{\frac{3}{4}}|\mathfrak{P} \psi|_{H^{s}} + \mu |B|_{H^{s-\frac{1}{2}}} \right).
\end{equation*}

\noindent Furthermore, if $1 \leq s \leq t_{0}$, $\w[\epsilon \zeta, \beta b]$ maps continuously $\Hdot^{s+1}(\RD) \times H^{s-\frac{1}{2}}(\RD)$ into $H^{s-\frac{1}{2}}(\RD)$

\begin{equation*}
|\w[\epsilon \zeta, \beta b](\psi,B)|_{H^{s-\frac{1}{2}}} \leq M \mu \left( |\mathfrak{P} \psi|_{H^{s+\frac{1}{2}}} + |B|_{H^{s-\frac{1}{2}}} \right).
\end{equation*}

\noindent Finally, we have the same continuity result for $\V[\epsilon \zeta,\beta b]$.
\end{prop}

\noindent We can also give some regularity estimates for $\GDD$ since it is the adjoint of $\GNN$.

\begin{prop}\label{controls_GDD}
Let $t_{0} > \frac{d}{2}$, $0 \leq s \leq t_{0} + \frac{1}{2}$ and $\zeta,b \in \Hzetaa$ such that Condition \eqref{nonvanishing} is satisfied. Then, $\GDD$ maps continuously $\Hdot^{s+\frac{1}{2}}(\RD)$ into itself

\begin{equation*}
\left\lvert \nabla \GDD(\psi) \right\rvert_{H^{s-\frac{1}{2}}} \leq M \left\lvert  \nabla \psi \right\rvert_{H^{s-\frac{1}{2}}}.
\end{equation*}
\end{prop}

\noindent Finally, we can give some regularity estimates for $\GND$.

\begin{prop}\label{controls_GND}
Let $t_{0} > \frac{d}{2}$, $0 \leq s \leq t_{0} + \frac{1}{2}$ and $\zeta,b \in \Hzetaa$ such that Condition \eqref{nonvanishing} is satisfied. Then, $\GND$ maps continuously $H^{s-\frac{1}{2}}(\RD)$ into $H^{s+\frac{1}{2}}(\RD)$

\begin{equation*}
\left\lvert \GND(B) \right\rvert_{H^{s+\frac{1}{2}}} \leq M \left\lvert   B \right\rvert_{H^{s-\frac{1}{2}}}.
\end{equation*}
\end{prop}

\noindent In the same way, we can extend also these estimates to $\ws[\epsilon \zeta, \beta b]$, the vertical velocity at the bottom and to $\Vs[\epsilon \zeta, \beta b]$ the horizontal velocity at the bottom. These operators appear naturally when we differentiate the Dirichlet-Neumann and the Neumann-Neumann operator with respect to the bottom $b$

\begin{equation}\label{ws}
\ws[\epsilon \zeta, \beta b](\psi,B)  = \frac{\mu B + \beta \mu \nabla b \cdot \nabla \left( \GDD(\psi) + \mu \GND(B) \right)}{1+\beta^{2} \mu |\nabla b|^{2}},     
\end{equation}

\hspace{-0.8cm} and

\begin{equation}\label{Vs}
\Vs[\epsilon \zeta, \beta b](\psi,B) = \nabla \left(\GDD(\psi) + \mu \GND(B) \right) - \beta \ws[\epsilon \zeta, \beta b](\psi,B) \nabla b.     
\end{equation}

\begin{prop}\label{controls_ws}
Let $t_{0} > \frac{d}{2}$, $0 \leq s \leq t_{0} + \frac{1}{2}$ and $\zeta,b \in \Hzetaa$ such that Condition \eqref{nonvanishing} is satisfied. Then, 
$\ws[\epsilon \zeta, \beta b]$ maps continuously $\Hdot^{s+\frac{1}{2}}(\RD) \times H^{s-\frac{1}{2}}(\RD)$ into $H^{s-\frac{1}{2}}(\RD)$ 

\begin{equation*}
|\ws[\epsilon \zeta, \beta b](\psi,B)|_{H^{s-\frac{1}{2}}} \leq M \left(|\nabla \psi|_{H^{s-\frac{1}{2}}} + \mu |B|_{H^{s-\frac{1}{2}}} \right).
\end{equation*}

\noindent Finally, we have the same continuity result for $\Vs[\epsilon \zeta,\beta b]$.
\end{prop}

\subsection{Shape derivatives}\label{shape_derivatives_estimates}

\noindent Let $t_{0} > \frac{d}{2}$. Given $B \in H^{\frac{1}{2}}(\RD)$. We denote by $\Gamma$ the set of functions $(\zeta,b)$ in $\Hzetaa$ satisfying \eqref{nonvanishing}. We introduce the map

\begin{equation}
  G^{N \! N}_{\mu}(B) := \left\{
  \begin{aligned}
   & \quad \Gamma  \;\; \rightarrow \; H^{\frac{1}{2}}(\RD) \\
   &(\zeta,b) \mapsto \GNN(B),     
  \end{aligned}
   \right.
\end{equation}

\noindent which is the Neumann-Neumann operator. We can also define $G_{\mu}(\psi)$, $\w(\psi,B)$ and $\V(\psi,B)$.

\begin{remark}\label{G_notation}
\noindent When no confusion is possible and to the sake of simplicity, we write $G_{\mu}(\psi)$, $G^{N \! N}_{\mu}(B)$, $\w(\psi,B)$ and $\V(\psi,B)$ instead of $\G(\psi)$, $\GNN(B)$,
\newline
$\w[\epsilon \zeta, \beta b](\psi,B)$ and $\V[\epsilon \zeta, \beta b](\psi,B)$.
\end{remark}

\noindent In order to linearize the water waves equations, we need a shape derivative formula for the Dirichlet-Neumann and the Neumann-Neumann operators. The following proposition is a summarize of Theorems 3.5 and 3.6 in \cite{Iguchi_tsunami} and Theorem 3.21 in \cite{Lannes_ww}.

\begin{prop}\label{differential_formula}
Let $t_{0} > \frac{d}{2}$, $\zeta,b \in \Hzetaa$, $\psi \in \Hdot^{\frac{3}{2}}(\RD)$ and $B \in H^{\frac{1}{2}}(\RD)$. Then, $G_{\mu}(\psi)$ and $G^{NN}_{\mu}(B)$ are Fr\'echet differentiable. For $(h,k) \in H^{t_{0}+1}(\RD)$, we have

\begin{align*}
dG_{\mu}(\psi).(h,0) + \mu dG^{N\! N}_{\mu}(B).(h,0) = &- \epsilon G_{\mu}[\epsilon \zeta, \beta b](h \, \w[\epsilon \zeta, \beta b](\psi,B) )\\
&- \epsilon \mu \nabla \cdot(h \, \V[\epsilon \zeta, \beta b](\psi,B)),
\end{align*}

\noindent and

\begin{equation*}
dG_{\mu}(\psi).(0,k) + \mu dG^{N \! N}_{\mu}(B).(0,k) =  \beta \mu \GNN \left(\nabla \cdot \left( k \, \Vs[\epsilon \zeta, \beta b](\psi, B) \right) \right).
\end{equation*}

\noindent Furthermore,

\begin{equation*}
dG^{D \! D}_{\mu}(\psi).(h,0) + \mu dG^{N \! D}_{\mu}(B).(h,0) = - \epsilon G^{D \! D}_{\mu}[\epsilon \zeta, \beta b](h \, \w[\epsilon \zeta, \beta b](\psi,B)).
\end{equation*}

\end{prop}

\noindent Thanks to these formulae we can give some controls to the first shape derivatives of the operators. For instance, we give an estimate for $d\ws$ and $d\Vs$.

\begin{prop}\label{controls_dVs}
Let $t_{0} > \frac{d}{2}$ and $(\zeta,b) \in \Hzetaa$ such that Condition \eqref{nonvanishing} is satisfied. Then, for $0 \leq s \leq t_{0} + \frac{1}{2}$, for $\psi \in \Hdot^{s+\frac{1}{2}}(\RD)$ and $B \in H^{s-\frac{1}{2}}(\RD)$, we have

\begin{equation*}
\left\lvert d \Vs (\psi,B).(h,k) \right\rvert_{H^{s-\frac{1}{2}}} \text{, } \hspace{-0.05cm} \left\lvert d \ws (\psi,B).(h,k) \right\rvert_{H^{s-\frac{1}{2}}} \hspace{-0.05cm} \leq M \left\lvert \left( h,k \right) \right\rvert_{H^{t_{0}+1}} \hspace{-0.05cm} \left( |\nabla \psi|_{H^{s-\frac{1}{2}}} \hspace{-0.05cm} + \hspace{-0.05cm} |B|_{H^{s-\frac{1}{2}}} \right) \hspace{-0.05cm}.
\end{equation*}
\end{prop}

\begin{proof}
\noindent This result follows from Proposition \ref{differential_formula} and Proposition \ref{controls_GNN}.
\end{proof}

\noindent We end this part by giving some controls of the shape derivatives of $G_{\mu}$ and $G_{\mu}^{N \! N}$. We do not use the previous method, we differentiate $j$ times directly the dual formulation of both operators. We refer to Proposition 3.28 in \cite{Lannes_ww} for a control of $d^{j} G_{\mu} .(\textbf{h},\textbf{k})(\psi)$.

\begin{prop}\label{controls_dGNN}
Let $t_{0} > \frac{d}{2}$ and $(\zeta,b) \in H^{t_{0}+1}(\RD)$ such that Condition \eqref{nonvanishing} is satisfied. Then for all $0 \leq s \leq t_{0} + \frac{1}{2}$ and $B \in H^{s-\frac{1}{2}}(\RD)$, we have

\begin{equation*}
\left\lvert d^{j} G^{N \! N}_{\mu} .(\textbf{h},\textbf{k})(B) \right\lvert_{H^{s-\frac{1}{2}}} \leq M \underset{i \geq 1}{\prod} \left\lvert (\epsilon h_{i},\beta k_{i}) \right\lvert_{H^{t_{0}+1}}  \left\lvert B \right\lvert_{H^{s-\frac{1}{2}}}.
\end{equation*}

\noindent Furthermore, if $0 \leq s \leq t_{0}$ and $B \in H^{t_{0}}(\RD)$,

\begin{equation*}
\left\lvert d^{j} G^{N \! N}_{\mu} .(\textbf{h},\textbf{k})(B) \right\lvert_{H^{s-\frac{1}{2}}} \leq M \left\lvert \left(\epsilon h_{1}, \beta k_{1} \right) \right\vert_{H^{s+\frac{1}{2}}} \underset{i \geq 2}{\prod} \left\lvert \left(\epsilon h_{i},\beta k_{i} \right) \right\lvert_{H^{t_{0}+1}}  \left\lvert B \right\lvert_{H^{t_{0}}}.
\end{equation*}
\end{prop}

\noindent We do not prove this Proposition here (which is based on a shape derivative of $B^{\mathfrak{d}}$). We refer to \cite{my_thesis}. 

\section{Useful Estimates}\label{estimates}

\noindent In this part, we give some useful estimates, product and commutator estimates. We refer to Appendix B in \cite{Lannes_ww}, \cite{Lannes_sharp_estimates}  and Chapter II in \cite{alinhac_gerard} for the proofs. The first estimates are useful to control $\mathfrak{P} f$. We recall that $\mathfrak{P} =\frac{|D|}{\sqrt{1+ \sqrt{\mu} |D|}}$.

\begin{prop}\label{P_estimates}
Let $f \in H^{1}(\RD)$ and $g \in H^{\frac{1}{2}}(\RD)$. Then,

\begin{equation*}
|\mathfrak{P} g|_{L^{2}} \leq \mu^{-\frac{1}{4}} |g|_{H^{\frac{1}{2}}} \text{, } |\mathfrak{P} f|_{H^{\frac{1}{2}}} \leq \max (1,\mu^{-\frac{1}{4}}) |\nabla f|_{L^{2}} \text{ and } |\nabla f|_{L^{2}} \leq \max (1,\mu^{\frac{1}{4}}) |\mathfrak{P} f|_{H^{\frac{1}{2}}}.
\end{equation*}
\end{prop}

\begin{proof}
\noindent The first inequality follows from the fact that $1+\sqrt{\mu} |\xi| \geq \sqrt{\mu} |\xi|$, the second inequality from $\frac{(1+|\xi|^{2})^{\frac{1}{4}}}{\sqrt{1+\sqrt{\mu} |\xi|}} \leq \max (1,\frac{1}{\mu^{\frac{1}{4}}} )$ and the third from $\frac{\sqrt{1+\sqrt{\mu}|\xi|}}{\sqrt{1+|\xi|}} \leq \max (1,\mu^{\frac{1}{4}})$.
\end{proof}

\noindent We need some product estimates in $\RD$. The following Proposition is Proposition 2.1.2 in \cite{alinhac_gerard}.

\begin{prop}\label{product_estimate1}
Let $s,s_{1},s_{2} \in \R$ such that $s \leq s_{1}$, $s \leq s_{2}$, $s_{1}+s_{2}\geq 0$ and $s < s_{1} + s_{2} - \frac{d}{2}$. Then, there exists a constant $C > 0$ such that for all $f \in H^{s_{1}}(\RD)$ and for all $g \in H^{s_{2}}(\RD)$, we have $fg \in H^{s}(\RD)$ and 

\begin{equation*}
|fg|_{H^{s}} \leq C |f|_{H^{s_{1}}} |g|_{H^{s_{2}}}.
\end{equation*}
\end{prop}

\noindent We also need some product estimates in $S := \RD \times (-1,0)$. The following Proposition is the Corollary B.5 in \cite{Lannes_ww}.

\begin{prop}\label{product_estimate2}
Let $s,s_{1},s_{2} \in \R$ such that $s \leq s_{1}$, $s \leq s_{2}$, $s_{1}+s_{2}\geq 0$, $s < s_{1} + s_{2} - \frac{d}{2}$ and $p \in \lbrace 2,+\infty \rbrace$. Then, there exists a constant $C > 0$ such that for all $f \in L^{\infty}_{z} H^{s_{1}}_{\! X}(S)$ and for all $g \in L^{p}_{z} H^{s_{2}}_{\! X}(S)$, we have $fg \in L^{p}_{z} H^{s}_{\! X}(S)$ and 

\begin{equation*}
|\Lambda^{s} \left(fg \right)|_{L^{p}_{z} L^{2}_{\! X} (S)} \leq C |\Lambda^{s_{1}} f|_{L^{\infty}_{z} L^{2}_{\! X} (S)} |\Lambda^{s_{2}} g|_{L^{p}_{z} L^{2}_{\! X} (S)}.
\end{equation*}
\end{prop}

\noindent The following propositions gives estimates for $1/(1+g)$ in the flat strip $S$. We refer to Corollary B.6 in \cite{Lannes_ww}.

\begin{prop}\label{inverse_estimate1}
Let $T_{0} > \frac{d}{2}$, $-T_{0} \leq s \leq T_{0}$, $k_{0} > 0$ and $p \in \lbrace 2,+\infty \rbrace$. Then, for all $f \in L^{p}_{z} H^{s}_{X}(S)$ and $g \in L^{\infty}_{z}  H^{T_{0}}_{X}(S)$ with $1+g \geq k_{0}$, we have

\begin{equation*}
\left\lvert \Lambda^{s} \frac{f}{1+g} \right\lvert_{L^{p}_{z} L^{2}(S)} \leq C \left(\frac{1}{k_{0}},|g|_{L^{\infty}_{z} H^{T_{0}}_{\! X}} \right) |\Lambda^{s} f|_{L^{p}_{z} L^{2}(S)}.
\end{equation*}

\end{prop}

\begin{prop}\label{inverse_estimate2}
Let $T_{0} > \frac{d}{2}$, $s \geq - T_{0}$ and $k_{0} > 0$. Then, for all $f \text{,} g \in L^{\infty}_{z} H^{T_{0}}_{\! X}(S) \cap {H^{s,0}}(S)$ with $1+g \geq k_{0}$, we have 

\begin{equation*}
\left\lvert \frac{f}{1+g} \right\lvert_{H^{s,0}} \leq C \left(\frac{1}{k_{0}},|g|_{L^{\infty}_{z} H^{T_{0}}_{\! X}} \right) \left( |f|_{H^{s,0}} + \indicatrice_{\{s > T_{0} \}}  |f|_{L^{\infty}_{\! z}  H^{T_{0}}_{\! X}} |g|_{H^{s,0}} \right).
\end{equation*}
\noindent Notice that if $s \leq T_{0}$, $f \in {H^{s,0}}(S)$ is enough.

\end{prop}

\noindent We need some commutator estimates in $S$. The following Propositions are Corollary B.17 in \cite{Lannes_ww}.

\begin{prop}\label{commutator_estimate1}
Let $T_{0} > \frac{d}{2}$, $\delta \geq 0$, $0 < t_{1} \leq 1$ with $t_{1} < T_{0} - \frac{d}{2}$ and $-\frac{d}{2} < s \leq T_{0} + t_{1}$. Then for all $u \in L^{\infty}_{z} H^{T_{0}}_{\! X}$ and $v \in H^{s-t_{1},0}(S)$ we have

\begin{equation*}
\left\lvert \left[ \Lambda^{s}_{\delta},u \right] v \right\lvert_{L^{2}(S)} \leq C \left\lvert \Lambda^{T_{0}}_{\delta} u \right\lvert_{L^{\infty}_{z} L^{2}_{X} (S)} \left\lvert \Lambda^{s-t_{1}}_{\delta} v \right\lvert_{L^{2}(S)}.
\end{equation*}

\end{prop}

\begin{prop}\label{commutator_estimate2}
Let $T_{0} > \frac{d}{2}$, $\delta \geq 0$, $0 < t_{1} \leq 1$ with $t_{1} < T_{0} - \frac{d}{2}$ and $-\frac{d}{2} < s \leq T_{0} + t_{1}$. Then for all $u \in H^{T_{0},0}$ and $v \in L^{\infty}_{z} H^{s-t_{1}}_{\! X}$ we have

\begin{equation*}
\left\lvert \left[ \Lambda^{s}_{\delta},u \right] v \right\lvert_{L^{2}(S)} \leq C \left\lvert \Lambda^{T_{0}}_{\delta} u \right\lvert_{L^{2}(S)} \left\lvert \Lambda^{s-t_{1}}_{\delta} v \right\lvert_{L^{\infty}_{z} L^{2}_{\! X}(S)}.
\end{equation*}

\end{prop}

\section*{Acknowledgments}

\noindent The author has been partially funded by the ANR project Dyficolti ANR-13-BS01-0003.

\newpage
\small{
\bibliographystyle{plain}
\bibliography{biblio}
}

\end{document}